\documentclass[12pt, leqno]{amsart}
\usepackage{amsmath}
\usepackage{amscd}
\usepackage{amssymb}
\usepackage{amsfonts}
\usepackage{latexsym}
\usepackage{amsthm}

\newtheorem{lemma}{{\sc Lemma}}[section]
\newtheorem{corollary}[lemma]{{\sc Corollary}}
\newtheorem{proposition}[lemma]{{\sc Proposition}}
\newtheorem{theorem}[lemma]{{\sc Theorem}}
\newtheorem{remark}[lemma]{{\sc Remark}}

\numberwithin{equation}{section}

\def\Ga{{\mathfrak{a}}}
\def\Gb{{\mathfrak{b}}}
\def\Gg{{\mathfrak{g}}}
\def\Gh{{\mathfrak{h}}}
\def\Gk{{\mathfrak{k}}}

\def\Gm{{\mathfrak{m}}}
\def\Gn{{\mathfrak{n}}}
\def\Gp{{\mathfrak{p}}}
\def\Gs{{\mathfrak{s}}}

\def\BA{{\mathbb{A}}}
\def\BB{{\mathbb{B}}}
\def\BC{{\mathbb{C}}}
\def\BF{{\mathbb{F}}}
\def\BK{{\mathbb{K}}}

\def\BQ{{\mathbb{Q}}}
\def\BZ{{\mathbb{Z}}}
\def\CA{{\mathcal A}}
\def\CB{{\mathcal B}}
\def\CC{{\mathcal C}}
\def\DD{{\mathcal D}}
\def\CO{{\mathcal O}}

\def\CL{{\mathcal L}}
\def\CM{{\mathcal M}}

\def\CR{{\mathcal R}}

\def\CV{{\mathcal V}}
\def\CW{{\mathcal W}}
\def\CX{{\mathcal X}}
\def\CZ{{\mathcal Z}}
\def\Ad{\mathop{\rm Ad}\nolimits}

\def\Aut{{\rm Aut}}

\def\Coker{\mathop{\rm Coker}\nolimits}
\def\Comod{\mathop{\rm Comod}\nolimits}
\def\deru{\partial}

\def\End{\mathop{\rm{End}}\nolimits}

\def\Hom{\mathop{\rm Hom}\nolimits}
\def\Ht{\mathop{\rm ht}\nolimits}
\def\id{\mathop{\rm id}\nolimits}
\def\Id{\mathop{\rm Id}\nolimits}
\def\Ind{\mathop{\rm Ind}\nolimits}
\def\Image{\mathop{\rm Im}\nolimits}
\def\Ker{\mathop{\rm Ker\hskip.5pt}\nolimits}

\def\Mod{\mathop{\rm Mod}\nolimits}

\def\ord{{\mathop{\rm ord}\nolimits}}
\def\Proj{\mathop{\rm Proj}\nolimits}
\def\rad{\mathop{\rm rad}\nolimits}

\def\Spec{{\rm{Spec}}}

\def\Tor{{\rm{Tor}}}
\def\CEnd{\mathop{{\mathcal{E}}nd}\nolimits}
\def\wt{\mathop{\rm wt}\nolimits}
\begin{document}
\title[quantized flag manifolds]
{Differential operators on quantized flag manifolds at roots of unity}
\dedicatory{Dedicated to Masaki Kashiwara with admiration}
\author{Toshiyuki TANISAKI}
\thanks
{
\indent
The author was partially supported by the Grants-in-Aid for Scientific Research, Challenging Exploratory Research No.\ 21654005
from Japan Society for the Promotion of Science.
}
\address{
Department of Mathematics, Osaka City University, 3-3-138, Sugimoto, Sumiyoshi-ku, Osaka, 558-8585 Japan}
\email{tanisaki@sci.osaka-cu.ac.jp}

\begin{abstract}
The quantized flag manifold, which is a $q$-analogue of the ordinary flag manifold, 
is realized as a non-commutative scheme, and 
we can define the category of $D$-modules on it using the framework of non-commutative algebraic geometry; however,
when the parameter $q$ is a root of unity,  Lusztig's Frobenius morphism
allows us to handle $D$-modules on the quantized flag manifold through modules over a certain sheaf of rings on the ordinary flag manifold.
In this paper we will show that this sheaf of rings on the ordinary flag manifold is an Azumaya algebra over its center.
We also show that its restriction to certain subsets are split Azumaya algebras.
These are analogues of some results of Bezrukavnikov-Mirkovi\'{c}-Rumynin on $D$-modules on flag manifolds in positive characteristics.
\end{abstract}
\maketitle
\setcounter{section}{-1}
\section{Introduction}
\label{sec:Intoro}
\subsection{}
In \cite{LR} Lunts and Rosenberg constructed the quantized flag manifold for a quantized enveloping algebra as a non-commutative projective scheme.
They also defined a category of $D$-modules on it, and conjectured a Beilinson-Bernstein type equivalence of categories.
In \cite{T2} we proposed a modification of the definition of the ring of differential operators on the quantized flag manifold, and established a Beilinson-Bernstein type equivalence for the modified ring of differential operators
(see also Backelin-Kremnizer \cite{BK1}).

The above mentioned results are for a quantized enveloping algebra when the parameter $q$ is transcendental.
The aim of this paper is to investigate the ring of differential operators on the quantized flag manifold when the parameter is a root of unity.
It is a general phenomenon that quantized objects at roots of unity resemble ordinary objects in positive characteristics.
Hence it is natural to pursue analogue of the theory of  $D$-modules on flag manifolds in positive characteristics due to Bezrukavnikov-Mirkovi\'{c}-Rumynin \cite{BMR}.
In \cite{BMR} an analogue of the Beilinson-Bernstein equivalence was established on the level of derived categories.
Moreover, it was also shown there that the ring of differential operators satisfies certain Azumaya properties.
In this paper we will be concerned with the Azumaya properties in the quantized situation.

\subsection{}
Let $G$ be a connected simply-connected simple algebraic group over $\BC$, and let $\Gg$ be its Lie algebra.
We fix Borel subgroups $B^+$ and $B^-$ of $G$ such that $H=B^+\cap B^-$ is a maximal torus of $G$.
We denote by $N^\pm$ the unipotent radical of $B^\pm$.
We denote by $Q$ and $\Lambda$ the root lattice and the weight lattice respectively.
We also denote by $\Lambda^+$ the set of dominant weights.
Set $\BF=\BQ(q^{1/|\Lambda/Q|})$, where $q^{1/|\Lambda/Q|}$ is an indeterminate.
We denote by $U_\BF$ the quantized enveloping algebra of $\Gg$ over $\BF$.
It is a Hopf algebra over $\BF$, and is generated as an $\BF$-algebra by the elements 
$k_\lambda, e_i, f_i\,\,(\lambda\in\Lambda, i\in I)$, where $I$ is the index set for simple roots for $\Gg$.
We can define a $q$-analogue $C_\BF$ of the coordinate algebra of $G$ as a Hopf algebra dual of $U_\BF$.
More precisely, we define $C_\BF$ to be the subspace of $\Hom_\BF(U_\BF,\BF)$ spanned by the matrix coefficients of type 1 representations of $U_\BF$.
Then we have a $U_\BF$-bimodule structure of $C_\BF$ given by
\[
\langle u_1\cdot \varphi\cdot u_2,u\rangle
=\langle \varphi,u_2uu_1\rangle
\qquad
(u, u_1, u_2\in U_\BF, \varphi\in C_\BF).
\]
Set 
\[
A_\BF=\bigoplus_{\lambda\in\Lambda^+}A_\BF(\lambda)\subset C_\BF
\]
with
\[
A_\BF(\lambda)
=\{\varphi\in C_\BF\mid
\varphi\cdot v=\chi_\lambda(v)\varphi\quad(v\in U_\BF^{\leqq0})\},
\] 
where $U^{\leqq0}_\BF$ is the subalgebra of $U_\BF$ generated by $k_\lambda, f_i\,\,(\lambda\in\Lambda, i\in I)$ and $\chi_\lambda$ is the character of $U^{\leqq0}_\BF$ corresponding to $\lambda$.
Note that $A_\BF$ is a non-commutative $\Lambda$-graded $\BF$-algebra.
The quantized flag manifold $\CB_q$ is defined as a non-commutative projective scheme by 
\[
\CB_q=\Proj_\Lambda(A_\BF).
\]
This actually means that we are given an abelian category $\Mod(\CO_{\CB_q})$ of ``quasi-coherent sheaves on $\CB_q$'' defined by
\[
\Mod(\CO_{\CB_q})=\Mod_\Lambda(A_\BF)/\Tor_{\Lambda^+}(A_\BF),
\]
where $\Mod_\Lambda(A_\BF)$ is the category of $\Lambda$-graded left $A_\BF$-modules, and $\Tor_{\Lambda^+}(A_\BF)$ denotes its full subcategory consisting of $M\in\Mod_\Lambda(A_\BF)$ such that for each $m\in M$ there exists some $\lambda\in\Lambda^+$ such that $A_\BF(\lambda+\mu)m=0$ for any $\mu\in\Lambda^+$.

Define an $\BF$-subalgebra $D_\BF$ of $\End_\BF(A_\BF)$ by
\[
D_\BF=
\langle \ell_\varphi, r_\varphi,\deru_u, \sigma_\lambda\mid\varphi\in A_\BF, u\in U_\BF, \lambda\in\Lambda\rangle,
\]
where $\ell_\varphi$ and $r_\varphi$ are the left and the right multiplications of $\varphi$ respectively, $\deru_u$ denotes the natural left action of $u\in U_\BF$ on $A_\BF$ induced by that on $C_\BF$, and $\sigma_\lambda$ is the grading operator given by
$\sigma_\lambda(\varphi)=q^{(\lambda,\mu)}\varphi$ for $\varphi\in A_\BF(\mu)$ (see \eqref{eq:killing} below for the notation).
Then $D_\BF$ is a $\Lambda$-graded ring  by 
$D_\BF(\lambda)=\{\Phi\in D_\BF\mid\Phi(A_\BF(\mu))\subset A_\BF(\mu+\lambda)\,\,(\mu\in\Lambda)\}$ for $\lambda\in\Lambda$.
By the aid of the universal $R$-matrix we can show that  $r_\varphi$ can be expressed using other type of generators.
Hence we have 
\[
D_\BF=
\langle \ell_\varphi, \deru_u, \sigma_\lambda\mid\varphi\in A_\BF, u\in U_\BF, \lambda\in\Lambda\rangle.
\]
A $D_\BF$-module is regarded as an $A_\BF$-module by the algebra homomorphism $A_\BF\ni\varphi\mapsto\ell_\varphi\in D_\BF$ in the following.
We define an abelian category $\Mod(\DD_{\CB_q})$ of ``quasi-coherent $\DD_{\CB_q}$-modules'' by
\[
\Mod(\DD_{\CB_q})
=
\Mod_\Lambda(D_\BF)/\Mod_\Lambda(D_\BF)\cap\Tor_{\Lambda^+}(A_\BF).
\]
Denote by $H(\BF)$ the set of $\BF$-rational points of the maximal torus $H$ of $G$.
For $t\in H(\BF)$ we define an abelian category $\Mod(\DD_{\CB_q,t})$ of ``quasi-coherent $\DD_{\CB_q,t}$-modules'' by
\[
\Mod(\DD_{\CB_q,t})
=
\Mod_{\Lambda,t}(D_\BF)/\Mod_{\Lambda,t}(D_\BF)\cap\Tor_{\Lambda^+}(A_\BF),
\]
where $\Mod_{\Lambda,t}(D_\BF)$ is the full subcategory of $\Mod_{\Lambda}(D_\BF)$ consisting of $M\in\Mod_{\Lambda}(D_\BF)$ such that $\sigma_\mu|_{M(\lambda)}=
\mu(t)q^{(\lambda,\mu)}\id$ for any $\lambda\in\Lambda$.
Then $\DD_{\CB_q,1}$ is ``the sheaf of differential operators on $\CB_q$'', and other $\DD_{\CB_q,t}$'s are its twisted analogues (although they have only symbolical meanings). 

Let us consider the specialization of the parameter $q$ to roots of unity.
We take an odd integer $\ell>1$ which is prime to $|\Lambda/Q|$, and prime to 3 if $\Gg$ is of type $G_2$, $F_4$, $E_6$, $E_7$, $E_8$.
We fix a primitive $\ell$-th root of unity $\zeta'\in\BC$ and consider the specialization
\begin{equation}
\label{eq:specialization}
q^{1/|\Lambda/Q|}\mapsto \zeta',\qquad
q\mapsto \zeta=(\zeta')^{|\Lambda/Q|}.
\end{equation}
Note that $\zeta$ is also a primitive $\ell$-th root of unity by our assumption.
Set
\[
\BA=
\{
f(q^{1/|\Lambda/Q|})\in\BF\mid
f(x) \mbox{ is regular at } x=\zeta'
\},
\]
and 
\begin{align*}
U_\BA^L&=\langle
k_\lambda, e_i^{(n)}, f_i^{(n)}\mid
\lambda\in\Lambda, i\in I, n\in\BZ_{\geqq0}
\rangle_{\BA-alg}\subset U_\BF,
\\
U_\BA&=\langle
k_\lambda, e_i, f_i\mid
\lambda\in\Lambda, i\in I
\rangle_{\BA-alg}\subset U_\BF,
\end{align*}
where $e_i^{(n)}, f_i^{(n)}$ denote standard divided powers.
$U_\BA^L$ and $U_\BA$ are Hopf algebras over $\BA$ called the  Lusztig form and the De Concini-Kac form of $U_\BF$ respectively.
Taking the Hopf algebra dual of $U^L_\BA$ we obtain an $\BA$-form $C_\BA$ of $C_\BF$.
We set 
\begin{align*}
A_\BA&=A_\BF\cap C_\BA=
\bigoplus_{\lambda\in\Lambda^+}
A_\BA(\lambda),\\
D_\BA&=
\langle \ell_\varphi, r_\varphi,\deru_u, \sigma_\lambda\mid\varphi\in A_\BA, u\in U_\BA, \lambda\in\Lambda\rangle\\
&=
\langle \ell_\varphi, \deru_u, \sigma_\lambda\mid\varphi\in A_\BA, u\in U_\BA, \lambda\in\Lambda\rangle\subset\End_\BA(A_\BA)
\subset\End_\BF(A_\BF).
\end{align*}
Now we consider the specializations
\begin{align*}
A_\zeta&=\BC\otimes_\BA A_\BA,
\qquad
U_\zeta^L=\BC\otimes_\BA U_\BA^L,\qquad
U_\zeta=\BC\otimes_\BA U_\BA,\\
D_\zeta&=\BC\otimes_\BA D_\BA,
\end{align*}
where $\BA\to\BC$ is given by \eqref{eq:specialization}.
Note that the natural homomorphism $D_\zeta\to\End_\BC(A_\zeta)$ is not injective.
Similarly to $\CB_q$ we obtain a non-commutative projective scheme $\CB_\zeta=\Proj_\Lambda(A_\zeta)$, which actually means we are given an abelian category $\Mod(\CO_{\CB_\zeta})$ defined similarly to $\Mod(\CO_{\CB_q})$.
We also have abelian categories 
$\Mod(\DD_{\CB_\zeta})$, $\Mod(\DD_{\CB_\zeta,t})$ for $t\in H(\BC)$ defined similarly to $\Mod(\DD_{\CB_q})$ and $\Mod(\DD_{\CB_q,t})$ respectively.

Let $U_\zeta^L\to U(\Gg)$ be Lusztig's Frobenius morphism, where $U(\Gg)$ is the enveloping algebra of $\Gg$.
By taking the dual Hopf algebras we obtain a central embedding 
$\BC[G]\to C_\zeta$ of the coordinate algebra $\BC[G]$ of $G$ into $C_\zeta$.
Let $A_1$ be the subalgebra of $\BC[G]$ defined similarly to $A_\BF$.
Then $A_1$ is a commutative $\Lambda$-graded $\BC$-algebra such that $\Proj_\Lambda(A_1)$ is naturally isomorphic to the flag manifold $\CB=B^-\backslash G$ of $G$.
Under the identification $\BC[G]\subset C_\zeta$ we have $A_1\subset A_\zeta$ and $A_1(\lambda)\subset A_\zeta(\ell\lambda)$ for $\lambda\in\Lambda^+$.
We denote by $Fr_*\CO_{\CB_\zeta}$ the $\CO_\CB$-module corresponding to the $\Lambda$-graded $A_1$-module
$A_\zeta^{(\ell)}=\bigoplus_{\lambda\in\Lambda^+}A_\zeta(\ell\lambda)$.
The $A_1$-algebra structure of $A_\zeta^{(\ell)}$ endows $Fr_*\CO_{\CB_\zeta}$ with a canonical $\CO_\CB$-algebra structure.
Then we have an equivalence 
\[
\Mod(\CO_{\CB_\zeta})\cong\Mod(Fr_*\CO_{\CB_\zeta})
\]
of abelian categories, where $\Mod(Fr_*\CO_{\CB_\zeta})$ denotes the category of quasi-coherent $Fr_*\CO_{\CB_\zeta}$-modules.
Similarly, we have
\begin{align*}
&\Mod(\DD_{\CB_\zeta})\cong\Mod(Fr_*\DD_{\CB_\zeta}),\\&\Mod(\DD_{\CB_\zeta, t})\cong\Mod(Fr_*\DD_{\CB_\zeta,t}) \quad(t\in H(\BC)),
\end{align*}
where $Fr_*\DD_{\CB_\zeta}$ and $Fr_*\DD_{\CB_\zeta,t}$ are $\CO_\CB$-algebras corresponding to $A_1$-algebras
$D_\zeta^{(\ell)}$ and $D_\zeta^{(\ell)}\otimes_{\BC[\Lambda]}\BC$ respectively.
Here $D_\zeta^{(\ell)}=\bigoplus_{\lambda\in\Lambda^+}D_\zeta(\ell\lambda)$, and $\BC[\Lambda]=\bigoplus_{\lambda\in\Lambda}\BC e(\lambda)$ denotes the group algebra of $\Lambda$.
Moreover, $\BC[\Lambda]\to D_\zeta^{(\ell)}$ and $\BC[\Lambda]\to\BC$ are given by 
$e(\lambda)\mapsto\sigma_\lambda$ and $e(\lambda)\mapsto \lambda(t)$ respectively.

Denote by $ZD_\zeta^{(\ell)}$ the central subalgebra of $D_\zeta^{(\ell)}$ generated by elements $\ell_\varphi, \deru_u, \sigma_\lambda\,\,(\varphi\in A_1, u\in Z_{Fr}(U_\zeta), \lambda\in\Lambda)$,  where $Z_{Fr}(U_\zeta)$ denotes the Frobenius center of $U_\zeta$. 
Let $\CZ_\zeta$ be the central $\CO_\CB$-subalgebra of $Fr_*\DD_{\CB_\zeta}$ corresponding to $ZD_\zeta^{(\ell)}$.
By \cite{DP} $Z_{Fr}(U_\zeta)$ is a Hopf subalgebra of $U_\zeta$ isomorphic to the coordinate algebra $\BC[K]$ of the algebraic group
\[
K=
\{(hg_+,h^{-1}g_-)\mid
h\in H, g_\pm\in N^\pm\}\subset B^+\times B^-.
\]
We note also that the group algebra $\BC[\Lambda]$ is naturally isomorphic to the coordinate algebra $\BC[H]$ of $H$.
Hence we  have a natural surjective algebra homomorphism $A_1\otimes \BC[K]\otimes\BC[H]\to ZD_\zeta^{(\ell)}$.
Correspondingly, we have a natural surjective $\CO_\CB$-algebra homomorphism 
$p_{*}\CO_{\CB\times K\times H}\to\CZ_\zeta$, where $p:\CB\times K\times H\to\CB$ is the projection.
Define $\kappa:K\to G$ by 
$\kappa(k_1,k_2)=k_1k_2^{-1}$.
\begin{theorem}
\label{theorem01}
We have
$\CZ_\zeta\cong p_{*}\CO_\CV$, where
\[
\CV=\{(B^-g,k,t)\in\CB\times K\times H\mid
g\kappa(k)g^{-1}\in t^{2\ell}N^-\}.
\]
\end{theorem}
The proof of this theorem is accomplished as follows.
In order to show that the kernel of $p_{*}\CO_{\CB\times K\times H}\to\CZ_\zeta$
contains defining equations of $\CV$ one needs to establish certain relations among elements of $ZD_\zeta^{(\ell)}\left(\subset D_\zeta^{(\ell)}\right)$.
As mentioned earlier, $r_\varphi$ for $\varphi\in A_\zeta$ can be expressed using other generators by the aid of the universal $R$-matrix $\CR$.
In fact we have two universal $R$-matrices $\CR$ and ${}^t\CR^{-1}$ by which we obtain two different expressions of the same element $r_\varphi$.
This gives our desired relations.
Hence $\CZ_\zeta$ is a quotient of $p_{*}\CO_\CV$.
To show that $\CZ_\zeta$ is isomorphic to $p_{*}\CO_\CV$ we use Poisson geometry.
We have a natural Poisson structure of $Y=(N^-\backslash G)\times K\times H$, and the support of the pullback of $\CZ_\zeta$ to $Y$ is a Poisson subvariety of $Y$ by \cite{DP}.
On the other hand we can show that the pullback of $\CV$ to $Y$ is a connected symplectic leaf of the Poisson manifold $Y$.
Hence the assertion follows from the fact that $\CZ_\zeta\ne0$ which is easy to check.

We denote by $\tilde{\DD}$ the localization of $Fr_*\DD_{\CB_\zeta}$ on $\CV$.
\begin{theorem}
\label{theorem02}
$\tilde{\DD}$ is locally free over $\CO_\CV$ of finite rank.
Moreover, for any $v\in \CV$ the fiber $\tilde{\DD}(v)$ of $\tilde{\DD}$ at $v$ is isomorphic to the matrix algebra $M_{\ell^{N}}(\BC)$, where $N$ is the number of the positive roots.
In particular, $\tilde{\DD}$ is an Azumaya algebra.
\end{theorem}
This result follows from a result of 
Brown-Gordon \cite{BG} and the fact that $\tilde{\DD}$ is locally generated by $\ell^{2N}$ sections.
Using the action of the braid group on $\tilde{\DD}$ the proof of the latter fact is reduced to a calculation on a standard open subset of $\CB$.

Let $W$ denote the Weyl group.
Consider the fiber product $K\times _{H/W}H$, 
where $K\to H/W$ is the composite of $\kappa:K\to G$ and the map $G\to H/W$ associating $g\in G$ with  its semisimple part, and $H\to H/W$ is given by associating $t\in H$ with the $W$-orbit of $t^{2\ell}$.
We define $\delta:\CV\to K\times _{H/W}H$ by
$\delta(B^-g,k,t)=(k,t)$.
\begin{theorem}
\label{theorem03}
For any $(k,t)\in K\times _{H/W}H$ there exists a locally free $\CO_{\delta^{-1}(k,t)}$-module $\CM$ such that $\tilde{\DD}|_{\delta^{-1}(k,t)}\cong\CEnd_{\CO_{\delta^{-1}(k,t)}}(\CM)$.
Hence $\tilde{\DD}|_{\delta^{-1}(k,t)}$ is a split Azumaya algebra.
\end{theorem}

The proof of Theorem \ref{theorem03} is similar to that for the corresponding fact in positive characteristics due to Bezrukavnikov-Mirkovi\'{c}-Rumynin \cite{BMR}.
By Brown-Gordon \cite{BG} the result is already known when $t\in H$ belongs to certain open dense subset $H_{ur}$ of $H$.
The proof for the general case is reduced to this special case by using certain isomorphisms of Azumaya algebras.

Let us give a comment on the Beilinson-Bernstein type derived equivalence in our context.
In \cite{T4} we have formulated a precise conjecture concerning it, and have shown that it follows from a statement about the derived global section of the ring of differential operators.
Unfortunately the conjecture itself is still open.
We hope to deal with the problem in the near future.

The content of this paper is as follows.
In Section 1 and Section 2 we  recall basic facts on a quantized enveloping algebra and its representations respectively.
In Section 3 the quantized flag manifold is introduced and some of its properties are investigated.
In Section 4 we define the ring of differential operators on the quantized flag manifold and establish some properties. 
In particular, we show that it acquires an action of the braid group.
Theorem \ref{theorem01} is proved in Section 5.
Theorem \ref{theorem02} and Theorem \ref{theorem03} are proved in Section 6.

We note that a closely related result is given in Backelin-Kremnizer \cite{BK2}.

\subsection{}
In this paper we shall use the following notation for a Hopf algebra $H$ over a field $\BK$.
The comultiplication, the counit, and the antipode of $H$ are denoted by
\begin{align}
&\Delta_H:H\to H\otimes_\BK H,\\
&\varepsilon_H:H\to\BK,\\
&S_H:H\to H
\end{align}
respectively.
The subscript $H$ will often be omitted.
For $n\in\BZ_{>0}$ we denote by
\[
\Delta_n:H\to H^{\otimes n+1}
\]
the algebra homomorphism given by 
\[
\Delta_1=\Delta,\qquad
\Delta_n=(\Delta\otimes \id_{H^{\otimes n-1}})\circ\Delta_{n-1},
\]
and write
\[
\Delta(h)=\sum_{(h)}h_{(0)}\otimes h_{(1)},
\qquad
\Delta_n(h)=\sum_{(h)_n}h_{(0)}\otimes\cdots\otimes h_{(n)}
\quad(n\geqq2).
\]

Moreover, for a $\BK$-algebra $A$ we denote by 
\[
m:A\otimes A\to A
\]
the $\BK$-linear map induced by the multiplication of $A$. 

\subsection{}
The author would like to thank the referee for the very careful reading and the  detailed comments.

\section{Quantized enveloping algebras}
\label{sec:QE}
\subsection{}
Let $G$ be a connected simply-connected simple algebraic group over the complex number field $\BC$.
We fix Borel subgroups $B^+$ and $B^-$ such that $H=B^+\cap B^-$ is a maximal torus of $G$.
Set $N^+=[B^+,B^+]$ and $N^-=[B^-,B^-]$.
We denote the Lie algebras of $G$, $B^+$, $B^-$, $H$, $N^+$, $N^-$ by $\Gg$, $\Gb^+$, $\Gb^-$, $\Gh$, $\Gn^+$, $\Gn^-$ respectively.
Let $\Delta\subset\Gh^*$ be the root system of $(\Gg,\Gh)$.
For $\alpha\in\Delta$ we denote by $\Gg_\alpha$ the corresponding root space.
We denote by $\Lambda\subset\Gh^*$ and $Q\subset\Gh^*$ the weight lattice and the root lattice respectively.
For $\lambda\in\Lambda$ we denote the corresponding character of $H$ by
$\theta_\lambda:H\to\BC^\times$.
We take a system of positive roots $\Delta^+$ such that $\Gb^+$ is the sum of weight spaces with weights in $\Delta^+\cup\{0\}$.
Let $\{\alpha_i\}_{i\in I}$ be the set of simple roots, and 
$\{\varpi_i\}_{i\in I}$ the corresponding set of fundamental weights.
We denote by 
$\Lambda^+$ the set of dominant integral weights.
We set $Q^+=\bigoplus_{i\in I}\BZ_{\geqq0}\alpha_i$.
Let $W\subset GL(\Gh^*)$ be the Weyl group.
For $i\in I$ we denote by $s_i\in W$ the corresponding simple reflection.
We take a $W$-invariant symmetric bilinear form
\begin{equation}
\label{eq:killing}
(\,,\,):\Gh^*\times\Gh^*\to\BC
\end{equation}
such that $(\alpha,\alpha)=2$ for short roots $\alpha$.
\begin{lemma}
We have
\[
(\Lambda,Q)\subset\BZ,\qquad
(\Lambda,\Lambda)\subset\frac1{|\Lambda/Q|}\BZ.
\]
\end{lemma}
\begin{proof}
For $\alpha\in\Delta$ and $\lambda\in\Lambda$ we have
\[
(\lambda,\alpha)=\frac{2(\lambda,\alpha)}{(\alpha,\alpha)}
\frac{(\alpha,\alpha)}2\in\BZ.
\]
The second formula follows from the first one.
\end{proof}
For $\alpha\in\Delta$ we set $\alpha^\vee=2\alpha/(\alpha,\alpha)$.
For $i\in I$ we fix $\overline{e}_i\in\Gg_{\alpha_i}$, $\overline{f}_i\in\Gg_{-\alpha_i}$ such that $[\overline{e}_i,\overline{f}_i]=\alpha_i^\vee$ under the identification $\Gh=\Gh^*$ induced by $(\,\,,\,\,)$.

For $\lambda=\sum_{i\in I}c_i\alpha_i\in\Gh^*$ we set $\Ht(\lambda)=\sum_{i\in I}c_i$.

We define $\rho\in\Lambda$ by $(\rho,\alpha_i^\vee)=1$ for any $i\in I$.

\subsection{}
For $n\in\BZ_{\geqq0}$ we set
\[
[n]_t=\frac{t^n-t^{-n}}{t-t^{-1}}\in\BZ[t,t^{-1}],
\qquad
[n]_t!=[n]_t[n-1]_t\cdots[2]_t[1]_t\in\BZ[t,t^{-1}].
\]

We denote by
$U_\BF$
the quantized enveloping algebra over $\BF=\BQ(q^{1/{|\Lambda/Q|}})$ associated to $\Gg$.
Namely, $U_\BF$ is the associative algebra over $\BF$ generated by elements 
\[
k_\lambda\quad(\lambda\in\Lambda),\qquad
e_i, f_i\quad( i\in I)
\]
satisfying the relations
\begin{align}
&k_0=1,\quad 
k_\lambda k_\mu=k_{\lambda+\mu}
\qquad(\lambda,\mu\in\Lambda),
\label{eq:def1}\\
&k_\lambda e_ik_\lambda^{-1}=q^{(\lambda,\alpha_i)}e_i,\qquad(\lambda\in\Lambda, i\in I),
\label{eq:def2a}\\
&k_\lambda f_ik_\lambda^{-1}=q^{-(\lambda,\alpha_i)}f_i
\qquad(\lambda\in\Lambda, i\in I),
\label{eq:def2b}\\
&e_if_j-f_je_i=\delta_{ij}\frac{k_i-k_i^{-1}}{q_i-q_i^{-1}}
\qquad(i, j\in I),
\label{eq:def3}\\
&\sum_{n=0}^{1-a_{ij}}(-1)^ne_i^{(1-a_{ij}-n)}e_je_i^{(n)}=0
\qquad(i,j\in I,\,i\ne j),
\label{eq:def4}\\
&\sum_{n=0}^{1-a_{ij}}(-1)^nf_i^{(1-a_{ij}-n)}f_jf_i^{(n)}=0
\qquad(i,j\in I,\,i\ne j),
\label{eq:def5}
\end{align}
where $q_i=q^{(\alpha_i,\alpha_i)/2}, k_i=k_{\alpha_i}, a_{ij}=2(\alpha_i,\alpha_j)/(\alpha_i,\alpha_i)$ for $i, j\in I$, and
\[
e_i^{(n)}=
e_i^n/[n]_{q_i}!,
\qquad
f_i^{(n)}=
f_i^n/[n]_{q_i}!
\]
for $i\in I$ and $n\in\BZ_{\geqq0}$.
We will use the  Hopf algebra structure of $U_\BF$ given by
\begin{align}
&\Delta(k_\lambda)=k_\lambda\otimes k_\lambda,\\
&\Delta(e_i)=e_i\otimes 1+k_i\otimes e_i,\quad
\Delta(f_i)=f_i\otimes k_i^{-1}+1\otimes f_i,
\nonumber\\
&\varepsilon(k_\lambda)=1,\quad
\varepsilon(e_i)=\varepsilon(f_i)=0,\\
&S(k_\lambda)=k_\lambda^{-1},\quad
S(e_i)=-k_i^{-1}e_i, \quad S(f_i)=-f_ik_i.
\end{align}

Define subalgebras $U_\BF^{0}$, $U_\BF^{+}$, $U_\BF^{-}$, $U_\BF^{\geqq0}$, $U_\BF^{\leqq0}$ of $U_\BF$ by
\begin{align*}
&U_\BF^{0}=\langle k_\lambda\mid\lambda\in\Lambda\rangle,\qquad
U_\BF^+=\langle e_i\mid i\in I\rangle, \qquad
U_\BF^-=\langle f_i\mid i\in I\rangle, \\
&U_\BF^{\geqq0}=\langle k_\lambda, e_i\mid\lambda\in\Lambda, i\in I\rangle
,\qquad
U_\BF^{\leqq0}=\langle k_\lambda, f_i\mid\lambda\in\Lambda, i\in I\rangle.
\end{align*}
Then the multiplication of $U_\BF$ induces isomorphisms
\begin{align}
\label{eq:tri:F1}
&U_\BF\cong U_\BF^{-}\otimes U_\BF^{0}\otimes U_\BF^{+},\\
\label{eq:tri:F2}
&
U_\BF^{\geqq0}\cong U_\BF^{0}\otimes U_\BF^{+}
\cong U_\BF^{+}\otimes U_\BF^{0},\qquad
U_\BF^{\leqq0}\cong U_\BF^{0}\otimes U_\BF^{-}
\cong U_\BF^{-}\otimes U_\BF^{0}
\end{align}
of vector spaces.
Moreover, $\{k_\lambda\mid\lambda\in\Lambda\}$ is an $\BF$-basis of $U_\BF^0$.
We have 
\begin{align}
\label{eq:Upm-decomp1}
&U_\BF^{\pm}=\bigoplus_{\gamma\in Q^+}U_{\BF,\pm\gamma}^{\pm},\quad\\
\label{eq:Upm-decomp2}
&U_{\BF,\pm\gamma}^{\pm}=
\{u\in U_{\BF}^{\pm}\mid
k_\lambda u k_\lambda^{-1}=q^{\pm(\lambda,\gamma)}u\,\,(\lambda\in\Lambda)\}.
\end{align}

We denote by
$\BB$
the braid group corresponding to $W$.
Namely, $\BB$ is a group generated by elements $T_i\,\,(i\in I)$ satisfying relations 
\[
\underbrace{T_iT_j\cdots\cdots}_{\mbox{\footnotesize{$\ord(s_is_j)$-times}}}=
\underbrace{T_jT_i\cdots\cdots}_{\mbox{\footnotesize{$\ord(s_is_j)$-times}}}
\qquad(i, j\in I, i\ne j),
\]
where $\ord(s_is_j)$ denotes the order of $s_is_j\in W$.
For $w\in W$ we set $T_w=T_{i_1}\cdots T_{i_r}$ where $w=s_{i_1}\cdots s_{i_r}$ is a reduced expression of $w$.
It does not depend on the choice of a reduced expression.
We have a group homomorphism
\[
\BB\to
\Aut_{alg}(U_{\BF})
\]
given by
\begin{align*}
&T_i(k_\mu)=k_{s_i\mu}\qquad(\mu\in\Lambda),\\
&T_i(e_j)=
\begin{cases}
\sum_{k=0}^{-a_{ij}}(-1)^kq_i^{-k}e_i^{(-a_{ij}-k)}e_je_i^{(k)}\qquad
&(j\in I,\,\,j\ne i),\\
-f_ik_i\qquad
&(j=i),
\end{cases}\\
&T_i(f_j)=
\begin{cases}
\sum_{k=0}^{-a_{ij}}(-1)^kq_i^{k}f_i^{(k)}f_jf_i^{(-a_{ij}-k)}\qquad
&(j\in I,\,\,j\ne i),\\
-k_i^{-1}e_i\qquad
&(j=i)
\end{cases}
\end{align*}
(see Lusztig \cite{Lbook}).

Let $w_0$ be the longest element of $W$.
We fix a reduced expression 
\[
w_0=s_{i_1}\cdots s_{i_N}
\]
of $w_0$, where $N=|\Delta^+|$, and set
\[
\beta_k=s_{i_1}\cdots s_{i_{k-1}}(\alpha_{i_k})\qquad
(1\leqq k\leqq N).
\]
Then we have $\Delta^+=\{\beta_k\mid1\leqq k\leqq N\}$.
For $1\leqq k\leqq N$ set 
\begin{equation}
\label{eq:root vector}
e_{\beta_k}=T_{i_1}\cdots T_{i_{k-1}}(e_{i_k}),\quad
f_{\beta_k}=T_{i_1}\cdots T_{i_{k-1}}(f_{i_k}).
\end{equation}
Then $\{e_{\beta_{N}}^{m_N}\cdots e_{\beta_{1}}^{m_1}\mid
m_1,\dots, m_N\geqq0\}$ (resp. 
\newline
$\{f_{\beta_{N}}^{m_N}\cdots f_{\beta_{1}}^{m_1}\mid
m_1,\dots, m_N\geqq0\}$)
is an $\BF$-basis of $U_\BF^+$ (resp. $U_\BF^-$), called the PBW-basis (see Lusztig \cite{L2}).
We have $e_{\alpha_i}=e_i$ and $f_{\alpha_i}=f_i$ for any $i\in I$.
For $1\leqq k\leqq N,\, m\geqq0$ we also set 
\begin{equation}
e^{(m)}_{\beta_k}=e^{m}_{\beta_k}/[m]_{q_{\beta_k}}!,\quad
f^{(m)}_{\beta_k}=f^{m}_{\beta_k}/[m]_{q_{\beta_k}}!,
\end{equation}
where $q_\beta=q^{(\beta,\beta)/2}$ for $\beta\in\Delta^+$.

Denote by
\begin{equation}
\label{eq:Drinfeld-paring}
\tau: U_\BF^{\geqq0}\times U_\BF^{\leqq0}\to\BF
\end{equation}
the Drinfeld pairing.
It is characterized as a bilinear form satisfying
\begin{align}
&\tau(x,y_1y_2)=(\tau\otimes\tau)(\Delta(x),y_1\otimes y_2)
&(x\in U_\BF^{\geqq0},\,y_1,y_2\in U_\BF^{\leqq0}),\\
&\tau(x_1x_2,y)=(\tau\otimes\tau)(x_2\otimes x_1,\Delta(y))
&(x_1, x_2\in U_\BF^{\geqq0},\,y\in U_\BF^{\leqq0}),\\
&\tau(k_\lambda,k_\mu)=q^{-(\lambda,\mu)}
&(\lambda,\mu\in\Lambda),\\
&\tau(k_\lambda, f_i)=\tau(e_i,k_\lambda)=0
&(\lambda\in\Lambda,\,i\in I),\\
&\tau(e_i,f_j)=\delta_{ij}/(q_i^{-1}-q_i)
&(i,j\in I)
\end{align}
(see \cite{T1}, \cite{Lbook}).
It satisfies the following (see \cite{T1},  \cite{Lbook}).
\begin{lemma}
\label{lem:Drinfeld paring}
\begin{itemize}
\item[\rm(i)]
$\tau(S(x),S(y))=\tau(x,y)$ for $x\in U_\BF^{\geqq0}, y\in U_\BF^{\leqq0}$.
\item[\rm(ii)]
For $x\in U_\BF^{\geqq0}, y\in U_\BF^{\leqq0}$ we have
\begin{align*}
yx=\sum_{(x)_2,(y)_2}
\tau(x_{(0)},S(y_{(0)}))\tau(x_{(2)},y_{(2)})x_{(1)}y_{(1)},\\
xy=\sum_{(x)_2,(y)_2}
\tau(x_{(0)},y_{(0)})\tau(x_{(2)},S(y_{(2)}))y_{(1)}x_{(1)}.
\end{align*}
\item[\rm(iii)]
$\tau(xk_\lambda, yk_\mu)=q^{-(\lambda,\mu)}\tau(x,y)$ for $\lambda, \mu\in\Lambda, x\in U_\BF^+, y\in U_\BF^-$.
\item[\rm(iv)]
$\tau(U^+_{\BF,\beta}, U^-_{\BF,-\gamma})=\{0\}$ for $\beta, \gamma\in Q^+$ with $\beta\ne\gamma$.
\item[\rm(v)]
For any $\beta\in Q^+$ the restriction $\tau|_{U^+_{\BF,\beta}\times U^-_{\BF,-\beta}}$ is non-degenerate.
\end{itemize}
\end{lemma}

Denote by $Z(U_{\BF})$ the center of $U_{\BF}$.
Let
\[
\BF[\Lambda]=\bigoplus_{\lambda\in\Lambda}\BF
e(\lambda)
\]
be the group algebra of $\Lambda$.
Define a linear map
\[
\iota:Z(U_\BF)\to\BF[\Lambda]
\]
as the composite of 
\[
Z(U_\BF)
\subset
U_\BF
\simeq
U^-_\BF\otimes U^0_\BF\otimes U^+_\BF
\xrightarrow{\varepsilon\otimes 1\otimes\varepsilon}
U_\BF^0
\cong
\BF[\Lambda],
\]
where $U_\BF^0\cong \BF[\Lambda]$ is given by $k_\lambda\leftrightarrow e(\lambda)$ for $\lambda\in\Lambda$.
Then $\iota$ is an injective algebra homomorphism and its image is described as follows.
Note that the Weyl group $W$ naturally acts on $\BF[\Lambda]$ by
\[
we(\lambda)=e(w\lambda)\qquad
(w\in W,\,\lambda\in\Lambda).
\]
We also consider a twisted action of $W$ on $\BF[\Lambda]$ given by
\[
w\circ e(\lambda)=q^{(w\lambda-\lambda,\rho)}e(w\lambda)\qquad
(w\in W,\,\lambda\in\Lambda).
\]
Then the image of $\iota$ coincides with 
\[
\BF[2\Lambda]^{W\circ}
=\{f\in\BF[2\Lambda]\mid 
w\circ f=f\quad(w\in W)\}
\]
(note that the twisted action of $W$ on $\BF[\Lambda]$ preserves $\BF[2\Lambda]$).
In particular, we have an isomorphism
\begin{equation}
\label{eq:HC-center-F}
Z(U_\BF)\simeq\BF[2\Lambda]^{W\circ}
\end{equation}
of $\BF$-algebras (see, e.g.\  \cite{T1}).
For $\lambda\in\Lambda^+$ we denote by $m(\lambda)$ the element of $Z(U_\BF)$ which corresponds to 
\[
\sum_{w\in W/W_\lambda}w\circ e(-2\lambda)\in \BF[2\Lambda]^{W\circ}
\]
under the identification \eqref{eq:HC-center-F}, where
$W_\lambda=\{w\in W\mid w\lambda=\lambda\}$.
Then we have 
\begin{equation}
\label{eq:mlambda:F}
Z(U_\BF)=\bigoplus_{\lambda\in\Lambda^+}\BF m(\lambda).
\end{equation}
\subsection{}
We fix an integer $\ell>1$ satisfying
\begin{itemize}
\item[(a)]
$\ell$ is odd,
\item[(b)]
$\ell$ is prime to 3 if $G$ is of type $G_2$, $F_4$, $E_6$, $E_7$, $E_8$,
\item[(c)]
$\ell$ is prime to $|\Lambda/Q|$,
\end{itemize}
and a primitive $\ell$-th root $\zeta'\in\BC$ of 1.
Define a subring $\BA$ of $\BF$ by
\[
\BA=\{f(q^{1/|\Lambda/Q|})\mid
f(x)\in\BQ(x),\,\mbox{$f$ is regular at $x=\zeta'$}\}.
\]
We set $\zeta=(\zeta')^{|\Lambda/Q|}$.
We note that $\zeta$ is also a primitive $\ell$-th root of 1 by the condition (c).

We denote by $U_\BA^L$, $U_\BA$ the $\BA$-forms of $U_\BF$ called the Lusztig form and  the De Concini-Kac form respectively.
Namely, we have
\begin{align*}
U_\BA^L
&=\langle
e_i^{(m)},\, f_i^{(m)},\,k_\lambda\mid
i\in I,\,m\in\BZ_{\geqq0},\,\lambda\in\Lambda
\rangle_{\BA-{\rm alg}}
\subset U_\BF
,\\
U_\BA
&=\langle
e_i,\, f_i,\,k_\lambda\mid
i\in I,\,\lambda\in\Lambda
\rangle_{\BA-{\rm alg}}
\subset U_\BF.
\end{align*}
We have obviously
$U_\BA\subset U^L_\BA$.
The Hopf algebra structure of $U_\BF$ induces Hopf algebra structures over $\BA$ of $U_\BA^{L}$ and $U_\BA$.
Setting
\[
U_\BA^{L,\flat}=U_\BA^{L}\cap U_\BF^{\flat},\qquad
U_\BA^{\flat}=U_\BA\cap U_\BF^{\flat}
\qquad(\flat=+, -, \geqq0, \leqq0),
\]
we have
\begin{align}
\label{eq:tri:L1}
&
U_\BA^{L}\cong U_\BA^{L,-}\otimes_\BA U_\BA^{L,0}\otimes_\BA U_\BA^{L,+},\\
\label{eq:tri:L2}
&
U_\BA^{L,\geqq0}\cong U_\BA^{L,0}\otimes_\BA U_\BA^{L,+}
\cong U_\BA^{L,+}\otimes_\BA U_\BA^{L,0},\\
\label{eq:tri:L3}
&
U_\BA^{L,\leqq0}\cong U_\BA^{L,0}\otimes_\BA U_\BA^{L,-}
\cong U_\BA^{L,-}\otimes_\BA U_\BA^{L,0},
\end{align}
and 
\begin{align}
\label{eq:tri:DK1}
&
U_\BA\cong U_\BA^{-}\otimes_\BA U_\BA^{0}\otimes_\BA U_\BA^{+},\\
\label{eq:tri:DK2}
&
U_\BA^{\geqq0}\cong U_\BA^{0}\otimes_\BA U_\BA^{+}
\cong U_\BA^{+}\otimes_\BA U_\BA^{0},\\
\label{eq:tri:DK3}
&
U_\BA^{\leqq0}\cong U_\BA^{0}\otimes_\BA U_\BA^{-}
\cong U_\BA^{-}\otimes_\BA U_\BA^{0}.
\end{align}
Set
\begin{align*}
\begin{bmatrix}{k_i;c}\\{m}\end{bmatrix}
&=\prod_{s=0}^{m-1}
\frac{q_i^{c-s}k_i-q_i^{-c+s}k_i^{-1}}
{q_i^{s+1}-q_i^{-s-1}}
\qquad(i\in I, m\in\BZ_{\geqq0}, c\in\BZ),\\
\begin{bmatrix}{k_i}\\{m}\end{bmatrix}
&=\begin{bmatrix}{k_i;0}\\{m}\end{bmatrix}
\qquad(i\in I, m\in\BZ_{\geqq0}).
\end{align*}
Then we have
\[
\begin{bmatrix}{k_i;c}\\{m}\end{bmatrix}
\in U_\BA^{L,0}
\qquad(i\in I, m\in\BZ_{\geqq0}, c\in\BZ)
\]
and
\begin{align*}
U_\BA^{L,0}
&=\bigoplus_{\lambda\in \Lambda', 
(\varepsilon_i)\in\{0,1\}^I, 
(n_i)\in\BZ_{\geqq0}^I}
\BA
k_\lambda\prod_{i\in I}
k_i^{\varepsilon_i}
\begin{bmatrix}{k_i}\\{n_i}
\end{bmatrix}
,\\
U_\BA^{0}
&
=\bigoplus_{\lambda\in\Lambda}\BA k_\lambda,
\end{align*}
where 
$\Lambda'\subset \Lambda$ is a representative of $\Lambda/Q$.
By Lusztig \cite{L2} we have the following.
\begin{lemma}
\label{lem:PBW-L}
\begin{itemize}
\item[\rm(i)]
$\{{e}_{\beta_{N}}^{(m_N)}\cdots {e}_{\beta_{1}}^{(m_1)}\mid
m_1,\dots, m_N\geqq0\}$ $($resp. 
\newline
$\{{f}_{\beta_{N}}^{(m_N)}\cdots {f}_{\beta_{1}}^{(m_1)}\mid
m_1,\dots, m_N\geqq0\}$$)$
is an ${\BA}$-basis of $U_{\BA}^{L,+}$ $($resp. $U_{\BA}^{L,-}$$)$.
\item[\rm(ii)]
$U^L_\BA$ is $\BB$-stable.
\end{itemize}
\end{lemma}
By De Concini-Kac \cite{DK} we have also the following.
\begin{lemma}
\label{lem:PBW-DK}
\begin{itemize}
\item[\rm(i)]
$\{{e}_{\beta_{N}}^{m_N}\cdots {e}_{\beta_{1}}^{m_1}\mid
m_1,\dots, m_N\geqq0\}$ $($resp. 
\newline
$\{{f}_{\beta_{N}}^{m_N}\cdots {f}_{\beta_{1}}^{m_1}\mid
m_1,\dots, m_N\geqq0\}$$)$
is an ${\BA}$-basis of $U_{\BA}^+$ $($resp. $U_{\BA}^-$$)$.
\item[\rm(ii)]
$U_\BA$ is $\BB$-stable.
\end{itemize}
\end{lemma}
By  Lemma \ref{lem:PBW-DK} (i), Lemma \ref{lem:PBW-L} (i) and Jantzen \cite[8.28]{Jan} we have 
\begin{lemma}
\label{lem:pm-duality}
\[
U_\BA^+=\{u\in U_\BF^+\mid
\tau(u,U_\BA^{L,-})\subset\BA)\},
\quad
U_\BA^-=\{u\in U_\BF^-\mid
\tau(U_\BA^{L,+},u)\subset\BA)\}.
\]
\end{lemma}

\subsection{}
Now we consider the specialization 
\[
\BA\to\BC\qquad
(q^{1/|\Lambda/Q|}\mapsto\zeta').
\]
Note that $q$ is mapped to $\zeta=(\zeta')^{|\Lambda/Q|}\in\BC$, which is also a primitive $\ell$-th root of 1.

We set
\begin{align*}
&U_\zeta^L=\BC\otimes_\BA U_\BA^L,\qquad
U_\zeta^{L,\flat}=\BC\otimes_\BA U_\BA^{L,\flat}\quad(\flat=+,-,\geqq0,\leqq0),\\
&U_\zeta=\BC\otimes_\BA U_\BA,\qquad
U_\zeta^{\flat}=\BC\otimes_\BA U_\BA^{\flat}\quad(\flat=+,-,\geqq0,\leqq0).
\end{align*}
Then $U^L_\zeta$ and $U_\zeta$ are Hopf algebras over $\BC$, and we have 
\begin{align*}
&
U_\zeta\cong U_\zeta^{-}\otimes_\BC U_\zeta^{0}\otimes_\BC U_\zeta^{+},\\
&
U_\zeta^{\geqq0}\cong U_\zeta^{0}\otimes_\BC U_\zeta^{+}
\cong U_\zeta^{+}\otimes_\BC U_\zeta^{0},\\
&
U_\zeta^{\leqq0}\cong U_\zeta^{0}\otimes_\BC U_\zeta^{-}
\cong U_\zeta^{-}\otimes_\BC U_\zeta^{0}.
\end{align*}
We denote by 
\[
^L\tau:U^{L,\geqq0}_\zeta\times U^{\leqq0}_\zeta\to\BC,\qquad
\tau^L:U^{\geqq0}_\zeta\times U^{L,\leqq0}_\zeta\to\BC
\]
the bilinear forms induced by the Drinfeld pairing $\tau$.

In general for a Lie algebra $\Gs$ we denote its enveloping algebra by $U(\Gs)$.

We denote by
\begin{equation}
\label{eq:LFrobenius}
\pi:U_\zeta^L\to U(\Gg)
\end{equation}
Lusztig's Frobenius homomorphism (\cite{L2}). 
Namely $\pi$ is the $\BC$-algebra homomorphism given by
\[
\pi(e_i^{(m)})=
\begin{cases}
\overline{e}_i^{(m/\ell)}\,\,&(\ell|m)\\
0&(\ell\not|m),
\end{cases}
\quad
\pi(f_i^{(m)})=
\begin{cases}
\overline{f}_i^{(m/\ell)}\,\,&(\ell|m)\\
0&(\ell\not|m),
\end{cases}
\quad
\pi(k_\lambda)=1
\]
for $i\in I$, $m\in\BZ_{\geqq0}$, $\lambda\in\Lambda$. 
Here, $\overline{e}_i^{(n)}=\overline{e}_i^n/n!$, $\overline{f}_i^{(n)}=\overline{f}_i^n/n!$ for $i\in I$ and $n\in\BZ_{\geqq0}$. 
Then $\pi$ is a homomorphism of Hopf algebras.

Let 
\begin{equation}
\label{eq:j}
j:U_\zeta\to U_\zeta^L
\end{equation}
be the Hopf algebra homomorphism induced by the embedding $U_\BA\subset U^L_\BA$.
Then we can easily check that 
\begin{equation}
\label{eq:pi-j}
(\pi\circ j)(u)=\varepsilon(u)1\qquad(u\in U_\zeta)
\end{equation}
using generators $k_\lambda, e_i, f_i\;(\lambda\in\Lambda, i\in I)$ of $U_\zeta$.

\subsection{}
We recall the description of the center $Z(U_\zeta)$ of the algebra $U_\zeta$ due to De Concini-Kac \cite{DK} and De Concini-Procesi \cite{DP}.

Denote by $Z(U_{\BA})$ the center of $U_{\BA}$.
Then by \cite{DK} we have
\begin{equation*}
Z(U_\BA)=\bigoplus_{\lambda\in\Lambda^+}\BA m(\lambda).
\end{equation*}
Define a subalgebra $Z_{Har}(U_\zeta)$ of $Z(U_\zeta)$ by
\[
Z_{Har}(U_\zeta)=\Image(Z(U_{\BA})\to U_\zeta).
\]
We define a twisted action of $W$ on the group algebra 
$\BC[\Lambda]=\bigoplus_{\lambda\in\Lambda}\BC
e(\lambda)$
of $\Lambda$ by
\[
w\circ e(\lambda)=\zeta^{(w\lambda-\lambda,\rho)}e(w\lambda)\qquad
(w\in W,\,\lambda\in\Lambda).
\]
By \cite{DK} \eqref{eq:HC-center-F} induces an isomorphism
\begin{equation}
\label{eq:HC-center}
Z_{Har}(U_\zeta)\simeq\BC[2\Lambda]^{W\circ}
\end{equation}
of $\BC$-algebras.
Namely, 
the linear map
\[
\iota:Z_{Har}(U_\zeta)\to\BC[\Lambda]
\]
defined as the composite of 
\[
Z_{Har}(U_\zeta)
\subset
U_\zeta
\simeq
U^-_\zeta\otimes U^0_\zeta\otimes U^+_\zeta
\xrightarrow{\varepsilon\otimes 1\otimes\varepsilon}
U_\zeta^0
\cong
\BC[\Lambda]
\]
is an injective algebra homomorphism whose image coincides with
$\BC[2\Lambda]^{W\circ}$.
Moreover, we have
\begin{equation}
\label{eq:mlambda:zeta}
Z_{Har}(U_\zeta)=\bigoplus_{\lambda\in\Lambda^+}\BC m(\lambda),
\end{equation}
where we also denote by $m(\lambda)$ its image in $U_\zeta$ by abuse of notation.

By \cite{DK} the elements
\[
{e_\beta}^\ell,\quad
{f_\beta}^\ell,\quad
k_{\ell\lambda}\qquad(\beta\in\Delta^+,\,\lambda\in\Lambda)
\]
are central  in $U_\zeta$.
Let $Z_{Fr}(U_\zeta)$ be the subalgebra of $U_\zeta$ generated by them.
$Z_{Fr}(U_\zeta)$ turns out to be a $\BB$-stable Hopf subalgebra of $U_\zeta$.
Define an algebraic subgroup $K$ of $B^+\times B^-$ by
\[
K=\{(gh, g'h^{-1})\mid
h\in H,\,g\in N^+,\,g'\in N^-\}.
\]
By \cite{DP} we have an isomorphism
\begin{equation}
\label{eq:Fr-center}
Z_{Fr}(U_\zeta)\cong \BC[K]
\end{equation}
of Hopf algebras.
The following description of the isomorphism \eqref{eq:Fr-center} is due to Gavarini \cite{Gav}.
Let us identify $\BC[K]$ with $\BC[N^+]\otimes\BC[N^-]\otimes\BC[H]$ via the isomorphism 
\[
N^+\times N^-\times H\cong
K\qquad
((g,g',h)\leftrightarrow(gh,g'h^{-1}))
\]
of algebraic varieties.
For $f\in\BC[N^-]$, $f'\in\BC[N^+]$, $\lambda\in\Lambda$ the element of $Z_{Fr}(U_\zeta)$ corresponding to $f'\otimes f\otimes \theta_\lambda$ is given by $uk_{\ell\lambda}(Su')$ where $u\in U_\zeta^+$, $u'\in U_\zeta^-$ are given by 
\begin{align*}
&\tau^L(u,y)=\langle f,\pi(y)\rangle\qquad(y\in U^{L,-}_\zeta),\\
&{}^L\tau(x,u')=\langle f',\pi(x)\rangle\qquad(x\in U^{L,+}_\zeta).
\end{align*}
Here we identify $\BC[N^\pm]$ with a subspace of $U(\Gn^\pm)^*$ via the canonical Hopf pairing.

Define 
\begin{equation}
\label{eq:kappa}
\kappa:K\to G
\end{equation}
by $\kappa(g_1, g_2)=g_1g_2^{-1}$.
Define $\eta:G\to H/W$ as follows.
For $g\in G$ let $g_s\in G$ be the semisimple part of $g$ with respect to the Jordan decomposition.
Then $\Ad(G)(g_s)\cap H$ coincides with a single $W$-orbit in $H$.
We define $\eta(g)\in H/W$ to be this $W$-orbit.
The morphism $\eta\circ\kappa:K\to H/W$ of algebraic varieties induces an injective algebra homomorphism $(\eta\circ\kappa)^*:\BC[H/W]\to\BC[K]$.
We identify $\BC[H/W]$ with 
\[
\BC[2\ell \Lambda]^W
=\{f\in\BC[2\ell\Lambda]\mid 
wf=f\quad(w\in W)\}
\]
using the identification 
\[
\BC[2\ell\Lambda]\cong\BC[H]
\qquad
(e(2\ell\lambda)\leftrightarrow\theta_\lambda).
\]
\begin{proposition}
[De Concini-Procesi \cite{DP}]
There exists an isomorphism
\begin{equation}
\label{eq:HCFr-center}
Z_{Har}(U_\zeta)\cap Z_{Fr}(U_\zeta)
\cong \BC[2\ell\Lambda]^W
\end{equation}
of algebras such that the diagram
\[
\begin{CD}
Z_{Har}(U_\zeta)
@<<<
Z_{Har}(U_\zeta)\cap Z_{Fr}(U_\zeta)
@>>>
Z_{Fr}(U_\zeta)
\\
@VVV @VVV @VVV
\\
\BC[2\Lambda]^{W\circ}
@<<<
\BC[2\ell \Lambda]^W
@>>>
\BC[K]
\end{CD}
\]
commutes.
Here the vertical arrows are the isomorphisms
\eqref{eq:HC-center}, \eqref{eq:HCFr-center}, \eqref{eq:Fr-center}, 
the upper horizontal arrows are the inclusions, and the lower horizontal arrows are the inclusion $\BC[2\ell \Lambda]^W
\subset\BC[2\Lambda]^{W\circ}$ and $(\eta\circ\kappa)^*$.
Moreover, we have an isomorphism 
\[
Z(U_\zeta)\cong
Z_{Har}(U_\zeta)\otimes_{Z_{Har}(U_\zeta)\cap Z_{Fr}(U_\zeta)}
Z_{Fr}(U_\zeta)
\qquad(z_1z_2\leftrightarrow z_1\otimes z_2)
\]
of algebras.
In particular, we have
\begin{equation}
\label{eq:full-center}
Z(U_\zeta)\cong
\BC[2\Lambda]^{W\circ}
\otimes_{\BC[2\ell \Lambda]^W}
\BC[K].
\end{equation}
\end{proposition}
\begin{corollary}
\label{cor:center} 
We have
\[
\Spec\, Z(U_\zeta)
\cong K\times_{H/W}H/{W\circ},
\]
where $H/{W\circ}\to{H/W}$ is given by $[t]\mapsto[t^\ell]$.
\end{corollary}

\section{Representation}
\label{sec:rep}
\subsection{}
If $R$ is a ring, we denote by $\Mod(R)$ the category of left $R$-modules.
\begin{remark}
{\rm
We will also use the notation like $\Mod(\CR)$ even when $\CR$ is not a ring (e.g.\ $\Mod(\CO_{\CB_\zeta})$).
The meaning of this type of notation will be explained separately when they appear.
}
\end{remark}

For $M_1, M_2\in\Mod(U_\BF)$ the tensor product $M_1\otimes M_2$ has a natural $U_\BF$-module structure by
\[
u\cdot(m_1\otimes m_2)=\Delta(u)(m_1\otimes m_2)
\qquad
(u\in U_\BF,\,m_1\in M_1,\, m_2\in M_2).
\]

For $\lambda\in\Lambda$ we define an algebra homomorphism $\chi_\lambda:U_\BF^0\to\BF$ by $\chi_\lambda(k_\mu)=q^{(\lambda,\mu)}\,(\mu\in\Lambda)$.
For $M\in\Mod(U_\BF)$ and $\lambda\in\Lambda$ we set
\[
M_\lambda=\{m\in M\mid
hm=\chi_\lambda(h)m\quad(h\in U_\BF^0)\}.
\]
We denote by $\Mod_f(U_\BF)$ the category of finite dimensional $U_\BF$-modules $M$ such that $M=\bigoplus_{\lambda\in\Lambda}M_\lambda$.
We also denote by $\Mod_{int}(U_\BF)$ the category of $U_\BF$-modules $M$ which is a sum of modules in $\Mod_f(U_\BF)$.
It is well-known that a $U_\BF$-module $M$ belongs to $\Mod_{int}(U_\BF)$  if and only if $M=\bigoplus_{\lambda\in\Lambda}M_\lambda$ and
for any $m\in M$ there exists $r\in\BZ_{>0}$ such that 
$e_i^{(r)}m=f_i^{(r)}m=0$ for any $i\in I$.
For $M_1, M_2\in\Mod_{int}(U_\BF)$ we have $M_1\otimes M_2\in\Mod_{int}(U_\BF)$.

For $\lambda\in\Lambda$ we define
$M_{+,\BF}(\lambda), M_{-,\BF}(\lambda)\in\Mod(U_\BF)$
by
\begin{align*}
M_{+,\BF}(\lambda)
=&U_\BF/
\sum_{y\in U_\BF^-}U_\BF(y-\varepsilon(y))+
\sum_{h\in U_\BF^0}U_\BF(h-\chi_\lambda(h)),\\
M_{-,\BF}(\lambda)
=&U_\BF/
\sum_{x\in U_\BF^+}U_\BF(x-\varepsilon(x))+
\sum_{h\in U_\BF^0}U_\BF(h-\chi_\lambda(h)).
\end{align*}
$M_{+,\BF}(\lambda)$ is a lowest weight module with lowest weight $\lambda$, and $M_{-,\BF}(\lambda)$ is a highest weight module with highest weight $\lambda$.
By \eqref{eq:tri:F1} we have isomorphisms 
\[
M_{+,\BF}(\lambda)
\cong
U_\BF^{+}\quad
(\overline{u}\leftrightarrow u),\qquad
M_{-,\BF}(\lambda)
\cong
U_\BF^{-}\quad
(\overline{u}\leftrightarrow u)
\]
of $\BF$-modules.
Moreover, we have weight space decompositions
\[
M_{+,\BF}(\lambda)
=\bigoplus_{\mu\in\lambda+Q^+}M_{+,\BF}(\lambda)_\mu,\qquad
M_{-,\BF}(\lambda)
=\bigoplus_{\mu\in\lambda-Q^+}M_{-,\BF}(\lambda)_\mu.
\]

For $\lambda\in\Lambda^+$ we define 
$L_{+,\BF}(-\lambda), L_{-,\BF}(\lambda)\in\Mod_f(U_\BF)$
by
\begin{align*}
&L_{+,\BF}(-\lambda)\\
=&U_\BF/
\sum_{y\in U_\BF^-}U_\BF(y-\varepsilon(y))+
\sum_{h\in U_\BF^0}U_\BF(h-\chi_{-\lambda}(h))
+\sum_{i\in I}U_\BF e_i^{((\lambda,\alpha_i^\vee)+1)},\\
&L_{-,\BF}(\lambda)\\
=&U_\BF/
\sum_{x\in U_\BF^+}U_\BF(x-\varepsilon(x))+
\sum_{h\in U_\BF^0}U_\BF(h-\chi_\lambda(h))
+\sum_{i\in I}U_\BF f_i^{((\lambda,\alpha_i^\vee)+1)}.
\end{align*}
$L_{+,\BF}(-\lambda)$ is a finite-dimensional irreducible lowest weight module with lowest weight $-\lambda$, and $L_{-,\BF}(\lambda)$ is a finite-dimensional irreducible highest weight module with highest weight $\lambda$.
We have weight space decompositions
\[
L_{+,\BF}(-\lambda)
=\bigoplus_{\mu\in-\lambda+Q^+}L_{+,\BF}(-\lambda)_\mu,\qquad
L_{-,\BF}(\lambda)
=\bigoplus_{\mu\in\lambda-Q^+}L_{-,\BF}(\lambda)_\mu.
\]
We have also $L_{-,\BF}(\lambda)\cong L_{+,\BF}(w_0\lambda)$.
Moreover, the category $\Mod_f(U_\BF)$ is semisimple, and its simple objects are $L_{-,\BF}(\lambda)$ for $\lambda\in\Lambda^+$ (see Lusztig \cite{L1}).

Let $M$ be a $U_\BF$-module with weight space decomposition
$M=\bigoplus_{\mu\in\Lambda}M_\mu$ such that $\dim M_\mu<\infty$ for any $\mu\in\Lambda$.
We define a  $U_\BF$-module $M^\bigstar$ by
\[
M^\bigstar=\bigoplus_{\mu\in\Lambda}
M_\mu^*\subset M^*=\Hom_\BF(M,\BF),
\]
where the action of $U_\BF$ is given by
\[
\langle um^*,m\rangle=\langle m^*,(Su)m\rangle
\qquad(u\in U_\BF, m^*\in M^\bigstar, m\in M).
\]
Here $\langle\,,\,\rangle:M^\bigstar\times M\to\BF$ is the natural pairing.

We set
\begin{align*}
M^*_{\pm,\BF}(\lambda)&=(M_{\mp,\BF}(-\lambda))^\bigstar\qquad
(\lambda\in\Lambda),\\
L^*_{\pm,\BF}(\mp\lambda)&=(L_{\mp,\BF}(\pm\lambda))^\bigstar\qquad
(\lambda\in\Lambda^+).
\end{align*}
Since $L_{\mp,\BF}(\pm\lambda)$ is irreducible, we have
\[
L^*_{\pm,\BF}(\mp\lambda)\cong L_{\pm,\BF}(\mp\lambda)\qquad
(\lambda\in\Lambda^+).
\]
Note that we have an injective $U_\BF$-homomorphism
\begin{equation}
\label{eq:LASTARtoMASTAR:F}
L^*_{\pm,\BF}(\mp\lambda)\to M^*_{\pm,\BF}(\mp\lambda)
\qquad(\lambda\in\Lambda^+).
\end{equation}
induced by the natural homomorphism
$M_{\mp,\BF}(\pm\lambda)\to L_{\mp,\BF}(\pm\lambda)$.

For $M\in\Mod_{int}(U_\BF)$ we have a group homomorphism
\[
\BB\to \End(M)^\times
\]
given by
\begin{align}
\label{eq:Ti}
T_i
&=
\exp_{q_i^{-1}}(q_ik_if_i)
\exp_{q_i^{-1}}(-e_i)
\exp_{q_i^{-1}}(q_i^{-1}k_i^{-1}f_i)H_i\\
\nonumber
&=
\exp_{q_i^{-1}}(-q_ik_i^{-1}e_i)
\exp_{q_i^{-1}}(f_i)
\exp_{q_i^{-1}}(-q_i^{-1}k_ie_i)H_i,
\end{align}
where
\[
\exp_t(x)=\sum_{n=0}^\infty\frac{t^{n(n-1)/2}}{[n]_t!}x^n
\in\BQ(t)[[x]],
\]
and $H_i$ is the operator on $M$ which acts by $q^{(\lambda,\alpha_i)((\lambda,\alpha_i^\vee)+1)/2}\id$ on $M_\lambda$ for each $\lambda\in\Lambda$.
This operator $T_i$ coincides with Lusztig's operator $T''_{i,1}$ in \cite[5.2]{Lbook}.
We have
\[
T_w(M_\lambda)=M_{w\lambda}\qquad
(M\in\Mod_{int}(U_\BF),\,\,\lambda\in\Lambda),
\]
and
\begin{equation*}
T_w(um)=T_w(u)T_w(m)
\qquad(w\in W, u\in U_\BF, m\in M\in\Mod_{int}(U_\BF)).
\end{equation*}
For $M_1, M_2\in\Mod_{int}(U_\BF)$ 
we sometimes write the action of $T\in\BB$ on $M_1\otimes M_2\in \Mod_{int}(U_\BF)$ by
$\Delta T:M_1\otimes M_2\to M_1\otimes M_2$.
We will need the following (see Lusztig \cite[5.3]{Lbook}).
\begin{lemma}
\label{lem:DeltaT}
For $M_1, M_2\in \Mod_{int}(U_\BF)$ and $i\in I$ we have 
\begin{align*}
\Delta T_i&=
\exp_{q_i}(q_i^{-2}(q_i-q_i^{-1})e_ik_i^{-1}\otimes f_ik_i)(T_i\otimes T_i)\\
&=(T_i\otimes T_i)\exp_{q_i}((q_i-q_i^{-1})f_i\otimes e_i)
\end{align*}
in $\End_\BF(M_1\otimes M_2)^\times$.
\end{lemma}

\subsection{}
For $M\in\Mod(U^L_\BA)$ and $\lambda\in\Lambda$ we set
\[
M_\lambda=\{m\in M\mid
hm=\chi_\lambda(h)m\quad(h\in U^{L,0}_\BA)\}.
\]
\begin{lemma}
\label{lemma:chi-indep0}
Let $\lambda, \mu\in\Lambda$ such that $\lambda\ne\mu$.
Then there exists $h\in U^{L,0}_\BA$ such that
$\chi_\lambda(h)=1$ and $\chi_\mu(h)=0$.
In particular, we have
$\chi_\lambda\ne\chi_\mu$
in $\Hom_\BA(U_\BA^{L,0},\BA)$.
\end{lemma}
\begin{proof}
Take $i\in I$ such that $(\lambda,\alpha_i^\vee)\ne(\mu,\alpha_i^\vee)$.
We may assume $(\lambda,\alpha_i^\vee)>(\mu,\alpha_i^\vee)$.
Then the assertion holds for
$h=\begin{bmatrix}
{k_i;-(\mu,\alpha_i^\vee)}\\
{(\lambda-\mu,\alpha_i^\vee)}
\end{bmatrix}
$.
\end{proof}

For $\lambda\in\Lambda$ we define 
$M_{+,\BA}(\lambda), M_{-,\BA}(\lambda)\in\Mod(U^L_\BA)$
by
\begin{align*}
M_{+,\BA}(\lambda)
=&U^L_\BA/
\sum_{y\in U_\BA^{L,-}}U^L_\BA(y-\varepsilon(y))+
\sum_{h\in U_\BA^{L,0}}U^L_\BA(h-\chi_\lambda(h)),\\
M_{-,\BA}(\lambda)
=&U^L_\BA/
\sum_{x\in U_\BA^{L,+}}U^L_\BA(x-\varepsilon(x))+
\sum_{h\in U_\BA^{L,0}}U^L_\BA(h-\chi_\lambda(h)).
\end{align*}
By \eqref{eq:tri:L1} we have isomorphisms
\[
M_{+,\BA}(\lambda)
\cong
U_\BA^{L,+}\quad
(\overline{u}\leftrightarrow u),\qquad
M_{-,\BA}(\lambda)
\cong
U_\BA^{L,-}\quad
(\overline{u}\leftrightarrow u)
\]
of $\BA$-modules.
In particular, $M_{\pm,\BA}(\lambda)$ is a free $\BA$-module and we have $\BF\otimes_\BA M_{\pm,\BA}(\lambda)\cong M_{\pm,\BF}(\lambda)$.
Moreover, we have weight space decompositions
\[
M_{+,\BA}(\lambda)
=\bigoplus_{\mu\in\lambda+Q^+}M_{+,\BA}(\lambda)_\mu,\qquad
M_{-,\BA}(\lambda)
=\bigoplus_{\mu\in\lambda-Q^+}M_{-,\BA}(\lambda)_\mu.
\]

For $\lambda\in\Lambda^+$ we define 
$L_{+,\BA}(-\lambda)\in\Mod(U^L_\BA)$ (resp. $L_{-,\BA}(\lambda)\in\Mod(U^L_\BA)$)
to be the $U^L_\BA$-submodule of
$L_{+,\BF}(-\lambda)$ (resp. $L_{-,\BF}(\lambda)$) 
generated by $\overline{1}\in L_{+,\BF}(-\lambda)$ (resp. $\overline{1}\in L_{-,\BF}(\lambda)$).
By definition $L_{\pm,\BA}(\mp\lambda)$ is a free $\BA$-module and we have $\BF\otimes_\BA L_{\pm,\BA}(\mp\lambda)\cong L_{\pm,\BF}(\mp\lambda)$.
Moreover, we have weight space decompositions
\[
L_{+,\BA}(-\lambda)
=\bigoplus_{\mu\in-\lambda+Q^+}L_{+,\BA}(-\lambda)_\mu,\qquad
L_{-,\BA}(\lambda)
=\bigoplus_{\mu\in\lambda-Q^+}L_{-,\BA}(\lambda)_\mu.
\]
The canonical surjective $U_\BF$-homomorphism $M_{\pm,\BF}(\mp\lambda)\to L_{\pm,\BF}(\mp\lambda)$ induces a surjective $U_\BA^L$-homomorphism 
\begin{equation}
\label{eq:MAtoLA}
M_{\pm,\BA}(\mp\lambda)\to L_{\pm,\BA}(\mp\lambda)
\qquad(\lambda\in\Lambda^+).
\end{equation}
Note that \eqref{eq:MAtoLA} is a split epimorphism of $\BA$-modules since $\BA$ is PID and $L_{\pm,\BA}(\mp\lambda)_\mu$ is a torsion free finitely generated $\BA$-module for any $\mu\in\mp\lambda\pm Q^+$.

Let $M$ be a $U^L_\BA$-module with weight space decomposition
$M=\bigoplus_{\mu\in\Lambda}M_\mu$ such that $M_\mu$ is a free $\BA$-module of finite rank for any $\mu\in\Lambda$.
We define a  $U^L_\BA$-module $M^\bigstar$ by
\[
M^\bigstar=\bigoplus_{\mu\in\Lambda}
\Hom_\BA(M_\mu,\BA)\subset \Hom_\BA(M,\BA),
\]
where the action of $U^L_\BA$ is given by
\[
\langle um^*,m\rangle=\langle m^*,(Su)m\rangle
\qquad(u\in U^L_\BA, m^*\in M^\bigstar, m\in M).
\]
Here $\langle\,,\,\rangle:M^\bigstar\times M\to\BA$ is the natural pairing.

We set
\begin{align*}
M^*_{\pm,\BA}(\lambda)&=(M_{\mp,\BA}(-\lambda))^\bigstar\qquad
(\lambda\in\Lambda),\\
L^*_{\pm,\BA}(\mp\lambda)&=(L_{\mp,\BA}(\pm\lambda))^\bigstar\qquad
(\lambda\in\Lambda^+).
\end{align*}
Then $M^*_{\pm,\BA}(\lambda)$ for $\lambda\in\Lambda$ and $L^*_{\pm,\BA}(\mp\lambda)$ for $\lambda\in\Lambda^+$ are free $\BA$-modules satisfying 
\[
\BF\otimes_\BA M^*_{\pm,\BA}(\lambda)
\cong
M^*_{\pm,\BF}(\lambda),\qquad
\BF\otimes_\BA L^*_{\pm,\BA}(\mp\lambda)
\cong
L^*_{\pm,\BF}(\mp\lambda).
\]
Moreover, we can identify $M^*_{\pm,\BA}(\lambda)$ and $L^*_{\pm,\BA}(\mp\lambda)$ with $\BA$-submodules of 
$M^*_{\pm,\BF}(\lambda)$ and $L^*_{\pm,\BF}(\mp\lambda)$ respectively.
Under this identification we have 
\begin{equation}
\label{eq:MLAF}
L^*_{\pm,\BA}(\mp\lambda)=
L^*_{\pm,\BF}(\mp\lambda)\cap M^*_{\pm,\BA}(\mp\lambda)\qquad(\lambda\in\Lambda^+).
\end{equation}
In particular, the $U_\BA^L$-homomorphism
\begin{equation}
\label{eq:LASTARtoMASTAR}
L^*_{\pm,\BA}(\mp\lambda)\to M^*_{\pm,\BA}(\mp\lambda)
\qquad(\lambda\in\Lambda^+)
\end{equation}
is a split monomorphism of $\BA$-modules.

\subsection{}
Let $\lambda\in\Lambda$.
By abuse of notation 
we also denote by $\chi_\lambda:U^{L,0}_\zeta\to\BC$ the $\BC$-algebra homomorphism induced by $\chi_\lambda:U^{L,0}_\BA\to\BA$.
\begin{lemma}
\label{lemma:chi-indep}
\begin{itemize}
\item[\rm(i)]
Let $\lambda, \mu\in\Lambda$.
If we have $\chi_\lambda=\chi_\mu$ in $\Hom_\BC(U^{L,0}_\zeta,\BC)$, then we have $\lambda=\mu$.
\item[\rm(ii)]
$\{\chi_\lambda\}_{\lambda\in\Lambda}$ is a linearly independent subset of  $\Hom_\BC(U^{L,0}_\zeta,\BC)$.
\end{itemize}
\end{lemma}
\begin{proof}
(i) is a consequence of Lemma \ref{lemma:chi-indep0}, and 
(ii) follows from (i) easily.
\end{proof}
For $M\in\Mod(U^L_\zeta)$ and $\lambda\in\Lambda$ we set
\[
M_\lambda=\{m\in M\mid hm=\chi_\lambda(h)m\,\,(h\in U^{L,0}_\zeta)\}.
\]
We denote by $\Mod_f(U^L_\zeta)$ the category of finite dimensional $U_\zeta^L$-modules $M$ such that $M=\bigoplus_{\lambda\in\Lambda}M_\lambda$.
We also denote by $\Mod_{int}(U^L_\zeta)$ the category of $U^L_\zeta$-modules $M$ which is a sum of modules in $\Mod_f(U^L_\zeta)$.
It is known that a $U^L_\zeta$-module $M$ belongs to $\Mod_{int}(U^L_\zeta)$  if and only if $M=\bigoplus_{\lambda\in\Lambda}M_\lambda$ and
for any $m\in M$ there exists $r\in\BZ_{>0}$ such that 
$e_i^{(n)}m=f_i^{(n)}m=0$ for any $i\in I$ and $n\geqq r$ (see, for example, Andersen-Polo-Wen \cite{APW1}).

For $M\in\Mod_f(U^L_\zeta)$ we have a group homomorphism
\[
\BB\to GL(M)
\]
given by the formula similar to \eqref{eq:Ti} ($q$ is replaced by $\zeta$).
\begin{lemma}
\label{lem:braid-Fr}
Let $M$ be a finite-dimensional $U(\Gg)$-module.
If we regard $M$ as a $U_\zeta^L$-module via $\pi:U_\zeta^L\to U(\Gg)$ $($see \eqref{eq:LFrobenius}$)$,
then the action of $T_i$ on the $U_\zeta^L$-module $M$ is given by 
\[
T_i=\exp(\overline{f}_i)\exp(-\overline{e}_i)\exp(\overline{f}_i).
\]
\end{lemma}
\begin{proof}
Note that for $\lambda\in\Lambda$ and $m\in M$ satisfying $hm=\lambda(h)m$ for any $h\in\Gh$ we have $tm=\chi_{\ell\lambda}(t)m$ for any $t\in U^{L,0}_\zeta$.
In particular, $k_i$ and $H_i$ in \eqref{eq:Ti} act trivially on $M$. From this we see easily that the assertion holds.
\end{proof}

For $\lambda\in\Lambda$ we set
\[
M_{\pm,\zeta}(\lambda)=\BC\otimes_\BA M_{\pm,\BA}(\lambda),
\qquad
M^*_{\pm,\zeta}(\lambda)=\BC\otimes_\BA M^*_{\pm,\BA}(\lambda).
\]
For $\lambda\in\Lambda^+$ we set
\[
L_{\pm,\zeta}(\mp\lambda)=\BC\otimes_\BA L_{\pm,\BA}(\mp\lambda),
\qquad
L^*_{\pm,\zeta}(\mp\lambda)=\BC\otimes_\BA L^*_{\pm,\BA}(\mp\lambda).
\]
We have canonical $U^L_\zeta$-homomorphisms
\begin{align}
\label{eq:MtoL}
M_{\pm,\zeta}(\mp\lambda)\to L_{\pm,\zeta}(\mp\lambda)
\qquad(\lambda\in\Lambda^+),
\\
\label{eq:LSTARtoMSTAR}
L^*_{\pm,\zeta}(\mp\lambda)\to M^*_{\pm,\zeta}(\mp\lambda)
\qquad(\lambda\in\Lambda^+).
\end{align}
Note that \eqref{eq:MtoL} is surjective, and \eqref{eq:LSTARtoMSTAR} is injective.

\section{Quantized flag manifold}
\label{qflag}
\subsection{}
We denote by $C_\BF$ the subspace of $U_\BF^*=\Hom_\BF(U_\BF,\BF)$ spanned by the matrix coefficients of $U_\BF$-modules belonging to $\Mod_f(U_\BF)$, and denote by
\begin{equation}
\label{eq:Hopf-paring}
\langle\,,\,\rangle:C_\BF\times U_\BF\to\BF
\end{equation}
the canonical pairing.
Then $C_\BF$ is endowed with a Hopf algebra structure dual to $U_\BF$ via \eqref{eq:Hopf-paring}.
Moreover, the bilinear pairing \eqref{eq:Hopf-paring} is a Hopf pairing in the sense that we have
\begin{itemize}
\item[(a)]
$u\in U_\BF,\,\langle C_\BF, u\rangle=0
\,\Longrightarrow
u=0$,
\item[(b)]
$\varphi\in C_\BF,\,\langle \varphi, U_\BF\rangle=0
\,\Longrightarrow
\varphi=0$,
\item[(c)]
$\langle \varphi\psi,u\rangle=\langle \varphi\otimes\psi,\Delta u\rangle
\qquad(\varphi, \psi\in C_\BF, u\in U_\BF)$,
\item[(d)]
$\langle \varphi,uv\rangle=\langle \Delta\varphi,u\otimes v\rangle
\qquad(\varphi\in C_\BF, u, v\in U_\BF)$,
\item[(e)]
$\langle 1,u\rangle=\varepsilon(u)
\qquad(u\in U_\BF)$,
\item[(f)]
$\langle \varphi,1\rangle=\varepsilon(\varphi)
\qquad(\varphi\in C_\BF)$,
\item[(g)]
$\langle S\varphi,Su\rangle=\langle \varphi,u\rangle
\qquad(\varphi\in C_\BF, u\in U_\BF)$
\end{itemize}
(see, for example, \cite[5.11]{Jan} for (a) and  \cite{T1} for (b),\dots, (g)).
Note also that we have a $U_\BF$-bimodule structure of $C_\BF$ by
\[
\langle u_1\cdot\varphi\cdot u_2,u\rangle
=\langle \varphi,u_2uu_1\rangle
\qquad(\varphi\in C_\BF, u, u_1, u_2\in U_\BF).
\]
Set
\begin{align*}
A_\BF&=\{\varphi\in C_\BF\mid\varphi\cdot f_i=0\quad(i\in I)\},\\
A_\BF(\lambda)&=
\{\varphi\in A_\BF\mid \varphi\cdot k_\mu=q^{(\mu,\lambda)}\varphi\quad(\mu\in\Lambda)\}
\qquad(\lambda\in\Lambda^+).
\end{align*}
Then we have 
\[
A_\BF=\bigoplus_{\lambda\in\Lambda^+}A_\BF(\lambda),
\]
and $A_\BF$ turns out to be a $\Lambda$-graded ring.
Note also that $A_\BF$ is a left $U_\BF$-submodule of $C_\BF$.
Moreover, for $\lambda\in\Lambda^+$ we have $A_\BF(\lambda)\cong L_\BF(\lambda)$ as a $U_\BF$-module
(see \cite{T2}).

We set
\[
(U_\BF^\pm)^\bigstar=\bigoplus_{\beta\in Q^+}\Hom_\BF(U_{\BF,\pm\beta}^\pm,\BF)\subset
\Hom_\BF(U_{\BF}^\pm,\BF)=(U_\BF^\pm)^*.
\]
We identify $(U^-_{\BF})^*\otimes (U^0_{\BF})^*\otimes (U^+_{\BF})^*$ with a subspace of $U_\BF^*$ by the embedding
\begin{align*}
&(U^-_{\BF})^*\otimes (U^0_{\BF})^*\otimes (U^+_{\BF})^*\to U_\BF^*\\
&\quad
(f\otimes\chi\otimes g\mapsto[uhv\mapsto
f(u)\chi(h)g(v)] \mbox{ for }u\in 
 U^-_{\BF}, h\in U^0_{\BF}, v\in U^+_{\BF}).
\end{align*}
Then  we have
\begin{align}
\label{eq:Cdecomp}
C_\BF&\subset
(U^-_{\BF})^\bigstar\otimes (\bigoplus_{\lambda\in \Lambda}\BF\chi_\lambda)\otimes (U^+_{\BF})^\bigstar,\\
\label{eq:Adecomp1}
A_\BF&=
(\varepsilon\otimes (\bigoplus_{\lambda\in \Lambda}\BF\chi_\lambda)\otimes (U^+_{\BF})^\bigstar)
\cap\, C_\BF,\\
\label{eq:Adecomp2}
A_\BF(\lambda)&=
(\varepsilon\otimes \chi_\lambda\otimes (U^+_{\BF})^\bigstar)
\cap\, C_\BF.
\end{align}

\subsection{}
We define $\BA$-forms $C_\BA$, $A_\BA$, $A_\BA(\lambda)$ ($\lambda\in\Lambda^+$) of $C_\BF$, $A_\BF$,  $A_\BF(\lambda)$ respectively by
\begin{align*}
&C_\BA=\{\varphi\in C_\BF\mid
\langle \varphi,U_\BA^L\rangle\subset\BA\},\quad
A_\BA=A_\BF\cap C_\BA,\quad
A_\BA(\lambda)=A_\BF(\lambda)\cap C_\BA.
\end{align*}
Then $C_\BA$ is an $\BA$-algebra and $A_\BA$ is its subalgebra.
Moreover, $C_\BA$ is a $U^L_\BA$-bimodule and $A_\BA$ is its left $U^L_\BA$-submodule.

We set
\[
(U_\BA^{L,\pm})^\bigstar=\bigoplus_{\beta\in Q^+}\Hom_\BA(U_{\BA,\pm\beta}^{L,\pm},\BA)\subset
\Hom_\BA(U_{\BA}^{L,\pm},\BA).
\]
By Lemma \ref{lemma:chi-indep0} we can easily show
\[
(\bigoplus_{\lambda\in \Lambda}\BF\chi_\lambda)
\cap \Hom_\BA(U^{L,0}_\BA,\BA)=
\bigoplus_{\lambda\in \Lambda}\BA\chi_\lambda.
\]
Hence by \eqref{eq:Cdecomp}, \eqref{eq:Adecomp1}, \eqref{eq:Adecomp2}
we have
\begin{align}
\label{eq:CdecompA}
C_\BA&=
(
(U^{L,-}_{\BA})^\bigstar\otimes (\bigoplus_{\lambda\in \Lambda}\BA\chi_\lambda)\otimes (U^{L,+}_{\BA})^\bigstar)
\cap C_\BF,\\
\label{eq:Adecomp1A}
A_\BA&=
(\varepsilon\otimes (\bigoplus_{\lambda\in \Lambda}\BA\chi_\lambda)\otimes (U^{L,+}_{\BA})^\bigstar)
\cap\, C_\BF,\\
\label{eq:Adecomp2A}
A_\BA(\lambda)&=
(\varepsilon\otimes \chi_\lambda\otimes (U^{L,+}_{\BA})^\bigstar)
\cap\, C_\BF.
\end{align}
In particular, we have
\begin{equation}
A_\BA=\bigoplus_{\lambda\in\Lambda^+}
A_\BA(\lambda).
\end{equation}
By \eqref{eq:CdecompA} we can easily show that $C_\BA$ is naturally a Hopf algebra over $\BA$.
\begin{lemma}
\label{lem:AA}
We have an isomorphism
\[
A_\BA(\lambda)\cong L^*_{-,\BA}(\lambda)
\]
of $U_\BA^L$-modules.
\end{lemma}
\begin{proof}
Note that we have an isomorphism 
\[
g_\lambda:L^*_{-,\BF}(\lambda)\to A_\BF(\lambda)
\]
of $U_\BF$-modules given by
\[
\langle g_\lambda(\ell^*),u\rangle
=\langle \ell^*, \overline{Su}\rangle
\qquad(\ell^*\in L^*_{-,\BF}(\lambda), u\in U_\BF).
\]
Here, the pairing in the right side is the canonical one 
$L^*_{-,\BF}(\lambda)\times
L_{+,\BF}(-\lambda)\to\BF$.
We have $g_\lambda(\ell^*)\in A_\BA(\lambda)$ if and only if $\langle g_\lambda(\ell^*),U^{L}_\BA\rangle\subset\BA$.
By
$\langle g_\lambda(\ell^*),U^{L}_\BA\rangle
=\langle \ell^*, \overline{U^L_\BA}\rangle
=\langle \ell^*, L_{+,\BA}(-\lambda)\rangle
$
we obtain
$g_\lambda^{-1}(A_\BA(\lambda))=L^*_{-,\BA}(\lambda)$.
\end{proof}
By setting 
\[
A_\BA(\lambda)_\xi
=A_\BF(\lambda)_\xi\cap A_\BA
\qquad(\lambda\in\Lambda^+, \xi\in\Lambda)
\]
we have
\begin{equation}
A_\BA(\lambda)=\bigoplus_{\gamma\in Q^+}
A_\BA(\lambda)_{\lambda-\gamma}.
\end{equation}

\subsection{}
We set 
\[
C_\zeta=\BC\otimes_\BA C_\BA,\quad
A_\zeta=\BC\otimes_\BA A_\BA,\quad
A_\zeta(\lambda)=\BC\otimes_\BA A_\BA(\lambda)\qquad
(\lambda\in\Lambda^+).
\]
Then $C_\zeta$ is a Hopf algebra over $\BC$.
Moreover, the $U_\BF$-bimodule structure of $C_\BF$ induces a $U_\zeta^L$-bimodule structure of
$C_\zeta$.
Let
\begin{equation}
\label{eq:Hopf-paring2}
\langle\,,\,\rangle:C_\zeta\times U^L_\zeta\to\BC.
\end{equation}
be the pairing induced by \eqref{eq:Hopf-paring}.

We set
\[
(U_\zeta^{L,\pm})^\bigstar=\bigoplus_{\beta\in Q^+}(U_{\zeta,\pm\beta}^{L,\pm})^*\subset
(U_{\zeta}^{L,\pm})^*.
\]
By Lemma \ref{lemma:chi-indep} (ii) we have
\[
(U^{L,-}_{\zeta})^\bigstar\otimes (\bigoplus_{\lambda\in \Lambda}\BC\chi_\lambda)\otimes (U^{L,+}_{\zeta})^\bigstar
\subset
(U^{L}_{\zeta})^*.
\]
Moreover, by \eqref{eq:CdecompA}, \eqref{eq:Adecomp1A}, \eqref{eq:Adecomp2A}
we have
\begin{align}
\label{eq:CdecompZ}
C_\zeta&\subset
(U^{L,-}_{\zeta})^\bigstar\otimes (\bigoplus_{\lambda\in \Lambda}\BC\chi_\lambda)\otimes (U^{L,+}_{\zeta})^\bigstar,\\
\label{eq:Adecomp1Z}
A_\zeta&\subset
\varepsilon\otimes (\bigoplus_{\lambda\in \Lambda}\BC\chi_\lambda)\otimes (U^{L,+}_{\zeta})^\bigstar,\\
\label{eq:Adecomp2Z}
A_\zeta(\lambda)&\subset
\varepsilon\otimes \chi_\lambda\otimes (U^{L,+}_{\zeta})^\bigstar.
\end{align}
In particular, we have
\begin{equation}
A_\zeta=\bigoplus_{\lambda\in\Lambda^+}
A_\zeta(\lambda)
\subset C_\zeta\subset(U_\zeta^L)^*.
\end{equation}
Hence we have
\begin{align*}
A_\zeta
&=\{\varphi\in C_\zeta\mid
\varphi\cdot u=\varepsilon(u)\varphi\quad(u\in U_\zeta^{L,-})\},\\
A_\zeta(\lambda)
&=\{\varphi\in A_\zeta\mid
\varphi\cdot h=\chi_\lambda(h)\varphi\quad(h\in U_\zeta^{L,0})\}
\qquad(
\lambda\in\Lambda^+).
\end{align*}
By Andersen-Wen \cite[4.2(2)]{AW} 
(see also Andersen-Polo-Wen \cite[1.31 Theorem(iii)]{APW1}) we have the following.
\begin{proposition}
\label{prop:AW}
$C_\zeta$ coincides with the subspace of $(U^L_\zeta)^*$ spanned by the matrix coefficients of $U^L_\zeta$-modules belonging to $\Mod_f(U^L_\zeta)$.
\end{proposition}
By Lemma \ref{lem:AA} we have the following.
\begin{lemma}
\label{lem:Azeta}
For any $\lambda\in\Lambda^+$ we have an isomorphism
\[
A_\zeta(\lambda)\cong L^*_{-,\zeta}(\lambda)
\]
of $U_\zeta^L$-modules.
\end{lemma}

\begin{lemma}
\label{lem:AAtoA}
For $\lambda,\mu\in\Lambda^+$ the canonical map 
\begin{equation}
\label{eq:AAtoA}
A_\zeta(\lambda)\otimes A_\zeta(\mu)\to A_\zeta(\lambda+\mu)
\end{equation}
induced by the multiplication of $A_\zeta$ is surjective.
\end{lemma}
\begin{proof}
Note that \eqref{eq:AAtoA} is a homomorphism of $U^L_\zeta$-modules.
Hence by Lemma \ref{lem:Azeta} we have only to show that the unique (up to scalar) homomorphism $L_{-,\zeta}(\lambda+\mu)\to
L_{-,\zeta}(\lambda)\otimes
L_{-,\zeta}(\mu)$ of $U^L_\zeta$-modules is injective.
For that it is sufficient to show that any non-trivial homomorphism
$L_{-,\BA}(\lambda+\mu)\to
L_{-,\BA}(\lambda)\otimes
L_{-,\BA}(\mu)$
of 
$U_\BA^L$-modules which maps $L_{-,\BA}(\lambda+\mu)_{\lambda+\mu}$ onto $L_{-,\BA}(\lambda)_\lambda\otimes
L_{-,\BA}(\mu)_\mu$
is a split monomorphism of $\BA$-modules.
This follows from \cite[Chapter 27]{Lbook}.
\end{proof}
\begin{lemma}
\eqref{eq:Hopf-paring2} is a Hopf pairing.
\end{lemma}
\begin{proof}
It is sufficient to show that the canonical map
$U_\zeta^L\to(C_\zeta)^*$ is injective.
By Proposition \ref{prop:AW} we have only to show that if $u\in U^L_\zeta$
satisfies 
$u\mapsto0$ under $U^L_\zeta\to 
\End_\BC(L_{-,\zeta}(\lambda)\otimes L_{+,\zeta}(-\mu))$ for any $\lambda, \mu\in\Lambda^+$, then $u=0$.
This can be proved as in \cite[5.11]{Jan}. 
Details are omitted.
\end{proof}

\subsection{}
Assume that we are given a homomorphism $\iota:A\to B$ of $\Lambda$-graded rings satisfying
\begin{equation}
\label{eq:cond-AK}
\iota(A(\lambda))B(\mu)=B(\mu)\iota(A(\lambda))\qquad
(\lambda, \mu\in\Lambda).
\end{equation}
For $M\in\Mod_{\Lambda}(B)$ let $\Tor(M)$ be the subset of $M$ consisting of $m\in M$ 
such that there exists $\lambda\in\Lambda^+$ satisfying $\iota(A(\lambda+\mu))m=\{0\}$ for any $\mu\in\Lambda^+$.
Then $\Tor(M)$ is a subobject of $M$ in $\Mod_{\Lambda}(B)$ by \eqref{eq:cond-AK}.
We denote by $\Tor_{\Lambda^+}(A,B)$ the full subcategory of $\Mod_{\Lambda}(B)$ consisting of $M\in\Mod_{\Lambda}(B)$ such that $\Tor(M)=M$.
Note that $\Tor_{\Lambda^+}(A,B)$ is closed under taking subquotients and extensions in $\Mod_{\Lambda}(B)$.
Let $\Sigma(A,B)$ denote the collection of morphisms $f$ of $\Mod_{\Lambda}(B)$ such that its kernel $\Ker(f)$ and its cokernel $\Coker(f)$ belong to $\Tor_{\Lambda^+}(A,B)$.
Then we define an abelian category $\CC(A,B)=
\Mod_\Lambda(B)/\Tor_{\Lambda^+}(A,B)$ as
the localization 
\[
\CC(A,B)
=\Sigma(A,B)^{-1}\Mod_{\Lambda}(B)
\]
of $\Mod_{\Lambda}(B)$ with respect to the multiplicative system $\Sigma(A,B)$
(see \cite{GZ}, \cite{Popescu} for the notion of localization of categories).
The abelian category $\CC(A,B)$ has enough injectives (see \cite[Ch4, Corollary 6.2]{Popescu}).
We denote by 
\begin{equation}
\label{eq:omega1}
\omega(A,B)^*:\Mod_{\Lambda}(B)\to
\CC(A,B)
\end{equation}
the canonical exact functor.
It admits a right adjoint 
\begin{equation}
\label{eq:omega2}
\omega(A,B)_*:\CC(A,B)
\to
\Mod_{\Lambda}(B),
\end{equation}
which is left exact (see \cite[Ch4, Proposition 5.2]{Popescu}).
Moreover, we have $\omega(A,B)^*\circ\omega(A,B)_*\cong \Id$
(see \cite[Ch4, Proposition 4.3]{Popescu}).

We apply the above arguments to the case $A=B=A_\zeta$.
By Lemma \ref{lem:AAtoA}
$\Tor(M)$ for $M\in\Mod_{\Lambda}(A_\zeta)$ consists of 
$m\in M$ such that there exists $\lambda\in\Lambda^+$ satisfying $A_\zeta(\lambda)m=\{0\}$.
We set 
\begin{equation}
\Mod(\CO_{\CB_\zeta})=
\CC(A_\zeta,A_\zeta).
\end{equation}
In this case the natural functors \eqref{eq:omega1}, \eqref{eq:omega2}
are simply denoted as
\begin{align}
\label{eq:omega1A}
\omega^*&:\Mod_{\Lambda}(A_\zeta)\to
\Mod(\CO_{\CB_\zeta}),\\
\omega_*&:\Mod(\CO_{\CB_\zeta})
\to
\Mod_{\Lambda}(A_\zeta).
\end{align}
\begin{remark}
{\rm
In the terminology of non-commutative algebraic geometry 
$\Mod(\CO_{\CB_\zeta})$ is the category of ``quasi-coherent sheaves'' on the quantized flag manifold $\CB_\zeta$, which is a ``non-commutative projective scheme''.
The notations ${\CB_\zeta}$, $\CO_{\CB_\zeta}$ have only symbolical meaning.
}
\end{remark}

\subsection{}
Using Lusztig's Frobenius homomorphism \eqref{eq:LFrobenius}  we will relate the quantized flag manifold $\CB_\zeta$ with the ordinary flag manifold $\CB=B^-\backslash G$.
Taking the dual Hopf algebras in \eqref{eq:LFrobenius} we obtain an injective homomorphism
$\BC[G]\to C_\zeta$
of Hopf algebras.
Moreover, its image is contained in the center of $C_\zeta$ (see Lusztig \cite{L2}).
We will regard $\BC[G]$ as a central Hopf subalgebra of $C_\zeta$ in the following.
Setting
\begin{align*}
A_1=&\{\varphi\in\BC[G]\mid
\varphi(ng)=\varphi(g)\,\,(n\in N^-,\,g\in G)\},\\
A_1(\lambda)
=&\{\varphi\in A_1\mid
\varphi(tg)=\theta_\lambda(t)\varphi(g)\,\,(t\in H,\,g\in G)\}
\qquad(\lambda\in\Lambda^+)
\end{align*}
we have a $\Lambda$-graded algebra
\[
A_1=\bigoplus_{\lambda\in\Lambda^+}A_1(\lambda).
\]
We have a left $G$-module structure of $A_1$ given by 
\[
(x\varphi)(g)=\varphi(gx)
\qquad(\varphi\in A_1, x, g\in G).
\]
Note that $A_1$ is a central subalgebra of $A_\zeta$ since $\BC[G]$ is central in $C_\zeta$.
In particular, $A_1$ is a $U(\Gg)$-module.
Moreover, for each $\lambda\in\Lambda^+$, $A_1(\lambda)$ is a $U(\Gg)$-submodule of $A_1$ which is an irreducible highest weight module with highest weight $\lambda$.
For $\lambda\in\Lambda^+$ and  $\xi\in\Lambda$ we set
\[
A_1(\lambda)_\xi
=
\{\varphi\in A_1(\lambda)\mid h\varphi=\xi(h)\varphi\,\,(h\in\Gh)\}.
\]

For $\lambda,\mu\in\Lambda^+$ the canonical map 
\[
A_1(\lambda)\otimes A_1(\mu)\to A_1(\lambda+\mu)
\]
induced by the multiplication of $A_1$ is surjective
since it is a non-trivial homomorphism of $U(\Gg)$-modules into an irreducible module.

Regarding $\BC[G]$ as a subalgebra of $C_\zeta$ we have 
\begin{align}
\label{eq:A1Azeta1}
A_1&=A_\zeta\cap\BC[G],\\
\label{eq:A1Azeta2}
A_1(\lambda)_\xi&=A_\zeta(\ell\lambda)_{\ell\xi}\cap\BC[G]
\quad(\lambda\in\Lambda^+, \xi\in\Lambda).
\end{align}
\begin{proposition}
\label{prop:AzetaoverA1}
$A_\zeta$ is finitely generated as an $A_1$-module.
\end{proposition}
\begin{proof}
For $i\in I$ set
\[
A_{\zeta, i}
=\bigoplus_{n\geqq0}A_\zeta(n\varpi_i)\subset
A_\zeta,\qquad
A_{1, i}
=\bigoplus_{n\geqq0}A_1(n\varpi_i)\subset A_1.
\]
Then the natural maps 
\[
\bigotimes_{i\in I}A_{\zeta, i}\to A_\zeta,
\qquad
\bigotimes_{i\in I}A_{1, i}\to A_1
\]
induced by the multiplications of $A_\zeta$ and $A_1$ respectively are surjective (see Lemma \ref{lem:AAtoA}).
Here the tensor product is defined with respect to some  fixed ordering of $I$.
Since $A_1$ is a central subalgebra of $A_\zeta$, it is sufficient to show that $A_{\zeta,i}$ is a finitely generated $A_{1,i}$-module for any $i\in I$.
Set $A_{\zeta,i}^{(\ell)}=
\bigoplus_{n\geqq0}A_\zeta(\ell n\varpi_i)$.
By Lemma \ref{lem:AAtoA} $A_{\zeta,i}$ is generated by 
$\bigoplus_{n<\ell}A_\zeta(n\varpi_i)$
as an $A_{\zeta,i}^{(\ell)}$-module, and hence
it is sufficient to show that $A_{\zeta,i}^{(\ell)}$ is a finitely generated $A_{1,i}$-module.
Assume we could show
\begin{equation}
\label{eq:formula-i}
A_1(\varpi_i)A_\zeta(\ell m\varpi_i)=A_\zeta(\ell(m+1)\varpi_i)
\end{equation}
for some $m>0$.
Then for any $n\geqq m$ we have
\begin{align*}
&A_1(\varpi_i)A_\zeta(\ell n\varpi_i)
=A_1(\varpi_i)A_\zeta(\ell m\varpi_i)A_\zeta(\ell (n-m)\varpi_i)\\
=&A_\zeta(\ell(m+1)\varpi_i)A_\zeta(\ell (n-m)\varpi_i)
=A_\zeta(\ell(n+1)\varpi_i),
\end{align*}
and hence
\[
A_1(k\varpi_i)A_\zeta(m\ell\varpi_i)
=A_\zeta(\ell(m+k)\varpi_i).
\]
This implies that $A_{\zeta,i}^{(\ell)}$ is generated by
$\bigoplus_{n=0}^mA_\zeta(\ell n\varpi_i)$.
Hence it is sufficient to show \eqref{eq:formula-i}.

Set $\Delta^+_\sharp=\Delta^+\cap \sum_{j\ne i}\BZ\alpha_j$ and 
$\Gp_\sharp=\Gb^-\oplus\bigoplus_{\alpha\in\Delta^+_\sharp}\Gg_\alpha$.
We denote by $P_\sharp$ the parabolic subgroup of $G$ with Lie algebra $\Gp_\sharp$.
We also denote by $U_\zeta^{L,\sharp}$ the subalgebra of $U_\zeta^L$ generated by $U_\zeta^{L,\leqq0}$ and $e_j^{(n)}$ for $j\ne i, n\geqq0$.
We define a Hopf algebra $C_\zeta^\sharp$ as the image of the composite of $C_\zeta\to  \Hom_\BC(U_\zeta^{L},\BC)\to\Hom_\BC(U_\zeta^{L,\sharp},\BC)$.
For a Hopf algebra $H$ we denote by $\Comod(H)$ the category of right $H$-comodules.
We have functors 
\begin{align*}
&r:\Comod(C_\zeta^\sharp)\to\Comod(C_\zeta^{\leqq0}),\\
&r':\Comod(\BC[P_\sharp])\to\Comod(C_\zeta^{\leqq0})
\end{align*}
such that $r(N)=N$ and $r'(N')=N'$ as vector spaces and the left 
$U_\zeta^{L,\leqq0}$-actions on $r(N)$ and $r'(N')$ are given by the algebra homomorphisms 
\[
U_\zeta^{L,\leqq0}
\to U_\zeta^{L,\sharp},\qquad
U_\zeta^{L,\leqq0}
\to U_\zeta^{L,\sharp}\to U(\Gp_\sharp)
\]
respectively.

Let $m>0$.
Denote by $M$ the kernel of the homomorphism
$A_1(\varpi_i)\to\BC_{\varpi_i}$ of $U(\Gp_\sharp)$-modules.
Then we have an exact sequence
\[
0\to r'(M\otimes\BC_{m\varpi_i})
\to r'(A_1(\varpi_i))\otimes\BC_{\ell m\varpi_i}
\to\BC_{\ell(m+1)\varpi_i}\to0
\]
of right $C_\zeta^{\leqq0}$-comodules.
Applying the induction functor 
\[
\Ind:\Comod(C_\zeta^{\leqq0})\to\Comod(C_\zeta)
\]
(see \cite{APW1}) to this exact sequence 
we obtain an exact sequence
\[
A_1(\varpi_i)\otimes A_\zeta(\ell m\varpi_i)
\to
A_\zeta(\ell(m+1)\varpi_i)
\to
R^1\Ind(r'(M\otimes\BC_{m\varpi_i}))
\]
of right $C_\zeta$-comodules.
By 
\[
\dim\Hom_{U^L_\zeta}(A_1(\varpi_i)\otimes A_\zeta(\ell m\varpi_i),
A_\zeta(\ell(m+1)\varpi_i))=1
\]
the map $A_1(\varpi_i)\otimes A_\zeta(\ell m\varpi_i)
\to
A_\zeta(\ell(m+1)\varpi_i)$
in the above exact sequence coincides with the one given by the multiplication in $A_\zeta$ up to a non-zero constant multiple.
Hence it is sufficient to show that for any finite-dimensional right $\BC[P_\sharp]$-comodule $N$ there exists some $m>0$ such that $R^1\Ind(r'(N\otimes\BC_{m\varpi_i}))=0$.
We may assume that $N$ is irreducible.
Hence it is sufficient to show $R^1\Ind(r'(N_1))=0$ for the irreducible $U(\Gp_\sharp)$-module  $N_1$ with highest weight $\lambda\in\Lambda^+$. 
Note that we have natural induction functors
\[
\Ind_1:\Comod(C_\zeta^{\leqq0})\to\Comod(C_\zeta^\sharp),\quad
\Ind_2:\Comod(C_\zeta^{\sharp})\to\Comod(C_\zeta)
\]
such that $\Ind=\Ind_2\circ\Ind_1$ (see \cite{APW1}).
By the Frobenius splitting theorem of Kumar-Littelmann \cite[Theorem 3.8, Corollary 3.9]{KR2}
$r'(N_1)$ is a direct summand of $r(\Ind_1(\BC_{\ell\lambda}))$.
Hence it is sufficient to show 
$R^1\Ind(r(\Ind_1(\BC_{\ell\lambda}))=0$.
By a standard fact on induction functors 
we have $\Ind_1(r(\Ind_1(\BC_{\ell\lambda})))=\Ind_1(\BC_{\ell\lambda})$ and $R^i\Ind_1(r(\Ind_1(\BC_{\ell\lambda})))=0$ for $i>0$.
Moreover, $R^i\Ind_1(\BC_{\ell\lambda})=0$ for $i>0$ by (a relative version of) the Kempf type vanishing theorem 
(\cite[Theorem 5.5]{R-H}, \cite[(7.5) Theorem]{W}, see also \cite{APW1}, \cite{AW}, \cite{Ka1}, \cite{Ka2}).
Hence we obtain
\[
R^1\Ind(r(\Ind_1(\BC_{\ell\lambda})))
=R^1\Ind_2(\Ind_1(\BC_{\ell\lambda}))
=R^1\Ind(\BC_{\ell\lambda})=0
\]
again by the Kempf type vanishing theorem.
\end{proof}
Since $A_1$ is a noetherian ring, we obtain
from Proposition \ref{prop:AzetaoverA1} the following.
\begin{proposition}
\label{prop:A-noether}
$A_\zeta$ is a left and right noetherian ring.
\end{proposition}

Note that the $\Lambda$-graded algebra $A_1$ is the
homogeneous coordinate algebra of the projective variety $\CB=B^-\backslash G$.
Hence we have an identification 
\begin{equation}
\label{eq:A1A1}
\Mod(\CO_\CB)=\CC(A_1,A_1)
\end{equation}
of abelian categories, where $\Mod(\CO_\CB)$ denotes the category of quasi-coherent $\CO_\CB$-modules on the ordinary flag manifold $\CB$.
We set 
\begin{equation}
\omega_{\CB*}=\omega(A_1,A_1)_*:\Mod(\CO_\CB)\to\Mod_\Lambda(A_1).
\end{equation}
For $\lambda\in\Lambda$ we denote by $\CO_\CB(\lambda)\in \Mod(\CO_\CB)$ the invertible $G$-equivariant $\CO_\CB$-module corresponding to $\lambda$.
Then under the identification \eqref{eq:A1A1} we have
\[
\omega_{\CB*}M=
\bigoplus_{\lambda\in\Lambda}\Gamma(\CB,M\otimes_{\CO_\CB}\CO_\CB(\lambda))
\qquad(M\in\Mod(\CO_\CB)),
\]
where $\Gamma(\CB,\,\,):\Mod(\CO_{\CB})\to\BC$ is the global section functor for the algebraic variety $\CB$.
In particular, the functor $\Gamma_{(A_1,A_1)}:\Mod(\CO_\CB)\to\Mod(\BC)$ is identified with 
$\Gamma(\CB,\,\,)$.

For $w\in W$
we set
\[
\Theta_{w}=\bigcup_{\lambda\in\Lambda^+}
(A_1(\lambda)_{w^{-1}\lambda}\setminus\{0\})
\subset A_1.
\]
Then $\Theta_{w}$ is a multiplicative subset of the commutative ring $A_1$, and
the localization $\Theta_{w}^{-1}A_1$ turns out to  be a $\Lambda$-graded $\BC$-algebra.
Moreover, 
the $\BC$-algebra $(\Theta_{w}^{-1}A_1)(0)$ is naturally regarded as the coordinate algebra of the affine open subset $\CB_w:=B^-\backslash B^-N^+w$ of $\CB$.
We denote by $\Mod(\CO_{\CB_w})$ the category of quasi-coherent $\CO_{\CB_w}$-modules.
We have $\Mod(\CO_{\CB_w})=\Mod((\Theta_{w}^{-1}A_1)(0))$.
The functor 
\[
j_{w}^*:\Mod(\CO_\CB)\to\Mod(\CO_{\CB_w})
\]
induced by 
\[
\Mod_\Lambda(A_1)\to\Mod((\Theta_{w}^{-1}A_1)(0))
\qquad
(M\mapsto (\Theta_{w}^{-1}A_1\otimes_{A_1}M)(0))
\]
is nothing but the inverse image functor with respect to the embedding $j_{w}:\CB_w\to\CB$.

\subsection{}
For a $\Lambda$-graded $\BC$-algebra $B$ we define a new $\Lambda$-graded  $\BC$-algebra $B^{(\ell)}$ by
\[
B^{(\ell)}(\lambda)=B(\ell\lambda)
\qquad(\lambda\in\Lambda).
\]
Let
\begin{equation}
\label{eq:(ell)}
(\,\,)^{(\ell)}:\Mod_\Lambda(B)\to\Mod_\Lambda(B^{(\ell)})
\end{equation}
be the exact functor given by 
\[
M^{(\ell)}(\lambda)=M(\ell\lambda)\qquad
(\lambda\in\Lambda)
\]
for $M\in\Mod_\Lambda(B)$.

\begin{lemma}
\label{lem:Fr1}
Let $B$ be a $\Lambda$-graded $\BC$-algebra.
Assume that we are given a homomorphism $\iota:A_\zeta\to B$ of $\Lambda$-graded $\BC$-algebras.
We denote by $\iota':A_1\to B^{(\ell)}$ the induced homomorphism of $\Lambda$-graded $\BC$-algebras.
Assume 
\begin{align*}
\iota(A_\zeta(\lambda)) B(\mu)&=B(\mu)\iota(A_\zeta(\lambda))
\qquad(\lambda,\mu\in\Lambda),\\
\iota'(A_1(\lambda))B^{(\ell)}(\mu)
&=B^{(\ell)}(\mu)\iota'(A_1(\lambda))
\qquad(\lambda,\mu\in\Lambda).
\end{align*}
Then the exact functor 
\begin{equation*}
(\,\,)^{(\ell)}:\Mod_\Lambda(B)\to\Mod_\Lambda(B^{(\ell)})
\end{equation*} induces an equivalence
\begin{equation}
Fr_*:\CC(A_\zeta,B)
\to
\CC(A_1,B^{(\ell)})
\end{equation}
of abelian categories.
Moreover, we have
\begin{equation}
\label{eq:Frobenius}
\omega(A_1,B^{(\ell)})_{*}\circ Fr_*=(\,\,)^{(\ell)}\circ\omega(A_\zeta,B)_{*}.
\end{equation}
\end{lemma}
\begin{proof}
For simplicity we write 
$\omega(A_\zeta,B)^{*}$, 
$\omega(A_1,B^{(\ell)})^{*}$,
$\omega(A_\zeta,B)_{*}$, 
$\omega(A_1,B^{(\ell)})_{*}$
as
$\omega_1^{*}$, 
$\omega_2^{*}$,
$\omega_{1*}$, 
$\omega_{2*}$
respectively.

By Proposition \ref{prop:AzetaoverA1} we see easily that  
for any $\lambda\in\Lambda^+$ there exists some $\mu\in\Lambda^+$ such that $A_\zeta(\nu)A_1(\lambda)=A_\zeta(\ell\lambda+\nu)$ for $\nu\in\mu+\Lambda^+$.
From this we obtain
\begin{equation}
\label{eq:Fr1}
(\Tor(M))^{(\ell)}=\Tor(M^{(\ell)})\qquad(M\in\Mod_\Lambda(B)).
\end{equation}
Hence $M\in \Tor_{\Lambda^+}(A_\zeta,B)$ implies $M^{(\ell)}\in\Tor_{\Lambda^+}(A_1,B^{(\ell)})$.
It follows that we have a well-defined functor 
$Fr_*:\CC(A_\zeta,B)
\to
\CC(A_1,B^{(\ell)})$
satisfying $Fr_*\circ\omega_1^*=\omega_2^*\circ (\,\,)^{(\ell)}$.
We see easily that
\begin{equation*}
\label{eq:Fr2}
(B\otimes_{B^{(\ell)}}N)^{(\ell)}\cong N
\qquad(N\in\Mod_\Lambda(B^{(\ell)})).
\end{equation*}
Hence we have
\[
(Fr_*\circ\omega_1^*)(B\otimes_{B^{(\ell)}}N)
=
\omega_2^*((B\otimes_{B^{(\ell)}}N)^{(\ell)})
=
\omega_2^*(N)
\]
for any $N\in\Mod_\Lambda(B^{(\ell)})$.
It follows that
$Fr_*$ is a dense functor.
Let us show that 
\[
\Hom(\omega_1^*M,\omega_1^*N)
\to
\Hom(\omega_2^*(M^{(\ell)}), \omega_2^*(N^{(\ell)}))
\]
is bijective for any $M, N\in\Mod_\Lambda(B)$.
By $(B\otimes_{B^{(\ell)}}M^{(\ell)})^{(\ell)}\cong M^{(\ell)}$
we see easily that  the canonical morphism 
$B\otimes_{B^{(\ell)}}M^{(\ell)}\to M$ 
belongs to $\Sigma(A_\zeta,B)$, that is, 
$\omega_1^*(B\otimes_{B^{(\ell)}}M^{(\ell)})\cong\omega_1^* M$.
Hence we have
\begin{align*}
&\Hom(\omega_1^*M,\omega_1^*N)\cong
\Hom(\omega_1^*(B\otimes_{B^{(\ell)}}M^{(\ell)}),\omega_1^*N)\\
\cong&
\Hom(B\otimes_{B^{(\ell)}}M^{(\ell)},\omega_{1*}\omega_1^*N)
\cong
\Hom(M^{(\ell)},(\omega_{1*}\omega_1^*N)^{(\ell)}).
\end{align*}
On the other hand we have
\begin{align*}
&\Hom(\omega_2^*(M^{(\ell)}), \omega_2^*(N^{(\ell)}))
\cong
\Hom(M^{(\ell)}, \omega_{2*}\omega_2^*(N^{(\ell)})).
\end{align*}
Therefore, it is sufficient to show 
\begin{equation}
\label{eq:lem:Frr}
(\omega_{1*}\omega_1^*N)^{(\ell)}
\cong
\omega_{2*}\omega_2^*(N^{(\ell)})
\end{equation}
(note that \eqref{eq:lem:Frr} is equivalent to \eqref{eq:Frobenius}).
We may assume that $N=B\otimes_{B^{(\ell)}}P$ for some $P\in\Mod_\Lambda(B^{(\ell)})$.
We may further assume that $\omega_{2*}\omega_2^*P\cong P$.
Hence it is sufficient to show for $P\in\Mod_\Lambda(B^{(\ell)})$ satisfying $\omega_{2*}\omega_2^*P\cong P$ that 
$P\cong (\omega_{1*}\omega_1^*(B\otimes_{B^{(\ell)}}P))^{(\ell)}$.
Since the canonical morphism
$B\otimes_{B^{(\ell)}}P\to
\omega_{1*}\omega_1^*(B\otimes_{B^{(\ell)}}P)$
belongs to $\Sigma(A_\zeta,B)$,
the corresponding morphism
$f:P\to
(\omega_{1*}\omega_1^*(B\otimes_{B^{(\ell)}}P))^{(\ell)}$
belongs to $\Sigma(A_1,B^{(\ell)})$.
By $\omega_{2*}\omega_2^*P\cong P$ we see that $f$ is injective and its cokernel is isomorphic to a submodule of
$(\omega_{1*}\omega_1^*(B\otimes_{B^{(\ell)}}P))^{(\ell)}$.
By 
\[
\Tor((\omega_{1*}\omega_1^*(B\otimes_{B^{(\ell)}}P))^{(\ell)})
=(\Tor(\omega_{1*}\omega_1^*(B\otimes_{B^{(\ell)}}P)))^{(\ell)}=0
\]
we obtain $\Coker(f)=0$. 
It follows that $f$ is an isomorphism.
\end{proof}
The following fact is concerned with ordinary (commutative) projective algebraic geometry and its proof is straightforward. Details are omitted.
\begin{lemma}
\label{lem:Fr2}
Let $F$ be a $\Lambda$-graded $\BC$-algebra, and let $A_1\to F$ be a homomorphism of $\Lambda$-graded $\BC$-algebras.
Assume that $\Image(A_1\to F)$ is central in $F$.
Regard $F$ as an object of $\Mod_\Lambda(A_1)$ and consider 
$\omega_\CB^*F\in\Mod(\CO_{\CB})$.
Then the multiplication of $F$ induces an $\CO_\CB$-algebra structure of $\omega_\CB^*F$, and we have an identification
\begin{equation}
\label{eq:lem:Fr2}
\CC(A_1,F)
=
\Mod(\omega_\CB^*F),
\end{equation}
of abelian categories,
where $\Mod(\omega_\CB^*F)$ denotes the category of quasi-coherent $\omega_\CB^*F$-modules.
\end{lemma}

We define an $\CO_\CB$-algebra $Fr_*\CO_{\CB_\zeta}$ by
\[
Fr_*\CO_{\CB_\zeta}=\omega_\CB^*(A_\zeta^{(\ell)}).
\]
We denote by $\Mod(Fr_*\CO_{\CB_\zeta})$ the category of quasi-coherent $Fr_*\CO_{\CB_\zeta}$-modules.
By Lemma \ref{lem:Fr1} and Lemma \ref{lem:Fr2} we have the following.
\begin{lemma}
We have an equivalence 
\[
Fr_*:\Mod(\CO_{\CB_\zeta})\to\Mod(Fr_*\CO_{\CB_\zeta})
\]
of abelian categories.
Moreover, 
for $M\in\Mod(\CO_{\CB_\zeta})$ we have
\[
R^i\Gamma(M)\simeq R^i\Gamma(\CB,Fr_*(M)),
\]
where $\Gamma(\CB,\,\,):\Mod(\CO_{\CB})\to\Mod(\BC)$ in the right side is the global section functor for $\CB$.
\end{lemma}

\begin{proposition}
\label{prop:AzetaA1}
$(\Theta_e^{-1}A_\zeta)(0)$ is a free $(\Theta_e^{-1}A_1)(0)$-module of rank $\ell^{|\Delta^+|}$. Here $e$ is the identity element of $W$.
Hence the restriction $j_e^*Fr_*\CO_{\CB_\zeta}$ of $Fr_*\CO_{\CB_\zeta}$ to $\CB_e=B^-\backslash B^-N^+$ is a free $\CO_{\CB_e}$-module of rank $\ell^{|\Delta^+|}$.
\end{proposition}
\begin{proof}
Denote by
\[
g:A_\zeta\to(U_\zeta^{L,\geqq0})^*
\]
the composite of 
\[
A_\zeta\subset C_\zeta\subset (U_\zeta^L)^*\to (U_\zeta^{L,\geqq0})^*.
\]
Then $g$ is an algebra homomorphism with respect to the multiplication of $(U_\zeta^{L,\geqq0})^*$ induced by the comultiplication of $U_\zeta^{L,\geqq0}$.
Set 
\[
(U_\zeta^{L,+})^\bigstar=\bigoplus_{\gamma\in Q^+}
(U_{\zeta,\gamma}^{L,+})^*\subset 
(U_{\zeta}^{L,+})^*,
\]
and identify $(U_\zeta^{L,+})^\bigstar$ with a subspace of  $(U_\zeta^{L,\geqq0})^*$ by the embedding
$(U_\zeta^{L,+})^\bigstar\ni\varphi\mapsto\tilde{\varphi}\in(U_\zeta^{L,\geqq0})^*$ 
given by 
\[
\tilde{\varphi}(hx)=\varepsilon(h)\varphi(x)\qquad
(h\in U^{L,0}_\zeta, x\in U^{L,+}_\zeta).
\]
For $\lambda\in\Lambda$ define the algebra homomorphism $\chi_\lambda:U^{L,\geqq0}_\zeta\to\BC$ by
\[
\chi_\lambda(hx)=\chi_\lambda(h)\varepsilon(x)\qquad
(h\in U^{L,0}_\zeta, x\in U^{L,+}_\zeta).
\]
Then for $\varphi\in(U_\zeta^{L,+})^\bigstar$ and $\lambda\in\Lambda$ we have
\[
(\tilde{\varphi}\chi_\lambda)(hx)=\chi_\lambda(h)\tilde{\varphi}(x)\qquad
(h\in U^{L,0}_\zeta, x\in U^{L,+}_\zeta).
\]
Moreover, $(U_\zeta^{L,+})^\bigstar$ is a subalgebra of 
$(U_\zeta^{L,\geqq0})^*$, and
\begin{align*}
&\chi_\lambda\chi_\mu=\chi_{\lambda+\mu}\qquad
(\lambda, \mu\in\Lambda),\\
&\chi_\lambda\tilde{\varphi}=\zeta^{(\lambda,\gamma)}\tilde{\varphi}\chi_\lambda
\qquad
(\lambda\in\Lambda, \varphi\in (U_{\zeta,\gamma}^{L,+})^*)
\end{align*}
in the algebra $(U_\zeta^{L,\geqq0})^*$.
In particular, 
\[
(U_\zeta^{L,\geqq0})^\bigstar
:=\bigoplus_{\lambda\in\Lambda}(U_\zeta^{L,+})^\bigstar
\chi_\lambda
\]
is a subalgebra of $(U_\zeta^{L,\geqq0})^*$.
By \eqref{eq:Adecomp1Z} $g$ induces an injective algebra homomorphism
\[
g':A_\zeta\to(U_\zeta^{L,\geqq0})^\bigstar.
\]
For $\varphi\in A_\zeta(\lambda)_\lambda\setminus\{0\}$ we have $g'(\varphi)\in\BC\chi_\lambda\setminus\{0\}$, and hence $g'$ induces an injective algebra homomorphism
\[
g'':\Theta_e^{-1}A_\zeta\to(U_\zeta^{L,\geqq0})^\bigstar.
\]
Let us show that $g''$ is surjective.
It is sufficient to show that for any $\gamma\in Q^+$ there exists $\lambda\in\Lambda^+$ such that
$g'(A_\zeta(\lambda)_{\lambda-\gamma})=
(U_{\zeta,\gamma}^{L,+})^*\chi_\lambda$.
In view of Lemma \ref{lem:Azeta} the surjectivity of 
$A_\zeta(\lambda)_{\lambda-\gamma}\to
(U_{\zeta,\gamma}^{L,+})^*\chi_\lambda$ is equivalent to the injectivity of 
$U_{\zeta,\gamma}^{L,+}\ni u\mapsto \overline{u}\in L_{+,\zeta}(-\lambda)_{-\lambda+\gamma}$.
This is known to be true for sufficiently large $\lambda$ (see for example \cite[Lemma 2.1]{T2}).
Hence $g''$ is an isomorphism.

Similarly to the above argument the natural algebra homomorphism
\[
g_1:A_1\to(U(\Gb^+))^*
\]
induces an algebra isomorphism
\[
g_1'':\Theta_e^{-1}A_1\to (U(\Gb^+))^\bigstar,
\]
where 
\begin{align*}
&(U(\Gb^+))^\bigstar=\bigoplus_{\lambda\in\Lambda}
(U(\Gn^+))^\bigstar\chi_{1,\lambda},\\
&(U(\Gn^+))^\bigstar
=\bigoplus_{\gamma\in Q^+}(U(\Gn^+)_\gamma)^*,\\
&\chi_{1,\lambda}(hx)=\langle\lambda, h\rangle\varepsilon(x)\qquad
(\lambda\in\Lambda, h\in U(\Gh), x\in U(\Gn^+)).
\end{align*}
Moreover, we have the following commutative diagram
\[
\begin{CD}
\Theta_e^{-1}A_1@>{g_1''}>> (U(\Gb^+))^\bigstar\\
@VVV @VVV\\
\Theta_e^{-1}A_\zeta @>>{g''}> (U_\zeta^{L,\geqq0})^\bigstar,
\end{CD}
\]
where 
$\Theta_e^{-1}A_1\to\Theta_e^{-1}A_\zeta$ is the embedding induced from the inclusion $A_1\subset A_\zeta$, and $(U(\Gb^+))^\bigstar\to(U_\zeta^{L,\geqq0})^\bigstar$ is the injective algebra homomorphism induced by the restriction of  \eqref{eq:LFrobenius}.
Restricting to the degree zero part we obtain algebra isomorphisms 
\[
(\Theta_e^{-1}A_\zeta)(0) \to(U_\zeta^{L,+})^\bigstar,\qquad
(\Theta_e^{-1}A_1)(0)\to (U(\Gn^+))^\bigstar
\]
and the commutative diagram
\[
\begin{CD}
(\Theta_e^{-1}A_1)(0)@>{g_1''}>> (U(\Gn^+))^\bigstar\\
@VVV @VVV\\
(\Theta_e^{-1}A_\zeta)(0) @>>{g''}> (U_\zeta^{L,+})^\bigstar.
\end{CD}
\]

Define a linear map $F:S({U}^-_\zeta)\to(U_\zeta^{L,\geqq0})^\bigstar$ by $(F(y))(z)={}^L\tau(z,y)$ for 
$y\in S({U}^-_\zeta)$ and $z\in U_\zeta^{L,\geqq0}$.
Then we see easily that $F$ is an injective algebra homomorphism whose image is $(U_\zeta^{L,+})^\bigstar$.
Hence we can identify the algebra $(U_\zeta^{L,+})^\bigstar$ with $S({U}^-_\zeta)$.
Under this identification the image of  $(U(\Gn^+))^\bigstar\to(U_\zeta^{L,+})^\bigstar$ coincides with the subalgebra of $S({U}^-_\zeta)$ generated by the central elements $S(f_{\beta_j}^\ell)\,\,(j=1,\dots, N)$.
Hence our assertion is clear from Lemma \ref{lem:PBW-DK}.
\end{proof}

\section{Ring of differential operators}
\subsection{}
We define a subalgebra $D_\BF$ of $\End_\BF(A_\BF)$ by
\[
D_\BF=
\langle
\ell_\varphi, r_\varphi, \deru_u, \sigma_\lambda\mid
\varphi\in A_\BF, u\in U_\BF, \lambda\in\Lambda\rangle,
\]
where 
\[
\ell_\varphi(\psi)=\varphi\psi,\quad
r_\varphi(\psi)=\psi\varphi,\quad
\deru_u(\psi)=u\cdot\psi,\quad
\sigma_\lambda(\psi)=q^{(\lambda,\mu)}\psi
\]
for $\psi\in A_\BF(\mu)$.
We have a natural grading
\begin{align*}
&D_\BF=\bigoplus_{\lambda\in\Lambda^+}D_\BF(\lambda),\\
&D_\BF(\lambda)=\{\Phi\in D_\BF\mid
\Phi(A_\BF(\mu))\subset A_\BF(\lambda+\mu)\quad(\mu\in\Lambda)\}\qquad(\lambda\in\Lambda)
\end{align*}
of $D_\BF$.
We have
\begin{align*}
\deru_u\ell_\varphi=&\sum_{(u)}\ell_{u_{(0)}\cdot\varphi}\deru_{u_{(1)}}
\qquad(u\in U_\BF, \varphi\in A_\BF),\\
\deru_u\sigma_\lambda=&\sigma_\lambda\deru_u
\qquad(u\in U_\BF, \lambda\in\Lambda),\\
\sigma_\lambda\ell_\varphi=&
q^{(\lambda,\mu)}\ell_\varphi\sigma_\lambda
\qquad(\lambda\in\Lambda, \varphi\in A_\BF(\mu)).
\end{align*}
We have also
\begin{equation}
\label{eq:Har-center-in-D0}
z\in Z(U_\BF),\,\,
\iota(z)=\sum_{\lambda\in\Lambda}a_\lambda e({2\lambda})\quad\Longrightarrow\quad
\deru_z
=\sum_{\lambda\in\Lambda}a_\lambda\sigma_{2\lambda}.
\end{equation}

We take bases $\{x_p\}_p$ and $\{y_p\}_p$ of $U_\BF^+$ and $U_\BF^-$ respectively and elements $\beta_p\in Q^+$ for each $p$ such that 
\begin{align}
\label{eq:xy1}
&x_p\in U_{\BF,\beta_p}^+,\qquad
y_p\in U_{\BF,-\beta_p}^-,\\
\label{eq:xy2}
&\tau(x_{p_1},y_{p_2})=\delta_{p_1,p_2}.
\end{align}
\begin{lemma}
\label{lem:rvarphi}
Let $\lambda\in\Lambda^+$ and $\xi\in \Lambda$.
For $\varphi\in A_\BF(\lambda)_\xi$ we have
\begin{equation}
\label{eq:rvarphi}
r_\varphi=
\sum_p\ell_{{y_p}\cdot\varphi}\deru_{x_pk_{-\xi}}\sigma_\lambda
=\sum_p\ell_{(Sx_p)\cdot\varphi}\deru_{y_pk_{\beta_p}k_\xi}\sigma_{-\lambda}.
\end{equation}
\end{lemma}
\begin{proof}
The first equality is shown in \cite[Lemma 5.1]{T2} by the following argument.
Let $\CR\in U_\BF\hat{\otimes}U_\BF$ be the universal $R$-matrix. 
Then we have
\begin{align*}
&\langle r_\varphi(\psi),u\rangle=
\langle \psi\varphi,u\rangle=
\langle \psi\otimes\varphi,\Delta(u)\rangle=
\langle \varphi\otimes\psi,\tau(\Delta(u))\rangle\\
=&
\langle \varphi\otimes\psi,\CR\Delta(u)\CR^{-1}\rangle=
\langle \CR^{-1}\cdot(\varphi\otimes\psi)\cdot\CR,\Delta(u)\rangle
\\
=&
\langle m(\CR^{-1}\cdot(\varphi\otimes\psi)\cdot\CR),u\rangle
\end{align*}
for $\varphi, \psi\in A_\BF$ and $u\in U_\BF$.
Here $\tau:U_\BF\otimes U_\BF\to U_\BF\otimes U_\BF$  is the linear map sending $a\otimes b$ to $b\otimes a$.
Hence we obtain $r_\varphi(\psi)=m(\CR^{-1}\cdot(\varphi\otimes\psi)\cdot\CR)$.
By rewriting it using an explicit form of $\CR$ we obtain the first equality in \eqref{eq:rvarphi}.
Applying the same argument to another $R$-matrix $\tau(\CR^{-1})$ we also obtain the second equality in \eqref{eq:rvarphi}.
Details are omitted.
\end{proof}
Set 
\begin{equation}
E_\BF=A_\BF\otimes U_\BF\otimes\BF[\Lambda],
\end{equation}
and regard $A_\BF$, $U_\BF$, $\BF[\Lambda]$ as subspaces of $E_\BF$ by the natural embeddings
$A_\BF\ni\varphi\mapsto\varphi\otimes1\otimes1\in E_\BF$ etc.
Then we have an $\BF$-algebra structure of $E_\BF$ such that
the natural embeddings $A_\BF\to E_\BF$, $U_\BF\to E_\BF$, $\BF[\Lambda]\to E_\BF$ are algebra homomorphisms, and
\begin{align*}
u\varphi=
&\sum_{(u)}
({u_{(0)}\cdot\varphi}){u_{(1)}}
\qquad(u\in U_\BF, \varphi\in A_\BF),\\
ue(\lambda)=&e(\lambda) u
\qquad(u\in U_\BF, \lambda\in\Lambda),\\
e(\lambda)\varphi=&
q^{(\lambda,\mu)}\varphi e(\lambda)
\qquad(\lambda\in\Lambda, \varphi\in A_\BF(\mu))
\end{align*}
in ${E_\BF}$.
Moreover, we have a surjective algebra homomorphism $E_\BF\to D_\BF$ given by $\varphi\mapsto\ell_\varphi\,(\varphi\in A_\BF)$, 
$u\mapsto\deru_u\,(u\in U_\BF)$, $e(\lambda)\mapsto\sigma_\lambda\,(\lambda\in\Lambda)$.

For $\varphi\in A_\BF(\lambda)_\xi$ with $\lambda\in\Lambda^+$, $\xi\in\Lambda$ we set
\begin{align}
\Omega_1(\varphi)&=
\sum_p({y_p}\cdot\varphi){x_pk_{-\xi}}e(\lambda)
\in E_\BF,\\
\Omega_2(\varphi)&=
\sum_p({(Sx_p)\cdot\varphi}){y_pk_{\beta_p}k_\xi}e({-\lambda})
\in E_\BF,\\
\Omega(\varphi)&=\Omega_1(\varphi)-\Omega_2(\varphi)
\in E_\BF.
\end{align}
We extend $\Omega$, $\Omega_1$, $\Omega_2$ to whole $A_\BF$ by linearity.
By Lemma \ref{lem:rvarphi} we have $\Omega(\varphi)\in\Ker(E_\BF\to D_\BF)$.
We set
\[
D'_\BF=E_\BF/
\sum_{\varphi\in A_\BF}E_\BF\Omega(\varphi)E_\BF.
\]
Then we have a sequence of surjective algebra homomorphisms
\begin{equation}
\label{EDD}
E_\BF\to{D}'_\BF\to{D}_\BF.
\end{equation}
\begin{lemma}
\label{lem:Omega}
For $\varphi\in A_\BF(\lambda)$ with $\lambda\in\Lambda^+$ and $i=1, 2$ we have
\begin{align}
\label{eq:lem:Omega1}
e(\mu)\Omega_i(\varphi)&=
q^{(\lambda,\mu)}\Omega_i(\varphi)e(\mu)
&(\mu\in\Lambda),\\
\label{eq:lem:Omega2}
\psi\Omega_i(\varphi)&=
\Omega_i(\varphi)\psi
&(\psi\in A_\BF),\\
\label{eq:lem:Omega3}
u\Omega_i(\varphi)&=
\sum_{(u)}\Omega_i(u_{(1)}\cdot\varphi)u_{(0)}
&(u\in U_\BF),\\
\label{eq:lem:Omega4}
\Omega_i(\varphi\psi)&=
\Omega_i(\psi)\Omega_i(\varphi)
&(\varphi, \psi\in A_\BF)
\end{align}
in $E_\BF$.
\end{lemma}
\begin{proof}
We will only give a proof for the case $i=1$.
The proof for the case $i=2$ is similar.
The proof of \eqref{eq:lem:Omega1} is easy and omitted.

Let us show \eqref{eq:lem:Omega2}.
Let $\varphi\in A_\BF(\lambda)_\xi$, $\psi\in A_\BF(\mu)_\eta$.
By the formula
\begin{equation}
\label{eq:lem:Omega:pf1}
\sum_p \Delta x_p\otimes y_p=
\sum_{p,r}x_pk_{\beta_r}\otimes x_r\otimes y_py_r
\end{equation}
(see \cite[(4.3.16)]{T1}) we have
\begin{align*}
\Omega_1(\varphi)\psi&=
\sum_p({y_p}\cdot\varphi){x_pk_{-\xi}}e(\lambda)\psi\\
&=q^{(\lambda,\mu)-(\xi,\eta)}
\sum_p({y_p}\cdot\varphi)x_p\psi k_{-\xi}e(\lambda)\\
&=q^{(\lambda,\mu)-(\xi,\eta)}
\sum_{p, r}({y_py_r}\cdot\varphi)(x_pk_{\beta_r}\cdot\psi)
x_r k_{-\xi}e(\lambda)\\
&=q^{(\lambda,\mu)-(\xi,\eta)}
\sum_{r}
q^{(\beta_r,\eta)}
\left(
\sum_p({y_py_r}\cdot\varphi)(x_p\cdot\psi)
\right)
x_r k_{-\xi}e(\lambda).
\end{align*}
By Lemma \ref{lem:rvarphi} we have
\begin{align*}
&\sum_p({y_py_r}\cdot\varphi)(x_p\cdot\psi)\\
=&
\sum_p
r_{x_p\cdot\psi}({y_py_r}\cdot\varphi)\\
=&
\sum_{p,s}
\ell_{S(x_s)x_p\cdot\psi}
\deru_{y_sk_{\beta_s+\beta_p+\eta}}\sigma_{-\mu}
({y_py_r}\cdot\varphi)\\
=&\sum_{p,s}
q^{-(\lambda,\mu)+(\beta_s+\beta_p+\eta,\xi-\beta_p-\beta_r)}
(S(x_s)x_p\cdot\psi)
({y_sy_py_r}\cdot\varphi)\\
=&\sum_{p,s}
q^{-(\lambda,\mu)+(\beta_s+\beta_p+\eta,\xi-\beta_r)}
(S(x_sk_{\beta_p})x_p\cdot\psi)
({y_sy_py_r}\cdot\varphi).
\end{align*}
By the formula
\begin{equation}
\label{eq:lem:Omega:pf2}
\sum_{\beta_p+\beta_r=\gamma}
S(x_pk_{\beta_r})x_r\otimes y_py_r
=\begin{cases}
1\otimes1\quad&(\gamma=0)\\
0&(\gamma\ne0),
\end{cases}
\end{equation}
which is a consequence of 
\eqref{eq:lem:Omega:pf1} and $m\circ(S\otimes1)\circ\Delta=\varepsilon$, 
we obtain
\[
\sum_p({y_py_r}\cdot\varphi)(x_p\cdot\psi)
=
q^{-(\lambda,\mu)+(\eta,\xi-\beta_r)}
\psi
({y_r}\cdot\varphi).
\]
It follows that 
\[
\Omega_1(\varphi)\psi
=\psi
\sum_{r}
({y_r}\cdot\varphi)
x_r k_{-\xi}e(\lambda)
=\psi\Omega_1(\varphi).
\]
The formula \eqref{eq:lem:Omega2} is verified.

Let us show \eqref{eq:lem:Omega3}.
It is sufficient to consider the three cases; $u\in U_\BF^0$, $u\in U_\BF^-$, $u\in U_\BF^+$.
Let $\varphi\in A_\BF(\lambda)_\xi$.
For $u\in U_\BF$ we have
\[
u\Omega_1(\varphi)
=\sum_{p,(u)}(u_{(0)}y_p\cdot\varphi)u_{(1)}x_pk_{-\xi}e(\lambda)
\]
and
\begin{align*}
&\sum_{(u)}
\Omega_1(u_{(1)}\cdot\varphi)u_{(0)}\\
=&
\sum_{p,(u)}(y_pu_{(1)}\cdot\varphi)x_pk_{-\xi-\wt(u_{(1)})}e(\lambda)
u_{(0)}\\
=&
\sum_{p,(u)}
q^{-(\xi+\wt(u_{(1)}),\wt(u_{(0)}))}
(y_pu_{(1)}\cdot\varphi)x_pu_{(0)}k_{-\xi-\wt(u_{(1)})}e(\lambda).
\end{align*}
Hence our assertion is equivalent to
\begin{equation}
\label{eq:lem:Omega:pf3}
\begin{split}
&\sum_{p,(u)}(u_{(0)}y_p\cdot\varphi)u_{(1)}x_p\\
=&\sum_{p,(u)}
q^{-(\xi+\wt(u_{(1)}),\wt(u_{(0)}))}
(y_pu_{(1)}\cdot\varphi)x_pu_{(0)}k_{-\wt(u_{(1)})}.
\end{split}
\end{equation}
Here, we have used the following notation. 
For $u\in U_\BF$ such that $k_\nu uk_\nu^{-1}=q^{(\nu,\mu)}u$ for any $\nu\in\Lambda$ we write $\mu=\wt(u)$.
Moreover, in the expansion $\Delta u=\sum_{(u)}u_{(0)}\otimes u_{(1)}$ the elements $u_{(0)}$ and $u_{(1)}$ are taken to be weight vectors.
The proof of \eqref{eq:lem:Omega:pf3} in the case $u\in U_\BF^0$ is easy and omitted.
Let us consider the case $u\in U_\BF^-$.
By Lemma \ref{lem:Drinfeld paring} and the formula
\begin{equation}
\label{eq:lem:Omega:pf4}
\sum_p\Delta_2 x_p\otimes y_p
=\sum_{p,r,s}
x_pk_{\beta_r+\beta_s}\otimes
x_rk_{\beta_s}\otimes
x_s\otimes
y_py_ry_s,
\end{equation}
which is a consequence of \eqref{eq:lem:Omega:pf1}
we have
\begin{align*}
&\sum_{p,(u)}(u_{(0)}y_p\cdot\varphi)u_{(1)}x_p\\
=&
\sum_{p,(u)_3}(u_{(0)}y_p\cdot\varphi)
\left(
\sum_{(x_p)_2}
\tau(x_{p(0)},Su_{(1)})
\tau(x_{p(2)},u_{(3)})
x_{p(1)}u_{(2)}
\right)\\
=&
\sum_{p,r,s,(u)_3}
\tau(x_pk_{\beta_r+\beta_s},Su_{(1)})
\tau(x_s,u_{(3)})
(u_{(0)}y_py_ry_s\cdot\varphi)
x_rk_{\beta_s}u_{(2)}.
\end{align*}
By the definition of the coproduct $\Delta$ and the antipode $S$ we have
$u_{(1)}\in U_{\BF}^-k_{\wt(u_{(0)})}$, 
$Su_{(1)}\in U_{\BF}^-k_{-\wt(u_{(0)})-\wt(u_{(1)})}$.
Hence by Lemma \ref{lem:Drinfeld paring} and the definition of $x_p$, $y_p$ we have
\begin{align*}
&\sum_{p,(u)}(u_{(0)}y_p\cdot\varphi)u_{(1)}x_p\\
=&
\sum_{p,r,s,(u)_3}
q^{(\beta_s+\beta_r,\wt(u_{(0)})+\wt(u_{(1)}))}
\tau(x_p,(Su_{(1)})k_{\wt(u_{(0)})+\wt(u_{(1)})})\\
&\qquad\qquad\qquad\qquad\qquad\qquad
\tau(x_s,u_{(3)})(u_{(0)}y_py_ry_s\cdot\varphi)
x_rk_{\beta_s}u_{(2)}\\
=&
\sum_{r,s,(u)_3}
q^{(\beta_s+\beta_r,\wt(u_{(0)})+\wt(u_{(1)}))}
\tau(x_s,u_{(3)})\\
&\qquad\qquad\qquad\qquad\qquad
(u_{(0)}(Su_{(1)})k_{\wt(u_{(0)})+\wt(u_{(1)})}y_ry_s\cdot\varphi)
x_rk_{\beta_s}u_{(2)}.
\end{align*}
By 
\[
\sum_{(u)_3}
u_{(0)}(Su_{(1)})\otimes u_{(2)}\otimes  u_{(3)}
=
\sum_{(u)}
1\otimes u_{(0)}\otimes  u_{(1)}
\]
we have
\begin{align*}
\sum_{p,(u)}(u_{(0)}y_p\cdot\varphi)u_{(1)}x_p
=
\sum_{r,s,(u)}
\tau(x_s,u_{(1)})
(y_ry_s\cdot\varphi)
x_rk_{\beta_s}u_{(0)}.
\end{align*}
By $u_{(1)}\in U_\BF^-k_{\wt(u_{(0)})}$ 
we finally obtain
\begin{align*}
&\sum_{p,(u)}(u_{(0)}y_p\cdot\varphi)u_{(1)}x_p\\
=&
\sum_{r,s,(u)}
\tau(x_s,u_{(1)}k_{-\wt(u_{(0)})})
(y_ry_s\cdot\varphi)
x_rk_{\beta_s}u_{(0)}
\\
=&
\sum_{r,(u)}
(y_ru_{(1)}k_{-\wt(u_{(0)})}\cdot\varphi)
x_rk_{-\wt(u_{(1)})}u_{(0)}
\\
=&\sum_{r,(u)}
q^{-(\xi+\wt(u_{(1)}),\wt(u_{(0)}))}
(y_ru_{(1)}\cdot\varphi)x_ru_{(0)}k_{-\wt(u_{(1)})}.
\end{align*}
The formula \eqref{eq:lem:Omega:pf3} for $u\in U_\BF^-$ is shown.
Let us consider the case $u\in U_\BF^+$.
By
\begin{equation}
\label{eq:lem:Omega:pf1a}
\sum_p x_p\otimes \Delta y_p=
\sum_{p,r}x_rx_p\otimes y_p\otimes y_rk_{-\beta_p}
\end{equation}
(see \cite[(4.3.17)]{T1}) we have
\[
\sum_px_p\otimes \Delta_2 y_p
=\sum_{p,r,s}
x_sx_rx_p\otimes
y_p\otimes
y_rk_{-\beta_p}\otimes
y_sk_{-\beta_p-\beta_r}.
\]
Hence by Lemma \ref{lem:Drinfeld paring}
we obtain
\begin{align*}
&\sum_{p,(u)}(u_{(0)}y_p\cdot\varphi)u_{(1)}x_p\\
=&
\sum_{p,(u)_3}
\sum_{(y_p)_2}
\tau(u_{(0)},y_{p(0)})
\tau(u_{(2)},Sy_{p(2)})
(y_{p(1)}u_{(1)}\cdot\varphi)u_{(3)}x_p\\
=&
\sum_{p,r,s,(u)_3}
\tau(u_{(0)},y_{p})
\tau(u_{(2)},S(y_sk_{-\beta_p-\beta_r}))
(y_rk_{-\beta_p}u_{(1)}\cdot\varphi)u_{(3)}x_sx_rx_p\\
=&
\sum_{p,r,s,(u)_3}
q^{-(\beta_p+\beta_r,\wt(u_{(2)})+\wt(u_{(3)}))}
\tau(u_{(0)}k_{-\wt(u_{(1)})-\wt(u_{(2)})-\wt(u_{(3)})},y_{p})
\\
&\qquad\qquad
\tau((S^{-1}u_{(2)})k_{\wt(u_{(2)})+\wt(u_{(3)})},y_s)
(y_rk_{-\beta_p}u_{(1)}\cdot\varphi)u_{(3)}x_sx_rx_p\\
=&
\sum_{r,(u)_3}
q^{-(\wt(u_{(0)})+\beta_r,\wt(u_{(2)})+\wt(u_{(3)}))}
(y_rk_{-\wt(u_{(0)})}u_{(1)}\cdot\varphi)
\\
&\qquad\qquad
u_{(3)}(S^{-1}u_{(2)})k_{\wt(u_{(2)})+\wt(u_{(3)})}
x_ru_{(0)}k_{-\wt(u_{(1)})-\wt(u_{(2)})-\wt(u_{(3)})}\\
=&
\sum_{r,(u)}
(y_rk_{-\wt(u_{(0)})}u_{(1)}\cdot\varphi)
x_ru_{(0)}k_{-\wt(u_{(1)})}\\
=&
\sum_{r,(u)}
q^{-(\xi+\wt(u_{(1)}),\wt(u_{(0)}))}
(y_ru_{(1)}\cdot\varphi)
x_ru_{(0)}k_{-\wt(u_{(1)})}.
\end{align*}
The formula \eqref{eq:lem:Omega:pf3} for $u\in U_\BF^+$ is also shown.

Let us finally show \eqref{eq:lem:Omega4}.
We may assume $\varphi\in A_\BF(\lambda)_\xi$ and $\psi\in A_\BF(\mu)_\eta$ for $\lambda, \mu\in \Lambda^+, \xi,\eta\in\Lambda$.
Then we have
\begin{align*}
&\Omega_1(\varphi\psi)\\
=&
\sum_p
\sum_{(y_p)}
(y_{p(0)}\cdot\varphi)(y_{p(1)}\cdot\psi)x_pk_{-(\xi+\eta)}
e(\lambda+\mu)\\
=&
\sum_{p,r}
(y_{p}\cdot\varphi)(y_{r}k_{-\beta_p}\cdot\psi)x_rx_pk_{-(\xi+\eta)}
e(\lambda+\mu)\\
=&
\sum_{p,r}
q^{-(\beta_p,\eta)}
(y_{p}\cdot\varphi)(y_{r}\cdot\psi)x_rx_pk_{-(\xi+\eta)}
e(\lambda+\mu)
\end{align*}
by \eqref{eq:lem:Omega:pf1a}.
On the other hand we have
\begin{align*}
\Omega_1(\psi)\Omega_1(\varphi)
&=\Omega_1(\psi)
\sum_p
(y_{p}\cdot\varphi)x_pk_{-\xi}
e(\lambda)\\
&=
\sum_p
(y_{p}\cdot\varphi)
\Omega_1(\psi)
x_pk_{-\xi}
e(\lambda)\\
&=
\sum_{p,r}
(y_{p}\cdot\varphi)
(y_{r}\cdot\psi)x_rk_{-\eta}
e(\mu)
x_pk_{-\xi}
e(\lambda)\\
&=
\sum_{p,r}
q^{-(\beta_p,\eta)}
(y_{p}\cdot\varphi)(y_{r}\cdot\psi)x_rx_pk_{-(\xi+\eta)}
e(\lambda+\mu).
\end{align*}
Here, the second equality is a consequence of \eqref{eq:lem:Omega2}.
We are done.
\end{proof}
By Lemma \ref{lem:Omega} we have 
\[
D'_\BF=
E_\BF/
\sum_{\varphi\in A_\BF}
A_\BF\Omega(\varphi)U_\BF\BF[\Lambda].
\]
We have a $\Lambda$-graded $\BF$-algebra structure of $E_\BF$ given by $E_\BF(\lambda)=A_\BF(\lambda)U_\BF\BF[\Lambda]$ for $\lambda\in\Lambda$.
This also induces a $\Lambda$-graded $\BF$-algebra structure of 
${D}'_\BF$ by
${D}'_\BF(\lambda)=\Image(E_\BF(\lambda)\to{D}'_\BF)$.
Then \eqref{EDD} is a sequence of homomorphisms of $\Lambda$-graded algebras.

\subsection{}
Since $A_\BF$ belongs to $\Mod_{int}(U_\BF)$ as a $U_\BF$-module, we have a natural group homomorphism
\begin{equation}
\label{eq:braid-D0}
\BB\to \End_\BF(A_\BF)^\times.
\end{equation}
It induces a group homomorphism
\begin{equation}
\label{eq:braid-D}
\BB\to\Aut_{alg}(D_\BF)
\qquad(T\mapsto[\Phi\mapsto T\star \Phi:=T\circ \Phi\circ T^{-1}])
\end{equation}
(see \cite[Proposition 5.2]{T2}).
We will show that this naturally lifts to
group homomorphisms
\[
\BB\to\Aut_{alg}(E_\BF),\qquad
\BB\to\Aut_{alg}(D'_\BF).
\]
\begin{lemma}
\label{lem:braid-D}
\begin{itemize}
\item[\rm(i)]
For $T\in\BB$ we have
\begin{align*}
&T\star\deru_u=\deru_{T(u)}\qquad
(u\in U_\BF),\\
&T\star\sigma_\lambda=\sigma_\lambda\qquad
(\lambda\in\Lambda).
\end{align*}
\item[\rm(ii)]
For $i\in I$ write
\begin{align*}
\exp_{q_i}((q_i-q_i^{-1})k_i^{-1}e_i\otimes f_ik_i)
&=\sum_{n=0}^\infty a_{i,n}\otimes b_{i, n},\\
\exp_{q_i^{-1}}(-(q_i-q_i^{-1})f_i\otimes e_i)
&=\sum_{n=0}^\infty a'_{i,n}\otimes b'_{i, n}.
\end{align*}
Then for $\varphi\in A_\BF$ we have
\begin{align*}
T_i\star\ell_\varphi
&=\sum_{n=0}^\infty \ell_{a_{i,n}\cdot(T_i(\varphi))}\deru_{b_{i, n}},\\
T_i^{-1}\star\ell_\varphi
&=\sum_{n=0}^\infty 
\ell_{a'_{i,n}\cdot(T_i^{-1}(\varphi))}\deru_{b'_{i, n}}.
\end{align*}
\end{itemize}
\end{lemma}
\begin{proof}
The proof of (i) is easy and omitted.
Let us show (ii).
In general for $T\in\BB$ and $\varphi, \psi\in A_\BF$ we have
\begin{align*}
&(T\star\ell_\varphi)(\psi)
=T(\varphi\cdot T^{-1}(\psi))
=Tm(\varphi\otimes T^{-1}(\psi))\\
=&m(\Delta T)(\varphi\otimes T^{-1}(\psi))
=m(\Delta T)(T^{-1}\otimes T^{-1})
(T(\varphi)\otimes \psi).
\end{align*}
Hence the assertion follows from Lemma \ref{lem:DeltaT}.
\end{proof}
In particular, the action of $\BB$ on $D_\BF$ preserves the subalgebra
\[
D^\dagger_\BF=\langle\deru_u,\ell_\varphi\mid u\in U_\BF, \varphi\in A_\BF\rangle\subset D_\BF.
\]
We first define an action of $\BB$ on the subalgebra $E^\dagger_\BF=A_\BF\otimes U_\BF$ of 
$E_\BF$.
 
For $\Phi=\sum_i\varphi_i\otimes u_i\in E^\dagger_\BF$ and $M\in\Mod_{int}(U_\BF)$ we define
\[
\Phi_M:M\to A_\BF\otimes M
\]
by $\Phi_M(m)=\sum_i\varphi_i\otimes u_im\,\,(m\in M)$.
By \cite[5.11]{Jan} we have the following.
\begin{lemma}
\label{lem:unique-PhiM}
Let $\Phi\in E^\dagger_\BF$. 
If $\Phi_M=0$ for any $M\in\Mod_{int}(U_\BF)$, then we have $\Phi=0$.
\end{lemma}
\begin{proposition}
\label{prop:braid-action-Edagger}
There exists a group homomorphism
\[
\BB\to\Aut_{alg}(E^\dagger_\BF)
\qquad(T\mapsto[\Phi\mapsto T\star \Phi])
\]
such that 
$(T\star \Phi)_M=(\Delta T)\Phi_M T^{-1}$ for $T\in\BB$, $\Phi\in E^\dagger_\BF$, $M\in\Mod_{int}(U_\BF)$.
Here, $\Delta T:A_\BF\otimes M\to A_\BF\otimes M$ is the action of $T\in\BB$ on $A_\BF\otimes M\in\Mod_{int}(U_\BF)$.
\end{proposition}
\begin{proof}
We first note the following formula whose proof is easy and omitted;
\begin{equation}
(\Phi\Psi)_M=(m\otimes 1)\Phi_{A_\BF\otimes M}\Psi_M\qquad
(\Phi, \Psi\in E_\BF^\dagger, M\in\Mod_{int}(U_\BF)).
\end{equation}

Let $T\in\BB$.
For $\Phi\in E_\BF^\dagger$ there exists at most one $T\star\Phi\in E_\BF^\dagger$ satisfying $(T\star \Phi)_M=(\Delta T)\Phi_M T^{-1}$ for any $M\in\Mod_{int}(U_\BF)$ (see Lemma \ref{lem:unique-PhiM}).
We claim that if $T\star\Phi$ exists for any $\Phi\in E_\BF^\dagger$, then 
$E_\BF^\dagger\ni\Phi\mapsto T\star\Phi\in E_\BF^\dagger$ is an algebra homomorphism.
We have
\[
(T\star 1)_M(v)
=(\Delta T)1_M T^{-1}(v)
=(\Delta T)(1\otimes T^{-1}(v))=1\otimes v
\]
for any $v\in M\in\Mod_{int}(U_\BF)$.
Here the last equality is a consequence of $\BF\otimes M\cong M$ as a $U_\BF$-module.
By $1_M(v)=1\otimes v$ we obtain $T\star 1=1$ by Lemma \ref{lem:unique-PhiM}.
For $\Phi, \Psi\in E_\BF^\dagger$
we have
\begin{align*}
&((T\star\Phi)(T\star\Psi))_M
=(m\otimes 1)(T\star\Phi)_{A\otimes M}(T\star\Psi)_M\\
=&(m\otimes 1)(\Delta_2T)\Phi_{A\otimes M}
(\Delta T)^{-1}(\Delta T)\Psi_M T^{-1}\\
=&(\Delta T)(m\otimes 1)\Phi_{A\otimes M}\Psi_M T^{-1}
=(\Delta T)(\Phi\Psi)_M T^{-1}=(T\star(\Phi\Psi))_M
\end{align*}
for any $M\in\Mod_{int}(U_\BF)$.
Here, we used the fact that the multiplication $m:A_\BF\otimes A_\BF\to A_\BF$ is a homomorphism of $U_\BF$-modules.
Hence we have $(T\star\Phi)(T\star\Psi)=T\star(\Phi\Psi)$.
Our claim is verified.

Let $T, T'\in\BB$, and assume that
$T\star\Phi$, $T'\star\Phi$ exist for any $\Phi\in E_\BF^\dagger$.
Then we have
\[
(T\star(T'\star\Phi))_M
=(\Delta T)(\Delta T')\Phi_M(T')^{-1}T^{-1}
=\Delta(TT')\Phi_M(TT')^{-1}
\]
for $\Phi\in E_\BF^\dagger$, $M\in\Mod_{int}(U_\BF)$.
Hence 
 $(TT')\star\Phi$ exists  and we have $(TT')\star\Phi=T\star(T'\star\Phi)$ for any $\Phi\in E_\BF^\dagger$.
 
 It remains to show the existence of $T\star\Phi$ for $T=T_i^{\pm1}$, $\Phi\in E_\BF^\dagger$.
 For $\Phi=\varphi\otimes u\in E_\BF^\dagger$ \,\,($\varphi\in A_\BF$, $u\in U_\BF$) we have
 \begin{align*}
 &((\Delta T_i) \Phi_M T_i^{-1})(v)
 =(\Delta T_i) (\varphi\otimes uT_i^{-1}v)\\
 =&(\Delta T_i)(T_i^{-1}\otimes T_i^{-1})
 (T_i(\varphi)\otimes T_i(u)v)\\
 =&\sum_n
 a_{i,n}\cdot T_i(\varphi)\otimes b_{i,n}T_i(u)v\\
 =&
 \left(\sum_n
 a_{i,n}\cdot T_i(\varphi)\otimes b_{i,n}T_i(u)\right)_M(v)
 \end{align*}
 for $v\in M\in\Mod_{int}(U_\BF)$.
 Hence we obtain
 \[
 T_i\star(\varphi\otimes u)=\sum_n
 a_{i,n}\cdot T_i(\varphi)\otimes b_{i,n}T_i(u)
 \qquad(\varphi\in A_\BF, u\in U_\BF).
 \]
 Similarly, we have
  \[
 T_i^{-1}\star(\varphi\otimes u)=\sum_n
 a'_{i,n}\cdot T_i^{-1}(\varphi)\otimes b'_{i,n}T_i^{-1}(u)
 \qquad(\varphi\in A_\BF, u\in U_\BF).
 \]
 We are done.
\end{proof}
\begin{lemma}
\label{lem:TstarFormula}
\begin{itemize}
\item[\rm(i)]
$T\star u=T(u)$ 
 \qquad$(T\in \BB, \,\,u\in U_\BF)$.
\item[\rm(ii)]
$T_i\star\varphi=\sum_n
 a_{i,n}\cdot T_i(\varphi)\otimes b_{i,n}$ 
 \qquad$(i\in I,\,\,\varphi\in A_\BF)$.
\item[\rm(iii)]
$T_i^{-1}\star\varphi=\sum_n
 a'_{i,n}\cdot T_i^{-1}(\varphi)\otimes b'_{i,n}$
 \qquad$(i\in I,\,\,\varphi\in A_\BF)$.
\end{itemize}
\end{lemma}
\begin{proof}
For $T\in\BB$, $u\in U_\BF$, $v\in M\in\Mod_{int}(U_\BF)$ we have
\begin{align*}
&(T\star u)_M(v)
=(\Delta T)u_MT^{-1}(v)
=(\Delta T)(1\otimes uT^{-1}(v))\\
=&1\otimes TuT^{-1}v
=1\otimes T(u)v=(T(u))_M(v).
\end{align*}
Hence (i) holds.
The statements (ii) and (iii) are already shown in the proof of Proposition \ref{prop:braid-action-Edagger}.
\end{proof}
By Lemma \ref{lem:TstarFormula}
we have 
\[
T\star(A_\BF(\lambda)\otimes U_\BF)=A_\BF(\lambda)\otimes U_\BF
\qquad(\lambda\in\Lambda^+).
\]
Hence, for $T\in\BB$ the algebra automorphism $T\star(\cdot):E^\dagger_\BF\to E^\dagger_\BF$ is  naturally extended  to that of $E_\BF$ by setting 
\[
T\star e(\mu)=e(\mu)\qquad(\mu\in \Lambda).
\]
By this we obtain
a group homomorphism
\begin{equation}
\label{eq:braid-action-E}
\BB\to\Aut_{alg}(E_\BF)
\qquad(T\mapsto[\Phi\mapsto T\star \Phi]).
\end{equation}

\begin{proposition}
\label{prop:braid-action-Dprime}
For $T\in\BB$ we have
\[
T\star\Ker(E_\BF\to D_\BF')\subset
\Ker(E_\BF\to D_\BF').
\]
\end{proposition}
\begin{proof}
It is sufficient to show
\[
T\star\Omega(\varphi)\in\sum_{\psi\in A_\BF}
E_\BF\Omega(\psi)E_\BF
\]
for $T=T_i^{\pm1}$, $\varphi\in A_\BF$.

In general, for $\varphi\in A_\BF(\lambda)_\xi$, we write
\begin{align*}
&
\Omega(\varphi)=\Omega_1'(\varphi)e(\lambda)-\Omega_2'(\varphi)e(-\lambda),
\\
&
\Omega_1'(\varphi)=
\sum_p(y_p\cdot\varphi)x_pk_{-\xi}
\in E_\BF^\dagger,
\\
&
\Omega_2'(\varphi)=
\sum_p((Sx_p)\cdot\varphi)y_pk_{\beta_p}k_\xi\in E_\BF^\dagger.
\end{align*}

Let $M, M'\in\Mod_{int}(U_\BF)$.
We define a linear automorphism $\kappa_{M,M'}$ of $M\otimes M'$ by
$\kappa_{M,M'}|_{M_\lambda\otimes M'_\mu}=q^{(\lambda,\mu)}\id$ for $\lambda, \mu\in \Lambda$.
We also define a linear isomorphism $\tau_{M,M'}:M\otimes M'\to M'\otimes M$ by $\tau_{M,M'}(v\otimes v')=v'\otimes v$.
Set
\[
\CR_{M,M'}
=\kappa_{M,M'}^{-1}\circ
\left(
\sum_pq^{(\beta_p,\beta_p)}
k_{\beta_p}^{-1}x_p\otimes k_{\beta_p}y_p
\right)
\in\End_\BF(M\otimes M').
\]
Then $\CR_{M,M'}$ is invertible and we have
\[
\CR_{M,M'}^{-1}
=
\left(
\sum_pq^{(\beta_p,\beta_p)}
Sx_p\otimes k_{\beta_p}y_p
\right)
\circ \kappa_{M,M'}
\in\End_\BF(M\otimes M').
\]
Moreover, we have
\begin{equation}
\label{eq:univR}
(\Delta' u)\CR_{M,M'}=\CR_{M,M'}(\Delta u)\in\End_\BF(M\otimes M')
\qquad(u\in U_\BF),
\end{equation}
where $\Delta':U_\BF\to U_\BF\otimes U_\BF$ is the opposite comultiplication given by $\Delta'=\tau\circ\Delta$ with
 $\tau(a\otimes b)=b\otimes a$ 
(see, for example, \cite[2.2]{T2}).
By \eqref{eq:univR}
we have
\begin{equation}
\label{eq:univR-braid}
(\Delta' T)\CR_{M,M'}=\CR_{M,M'}(\Delta T)\in\End_\BF(M\otimes M')
\qquad(T\in\BB),
\end{equation}
where $\Delta' T=\tau_{M,M'}^{-1}\circ\Delta T\circ\tau_{M,M'}$.

Let $\varphi\in A_\BF(\lambda)_\xi$
and $v\in M\in\Mod_{int}(U_\BF)$.
By the definition we see easily that
\begin{equation}
\CR_{A_\BF,M}^{-1}
(\varphi\otimes v)=\Omega_2'(\varphi)_M(v).
\end{equation}
Hence we have
\begin{align*}
&(T\star\Omega_2'(\varphi))_M(v)
=(\Delta T)\Omega_2'(\varphi)_MT^{-1}(v)
=(\Delta T)
\CR_{A_\BF,M}^{-1}(\varphi\otimes T^{-1}v)\\
=&
\CR_{A_\BF,M}^{-1}(\Delta' T)(\varphi\otimes T^{-1}v)
=
\CR_{A_\BF,M}^{-1}(\Delta' T)(T^{-1}\otimes T^{-1})(T\varphi\otimes v)
\end{align*}
for any $T\in\BB$.
For $T=T_i^{\pm1}$ we can write 
$(\Delta T)(T^{-1}\otimes T^{-1})=\sum_nc_n\otimes d_n$ for $c_n, d_n\in U_\BF$ (see Lemma \ref{lem:DeltaT}),
and hence
\begin{align*}
&(T\star\Omega_2'(\varphi))_M(v)
=
\CR_{A_\BF,M}^{-1}(\sum_nd_n\cdot T(\varphi)\otimes c_n v)
\\
=&
\sum_n
\Omega_2'(d_n\cdot T(\varphi))_M(c_nv)
=
\left(
\sum_n
\Omega_2'(d_n\cdot T(\varphi))c_n
\right)_M
(v).
\end{align*}
It follows that 
\begin{equation}
\label{eq:TOmega1}
T\star\Omega_2'(\varphi)
=\sum_n\Omega_2'(d_n\cdot T(\varphi))c_n
\qquad
(T=T_i^{\pm1}).
\end{equation}

For $M, M'\in\Mod_{int}(U_\BF)$ define 
$\CR'_{M, M'}\in\End_\BF(M\otimes M')$
by $\CR'_{M, M'}=\tau_{M, M'}^{-1}\circ\CR_{M', M}\circ \tau_{M, M'}$.
Then by a similar argument as above using
\begin{align}
\label{eq:univR-braid2}
&(\Delta T)\CR'_{M,M'}=\CR'_{M,M'}(\Delta' T)\in\End_\BF(M\otimes M')
\qquad(T\in\BB),\\
&(\CR'_{A_\BF,M})^{-1}
(\varphi\otimes v)=\Omega_1'(\varphi)_M(v)
\quad
(\varphi\in A_\BF(\lambda)_\xi,\,\,
v\in M)
\end{align}
we obtain
\begin{equation}
\label{eq:TOmega2}
T\star\Omega_1'(\varphi)
=\sum_n\Omega_1'(d_n\cdot T(\varphi))c_n
\qquad
(T=T_i^{\pm1}).
\end{equation}
By \eqref{eq:TOmega1}, \eqref{eq:TOmega2} we finally obtain 
\begin{equation}
\label{eq:TOmega}
T\star\Omega(\varphi)
=\sum_n\Omega(d_n\cdot T(\varphi))c_n
\qquad
(T=T_i^{\pm1},\,\,\varphi\in A_\BF).
\end{equation}
We are done.
\end{proof}

We see from  Proposition \ref{prop:braid-action-Dprime} that \eqref{eq:braid-action-E}
induces a group homomorphism
\begin{equation}
\label{eq:braid-action-Dprime}
\BB\to\Aut_{alg}(D'_\BF)
\qquad(T\mapsto[\Phi\mapsto T\star \Phi]).
\end{equation}
Note that \eqref{eq:braid-action-E} and \eqref{eq:braid-action-Dprime} are natural lifts of \eqref{eq:braid-D} by Lemma \ref{lem:braid-D}.

\subsection{}
Set
\[
D_\BA=
\langle
\ell_\varphi, r_\varphi, \deru_u,\sigma_\lambda\mid
\varphi\in A_\BA, u\in U_\BA, \lambda\in\Lambda\rangle_{\BA-{\rm{alg}}}\subset D_\BF.
\]
We have a canonical embedding
\[
D_\BA\to\End_\BA(A_\BA).
\]
Recall that we have fixed $\BF$-bases $\{x_p\}_p$ and $\{y_p\}_p$ of $U_\BF^+$ and $U_\BF^-$ respectively and $\beta_p\in Q^+$ for each $p$ satisfying \eqref{eq:xy1}, \eqref{eq:xy2}.
By Lemma \ref{lem:pm-duality} we can renormalize them in the following two manners;
\begin{itemize}
\item[(a)]
$\{x_p\}_p$ and $\{y_p\}_p$ are $\BA$-bases of  $U_\BA^{L,+}$ and $U_\BA^-$ respectively,
\item[(b)]
$\{x_p\}_p$ and $\{y_p\}_p$ are $\BA$-bases of  $U_\BA^{+}$ and $U_\BA^{L,-}$ respectively.
\end{itemize}
In particular, we have 
\[
D_\BA=
\langle
\ell_\varphi, \deru_u,\sigma_\lambda\mid
\varphi\in A_\BA, u\in U_\BA, \lambda\in\Lambda\rangle_{\BA-{\rm{alg}}}\subset D_\BF
\]
by Lemma \ref{lem:rvarphi}.
In the case (a) (resp. (b)) we write 
$\{x_p\}_p$ and $\{y_p\}_p$ as above as 
$\{x^L_p\}_p$ and $\{y_p\}_p$ 
(resp. $\{x_p\}_p$ and $\{y^L_p\}_p$).

Define an $\BA$-subalgebra $E_\BA$ of $E_\BF$ by
\[
E_\BA=A_\BA U_\BA \BA[\Lambda]\,
(\cong A_\BA\otimes_\BA U_\BA\otimes_\BA \BA[\Lambda]),
\]
and set
\[
D'_\BA
=\Image(E_\BA\to D'_\BF)\subset D'_\BF.
\]
For $\varphi\in A_\BA(\lambda)_\xi$ with $\lambda\in\Lambda^+$, $\xi\in\Lambda$ we have
\begin{align*}
\Omega_1(\varphi)&=
\sum_p({y^L_p}\cdot\varphi){x_pk_{-\xi}}e(\lambda)
\in E_\BA,\\
\Omega_2(\varphi)&=
\sum_p({(Sx^L_p)\cdot\varphi}){y_pk_{\beta_p}k_\xi}e({-\lambda})
\in E_\BA,
\end{align*}
and hence
$\Omega(\varphi)\in E_\BA$.
It follows that
\begin{equation*}
D'_\BA=
E_\BA/
\sum_{\varphi\in A_\BA}E_\BA\Omega(\varphi)E_\BA
=
E_\BA/
\sum_{\varphi\in A_\BA}A_\BA\Omega(\varphi)U_\BA \BA[\Lambda].
\end{equation*}
Note that
$E_\BA$, $D'_\BA$, $D_\BA$ are $\Lambda$-graded $\BA$-algebras
by
\[
E_\BA(\lambda)=E_\BF(\lambda)\cap E_\BA,\quad
D'_\BA(\lambda)=D'_\BF(\lambda)\cap D'_\BA,\quad
D_\BA(\lambda)=D_\BF(\lambda)\cap D_\BA\quad
\]
for $\lambda\in\Lambda^+$.
We have a sequence
\[
E_\BA\to D'_\BA\to D_\BA
\]
of surjective homomorphisms of $\Lambda$-graded $\BA$-algebras.

The braid group actions on $E_\BF$, $D'_\BF$, $D_\BF$ induce those on 
$E_\BA$, $D'_\BA$, $D_\BA$.
Namely we have group homomorphisms
\begin{equation}
\label{eq:braidA}
\BB\to\Aut_{alg}(E_\BA),\qquad
\BB\to\Aut_{alg}(D'_\BA),\qquad
\BB\to\Aut_{alg}(D_\BA)
\end{equation}
denoted by $T\mapsto[\Phi\mapsto T\star \Phi]$.

\subsection{}
We set
\[
E_\zeta=\BC\otimes_\BA E_\BA,\qquad
D'_\zeta=\BC\otimes_\BA D'_\BA,\qquad
D_\zeta=\BC\otimes_\BA D_\BA.
\]
Then we have 
\begin{align}
\label{eq:Ezeta1}
&E_\zeta=A_\zeta\otimes U_\zeta\otimes \BC[\Lambda]\\
\label{eq:Ezeta2}
&D'_\zeta=
E_\zeta/
\sum_{\varphi\in A_\zeta}E_\zeta\Omega(\varphi)E_\zeta
=
E_\zeta/
\sum_{\varphi\in A_\zeta}A_\zeta\Omega(\varphi)U_\zeta \BC[\Lambda].
\end{align}
We have a sequence
\begin{equation}
\label{eq:ED'D:zeta}
E_\zeta\to D'_\zeta\to D_\zeta
\end{equation}
of surjective homomorphisms of $\Lambda$-graded $\BC$-algebras.
By \eqref{eq:HC-center-F}, \eqref{eq:HC-center}, \eqref{eq:Har-center-in-D0}
we have
\begin{equation}
\label{eq:Har-center-in-D}
z\in Z_{Har}(U_\zeta),\,\,
\iota(z)=\sum_{\lambda\in\Lambda}a_\lambda e({2\lambda})\quad\Longrightarrow\quad
\deru_z
=\sum_{\lambda\in\Lambda}a_\lambda\sigma_{2\lambda}
\end{equation}
in $D_\zeta$.
\begin{remark}
{\rm
The natural algebra homomorphism
$D_\zeta\to\End_\BC(A_\zeta)$
is not injective.
}
\end{remark}

The actions of $\BB$ on $E_\BA$, $D'_\BA$, $D_\BA$  induce the group homomorphisms
\begin{equation}
\label{eq:braid-zeta}
\BB\to\Aut_{alg}(E_\zeta),\qquad
\BB\to\Aut_{alg}(D'_\zeta),\qquad
\BB\to\Aut_{alg}(D_\zeta)
\end{equation}
denoted by $T\mapsto[\Phi\mapsto T\star \Phi]$.

We have natural algebra homomorphisms
\begin{equation}
\label{eq:A1toED}
A_1\to E_\zeta,\qquad
A_1\to D'_\zeta,\qquad
A_1\to D_\zeta
\end{equation}
induced by \eqref{eq:ED'D:zeta} and the embedding 
\[
A_1\subset A_\zeta=A_\zeta\otimes 1\otimes 1\subset A_\zeta\otimes U_\zeta\otimes\BC[\Lambda]
=E_\zeta.
\]
\begin{lemma}
\label{lem:A1-cent}
The images of \eqref{eq:A1toED} are contained in the center.
\end{lemma}
\begin{proof}
It is sufficient to show the statement for $A_1\to E_\zeta$.
Since $A_1$ is a central subalgebra of $A_\zeta$, we have only to show that $A_1$ commutes with $U_\zeta$ and $\BC[\Lambda]$.
The commutativity with $\BC[\Lambda]$ is a consequence of \eqref{eq:A1Azeta2}.
It remains to show that $A_1$ commutes with $U_\zeta$ in the algebra $E_\zeta$.
By definition $A_1$ is a $U_\zeta^L$-submodule of $A_\zeta$ and the action of $U_\zeta^L$ on $A_1$ is induced from the natural $U(\Gg)$-module structure of $A_1$ through $\pi:U^L_\zeta\to U(\Gg)$.
Hence for $\varphi\in A_1$ and $u\in U_\zeta$ we have
\[
u\varphi=\sum_{(u)}((\pi\circ j)(u_{(0)})\cdot\varphi)u_{(1)}
=\sum_{(u)}\varepsilon(u_{(0)})\varphi u_{(1)}=\varphi u
\]
by \eqref{eq:pi-j}.
\end{proof}
\begin{lemma}
\label{lem:braid-action-on-D}
Identify $\Theta_x$ for $x\in W$ as a subset of $E_\zeta$ via \eqref{eq:A1toED}.
Then we have
$T_i^{-1}\star\Theta_w=\Theta_{ws_i}$
for $w\in W$ and $i\in I$ such that $ws_i>w$ with respect to the standard partial order. 
\end{lemma}
\begin{proof}
Note that  for any $\lambda\in\Lambda^+$ we have $A_1(\lambda)_{w^{-1}\lambda}=A_\zeta(\ell\lambda)_{\ell w^{-1}\lambda}$.
Hence it is sufficient to show $T_i^{-1}\star\varphi=T_i^{-1}(\varphi)$ for any $\lambda\in\Lambda^+$ and 
$\varphi\in A_\zeta(\ell\lambda)_{\ell w^{-1}\lambda}$.
By our assumption on $w$ and $i$ we have
$(s_iw^{-1}\lambda,\alpha_i^\vee)\leqq0$, and hence our assertion follows from Lemma \ref{lem:TstarFormula}.
\end{proof}

Let $w\in W$.
By Lemma \ref{lem:braid-action-on-D} 
we have $T_{w^{-1}}^{-1}\star\Theta_e=\Theta_w$.
Hence
the algebra automorphisms
\[
T_{w^{-1}}^{-1}\star(\bullet):E_\zeta\to E_\zeta,\qquad
T_{w^{-1}}^{-1}\star(\bullet):D'_\zeta\to D'_\zeta,\qquad
T_{w^{-1}}^{-1}\star(\bullet):D_\zeta\to D_\zeta
\]
induce isomorphisms
\[
\Theta_e^{-1}E_\zeta\to\Theta_w^{-1}E_\zeta,
\qquad
\Theta_e^{-1}D'_\zeta\to\Theta_w^{-1}D'_\zeta,
\qquad
\Theta_e^{-1}D_\zeta\to\Theta_w^{-1}D_\zeta
\]
of $\Lambda$-graded algebras.

For $w\in W$ set
\[
\tilde{\Theta}_w=\bigcup_{\lambda\in\Lambda^+}(A_\zeta(\lambda)_{w^{-1}\lambda}\setminus\{0\}).
\]
It is a multiplicative subset of $A_\zeta$.
Moreover, for any $s\in \tilde{\Theta}_w$ we have $s^{\ell}\in\Theta_w$.
Hence if we are given a ring homomorphism $A_\zeta\to R$ such that the image of $A_1$ is contained in the center of $R$, then the image of $\tilde{\Theta}_w$ in $R$ satisfies the left and right Ore conditions. 
Moreover, in this situation we have $\tilde{\Theta}_w^{-1}R\cong {\Theta}_w^{-1}R$.
In particular, we have
\[
\Theta_w^{-1}A_\zeta\cong\tilde{\Theta}_w^{-1}A_\zeta,
\qquad
\Theta_w^{-1}E_\zeta\cong\tilde{\Theta}_w^{-1}E_\zeta,
\qquad
\Theta_w^{-1}D'_\zeta\cong\tilde{\Theta}_w^{-1}D'_\zeta.
\]
\begin{proposition}
\label{prop:local-verma}
We have a natural $U_\zeta^L$-module structure of $\tilde{\Theta}_e^{-1}A_\zeta$ such that $A_\zeta\to\tilde{\Theta}_e^{-1}A_\zeta$ is a homomorphism of $U_\zeta^L$-modules and 
\[
u\cdot(\varphi\psi)=\sum_{(u)}(u_{(0)}\cdot\varphi)(u_{(1)}\cdot\psi)
\qquad
(u\in U_\zeta^L, \varphi, \psi\in \tilde{\Theta}_e^{-1}A_\zeta).
\]
Moreover, for any $\lambda\in\Lambda$ we have
\[
(\tilde{\Theta}_e^{-1}A_\zeta)(\lambda)\cong
M^*_{-,\zeta}(\lambda).
\]
as a $U^L_\zeta$-module.
\end{proposition}
\begin{proof} 
It is not difficult to deduce our assertion from
the corresponding fact over $\BF$, which is shown in
\cite[Proposition 4.3, Proposition 4.6]{T2}.
Details are omitted.
\end{proof}
Denote by
\[
\jmath:\tilde{\Theta}_e^{-1}E_\zeta\to\tilde{\Theta}_e^{-1}D'_\zeta
\]
the canonical algebra homomorphism.
\begin{proposition}
\label{prop:local-formula1}
Let $\lambda\in\Lambda^+$ and $\gamma\in Q^+$.
For $\varphi\in A_\zeta(\lambda)_{\lambda-\gamma}$  and $s\in A_\zeta(\lambda)_\lambda\setminus\{0\}$ we have
\[
\jmath(
\sum_{p}((Sx^L_p)\cdot(\varphi s^{-1}))
y_pk_{\beta_p}
)
=
\zeta^{(\lambda,\gamma)}
\jmath(
\sum_ps^{-1}(y^L_p\cdot\varphi)x_pk_{-2(\lambda-\gamma)}e(2\lambda)
)
\]
in $(\tilde{\Theta}_e^{-1}D'_\zeta)(0)$.
\end{proposition}
\begin{proof}
By Lemma \ref{lem:Omega} we have algebra anti-homomorphisms
\begin{equation}
\label{eq:Omega-bar}
\overline{\Omega}_i:A_\zeta\to D'_\zeta
\qquad(i=1, 2)
\end{equation}
as the composite of 
\[
A_\zeta \xrightarrow{\Omega_i} E_\zeta\to D'_\zeta.
\]
By the definition of $D_\zeta'$ we have $\overline{\Omega}_1=\overline{\Omega}_2$.
For $\psi\in A_\zeta(\lambda)_\lambda$ with $\lambda\in\Lambda^+$ we have 
$\overline{\Omega}_2(\psi)=
\overline{k_{\lambda}e(-\lambda)\psi}$
by definition.
Hence \eqref{eq:Omega-bar} induces an anti-homomorphism
\begin{equation}
\label{eq:Omega-bar2}
\tilde{\Theta}_e^{-1}\overline{\Omega}_1=
\tilde{\Theta}_e^{-1}\overline{\Omega}_2:\tilde{\Theta}_e^{-1}A_\zeta\to \tilde{\Theta}_e^{-1}D'_\zeta
\end{equation}
of $\Lambda$-graded algebras.

For $x\in U_\zeta^{L,+}$ we have
\begin{align*}
&\varepsilon(x)1=x\cdot 1=x\cdot(s^{-1}s)
=\sum_{(x)}(x_{(0)}\cdot s^{-1})(x_{(1)}\cdot s)\\
=&
\sum_{(x)}\varepsilon(x_{(1)})(x_{(0)}\cdot s^{-1})s
=(x\cdot s^{-1})s,
\end{align*}
and hence $x\cdot s^{-1}=\varepsilon(x)s^{-1}$.
Therefore, by
\[
\sum_{(x^L_p)}
x^L_{p(0)}\otimes x^L_{p(1)}\in k_{\beta_p}\otimes x_p^L+
U_\zeta^{L,0}\Ker(\varepsilon:U_\zeta^{L,+}\to\BC)\otimes U_\zeta^{L,+}
\]
we have
\begin{align*}
&(Sx^L_p)\cdot(\varphi s^{-1})=
\sum_{(x^L_p)}(S(x^L_{p(1)})\cdot \varphi)
(S(x^L_{p(0)})\cdot s^{-1})\\
=&((Sx^L_{p})\cdot \varphi)
(k_{-\beta_p}\cdot s^{-1})
=\zeta^{(\lambda,\beta_p)}
((Sx^L_{p})\cdot \varphi)s^{-1}.
\end{align*}
By $\overline{\Omega}_2(s)=\jmath(k_\lambda e(-\lambda)s)$ we have
$\jmath(s^{-1})=(\overline{\Omega}_2(s))^{-1}\jmath(k_\lambda e(-\lambda))$, and hence
\begin{align*}
\jmath
&((Sx^L_p)\cdot(\varphi s^{-1}))
=\zeta^{(\lambda,\beta_p)}
\jmath((Sx^L_{p})\cdot \varphi)
(\overline{\Omega}_2(s))^{-1}\jmath(k_\lambda e(-\lambda))\\
=&
\zeta^{(\lambda,\beta_p)}
(\overline{\Omega}_2(s))^{-1}
\jmath(((Sx^L_{p})\cdot \varphi) k_\lambda e(-\lambda))
\end{align*}
by Lemma \ref{lem:Omega}.
Therefore, we have
\begin{align*}
&\jmath(
\sum_{p}((Sx^L_p)\cdot(\varphi s^{-1}))
y_pk_{\beta_p}
)\\
=&
(\overline{\Omega}_2(s))^{-1}
\jmath(
\sum_{p}
\zeta^{(\lambda,\beta_p)}
((Sx^L_{p})\cdot \varphi) k_\lambda e(-\lambda)
y_pk_{\beta_p})\\
=&
(\overline{\Omega}_2(s))^{-1}
\jmath(
\sum_{p}
((Sx^L_{p})\cdot \varphi)y_p k_{\lambda+\beta_p} e(-\lambda))\\
=&
(\overline{\Omega}_2(s))^{-1}
(\overline{\Omega}_2(\varphi))
\jmath(k_\gamma)
\\
=&(\tilde{\Theta}_e^{-1}\overline{\Omega}_2)(\varphi s^{-1})\jmath(k_\gamma).
\end{align*}

On the other hand
we have
\begin{align*}
&(\tilde{\Theta}_e^{-1}\overline{\Omega}_1)(\varphi s^{-1})
=
\overline{\Omega}_2(s)^{-1}
\overline{\Omega}_1(\varphi)
\\
=&
\jmath(
\sum_p
s^{-1}e(\lambda)k_{-\lambda}(y^L_p\cdot\varphi)x_pk_{-(\lambda-\gamma)}e(\lambda))
\\
=&
\zeta^{(\lambda,\gamma)}
\jmath(
\sum_p
s^{-1}(y^L_p\cdot\varphi)x_pk_{-(2\lambda-\gamma)}e(2\lambda)).
\end{align*}

We obtain the desired result by $\tilde{\Theta}_e^{-1}\overline{\Omega}_1=\tilde{\Theta}_e^{-1}\overline{\Omega}_2$
\end{proof}
Considering the case $\varphi=s\in A_\zeta(\lambda)_\lambda\setminus\{0\}$ in Proposition \ref{prop:local-formula1} we obtain the following.
\begin{proposition}
\label{prop:local-formula2}
Let $\lambda\in\Lambda^+$.
For $s\in A_\zeta(\lambda)_\lambda\setminus\{0\}$ we have
\[
\jmath(
k_{2\lambda})
=
\jmath(
\sum_ps^{-1}(y^L_p\cdot s)x_pe(2\lambda))
\]
in $(\Theta_e^{-1}D'_\zeta)(0)$.
\end{proposition}
\subsection{}
The natural embedding $A_\zeta\to E_\zeta$ induces
homomorphisms
\[
A_\zeta\to D'_\zeta\quad
(\varphi\mapsto \overline{\varphi}),
\qquad
A_\zeta\to D_\zeta\quad
(\varphi\mapsto \ell_\varphi)
\]
of graded $\BC$-algebras.
We define abelian categories $\Mod(\DD'_{\CB_\zeta})$, $\Mod(\DD_{\CB_\zeta})$ by
\[
\Mod(\DD'_{\CB_\zeta})=
\CC(A_\zeta,D'_\zeta),\qquad
\Mod(\DD_{\CB_\zeta})=
\CC(A_\zeta,D_\zeta).
\]
We also define $\CO_\CB$-algebras $Fr_*\DD'_{\CB_\zeta}$, $Fr_*\DD_{\CB_\zeta}$ by
\[
Fr_*\DD'_{\CB_\zeta}=\omega_\CB^*{D'_\zeta}{}^{(\ell)},
\qquad 
Fr_*\DD_{\CB_\zeta}=\omega_\CB^*D_\zeta^{(\ell)}.
\]
By Lemma \ref{lem:Fr1} and Lemma \ref{lem:Fr2}
we have 
equivalences
\begin{align}
\label{eq:equiv-D1}
&Fr_*:\Mod(\DD'_{\CB_\zeta})\to
\Mod(Fr_*\DD'_{\CB_\zeta}),\\
\label{eq:equiv-D2}
&Fr_*:
\Mod(\DD_{\CB_\zeta})\to
\Mod(Fr_*\DD_{\CB_\zeta})
\end{align}
of abelian categories, 
where $\Mod(Fr_*\DD'_{\CB_\zeta})$ (resp.\
$\Mod(Fr_*\DD_{\CB_\zeta})$) denotes the category of quasi-coherent $Fr_*\DD'_{\CB_\zeta}$-modules (resp.\ quasi-coherent $Fr_*\DD_{\CB_\zeta}$-modules).

For $t\in H$ we define an abelian category $\Mod(\DD_{\CB_\zeta,t})$ by
\[
\Mod(\DD_{\CB_\zeta,t})=
\Mod_{\Lambda, t}(D_\zeta)/(\Mod_{\Lambda,t}(D_\zeta)\cap\Tor_{\Lambda^+}(A_\zeta,D_\zeta)),
\]
where $\Mod_{\Lambda,t}(D_\zeta)$ is the full subcategory of 
$\Mod_{\Lambda}(D_\zeta)$ consisting of $M\in\Mod_{\Lambda}(D_\zeta)$ so that $\sigma_\lambda|_{M(\mu)}=\theta_\lambda(t) \zeta^{(\lambda,\mu)}\id$ for any $\lambda, \mu\in\Lambda$.
Then we can regard $\Mod(\DD_{\CB_\zeta,t})$ as a full subcategory of $\Mod(\DD_{\CB_\zeta})$
(see \cite[Lemma 4.6]{T2}).
Set
\[
Fr_*\DD_{\CB_\zeta,t}=Fr_*\DD_{\CB_\zeta}\otimes_{\BC[\Lambda]}\BC_t,
\]
where $\BC_t$ denotes the one-dimensional $\BC[\Lambda]$-module given by $e(\lambda)\mapsto\theta_\lambda(t)$ for $\lambda\in\Lambda$.
The equivalence \eqref{eq:equiv-D2} induces the equivalence
\begin{equation}
\label{eq:equiv-D3}
Fr_*:
\Mod(\DD_{\CB_\zeta,t})\to
\Mod(Fr_*\DD_{\CB_\zeta,t}),
\end{equation}
where $\Mod(Fr_*\DD_{\CB_\zeta,t})$ denotes the category of quasi-coherent $Fr_*\DD_{\CB_\zeta,t}$-modules.

\section{Center}
\subsection{}
Note
\[
E_\zeta^{(\ell)}=A_\zeta^{(\ell)}\otimes_\BC U_\zeta\otimes_\BC \BC[\Lambda].
\]
We set 
\begin{equation}
ZE_\zeta^{(\ell)}=
A_1\otimes_\BC Z_{Fr}(U_\zeta)\otimes_\BC \BC[\Lambda]
\subset E_\zeta^{(\ell)}.
\end{equation}
\begin{lemma}
\label{lem:center}
$ZE_\zeta^{(\ell)}$ is a central subalgebra of $E_\zeta^{(\ell)}$.
\end{lemma}
\begin{proof}
It is easily seen that $\BC[\Lambda]$ is contained in the center of $E_\zeta^{(\ell)}$.
Moreover, we have already shown $[A_1,U_\zeta]=\{0\}$ in Lemma \ref{lem:A1-cent}.
Hence it is sufficient to show $[A_\zeta^{(\ell)},Z_{Fr}(U_\zeta)]=\{0\}$.
For that we have only to show $j(z)\cdot\varphi=\varepsilon(z)\varphi$ for $z\in Z_{Fr}(U_\zeta)$, $\varphi\in A_\zeta$, where 
$j:U_\zeta\to U_\zeta^L$ is the homomorphism induced by the inclusion $U_\BA\subset U^L_\BA$.
Since $j$ preserves the action of the braid group $\BB$, 
it is a consequence of the fact that 
$Z_{Fr}(U_\zeta)$ is generated by $e_i^\ell$, $f_i^\ell$, $k_{\ell\lambda}$ ($i\in I$, $\lambda\in\Lambda$)  and their $\BB$-conjugates.
\end{proof}
We denote by $ZD'_\zeta{}^{(\ell)}$, $ZD_\zeta^{(\ell)}$ the images of  $ZE_\zeta^{(\ell)}$ in $D'_\zeta{}^{(\ell)}$, $D_\zeta^{(\ell)}$ respectively.
We have 
\begin{align*}
\omega_\CB^*ZE_\zeta^{(\ell)}&=
\CO_\CB\otimes_\BC Z_{Fr}(U_\zeta)\otimes_\BC \BC[\Lambda]\\
&\cong
\CO_\CB\otimes_\BC \BC[K]\otimes_\BC \BC[H]\\
&\cong p_*\CO_{\CB\times K\times H},
\end{align*}
where $p:\CB\times K\times H\to\CB$ is the projection.
Note that we identify $Z_{Fr}(U_\zeta)$ and $\BC[\Lambda]$ with $\BC[K]$ and $\BC[H]$ respectively (see \eqref{eq:Fr-center}).
Set
\[
\CZ'_\zeta=\omega_\CB^*ZD'_\zeta{}^{(\ell)},
\qquad
\CZ_\zeta=\omega_\CB^*{ZD}_\zeta^{(\ell)}.
\]
Then $\CZ'_\zeta$ and $\CZ_\zeta$ are central $\CO_\CB$-subalgebras of $Fr_*\DD'_{\CB_\zeta}$ and $Fr_*\DD_{\CB_\zeta}$ respectively.
Moreover, 
we have a sequence 
\[
p_*\CO_{\CB\times K\times H}
\to
\CZ'_\zeta
\to
\CZ_\zeta
\]
of surjective $\CO_\CB$-algebra homomorphisms.

We define a subvariety $\CV$ of $\CB\times K\times H$ by
\[
\CV=
\{
(B^-g,k,t)\in\CB\times K\times H\mid
g\kappa(k)g^{-1}\in t^{2\ell}N^-\}.
\]
We denote by 
\[
p_\CV:\CV\to\CB
\]
the projection.
The aim of this section is to prove the following.
\begin{theorem}
\label{thm:center}
We have
\[
\CZ'_\zeta\cong
\CZ_\zeta
\cong
p_{\CV*}\CO_\CV.
\]
\end{theorem}

\subsection{}
Set 
\begin{align*}
\tilde{\CB}=&N^-\backslash G,\\
\tilde\CV=&
\{
(N^-g,k,t)\in\tilde\CB\times K\times H\mid
g\kappa(k)g^{-1}\in t^{2\ell}N^-\}.
\end{align*}
\begin{lemma}
\label{lem:tilde-V}
$\tilde{\CV}$ is a connected smooth variety with $\dim\tilde{\CV}=2\dim\tilde{\CB}$.
\end{lemma}
\begin{proof}
Set
\begin{align*}
\CW&=
\{
(N^-g,x,t)\in\tilde\CB\times G\times H\mid
gxg^{-1}\in t^{2\ell}N^-\},\\
\CW^0&=\{
(N^-g,x,t)\in\CW\mid
x\in N^+HN^-\}.
\end{align*}
Then $\CW$ is a fiber bundle over $\tilde{\CB}$ whose fiber at the origin $N^-\in\tilde{\CB}$ is isomorphic to $H\times N^-$.
Hence $\CW$ is a smooth connected variety with with $\dim{\CW}=2\dim\tilde{\CB}$.
It follows that its Zariski open subvariety $\CW^0$ is also a smooth connected variety of the same dimension.
Note that $\kappa:K\to G$ is a composite of the Galois covering $\overline{\kappa}:K\to N^+HN^-$ 
with Galois group $\Gamma=\{\gamma\in H\mid\gamma^2=1\}$ 
and the open embedding $N^+HN^-\to G$.
By $\tilde{\CV}\cong\CW^0\times_{N^+HN^-}K$ we see that $\tilde{\CV}$ is a Galois covering of $\CW^0$ with Galois group $\Gamma$.
Hence $\tilde{\CV}$ is a smooth variety.
It remains to show that $\tilde{\CV}$ is connected.
Since $\Gamma$ acts transitively on the set $\CX$ of connected components of $\tilde{\CV}$, it is sufficient to show that $\gamma(X)=X$ for any $\gamma\in\Gamma$ and $X\in\CX$.
Since $G$ was chosen to be simply-connected, the group $\Gamma$ is generated by elements $\gamma_i\in\Gamma\,\,(i\in I)$ with $\theta_{\varpi_j}(\gamma_i)=1$
 for $i\ne j$ and $\theta_{\varpi_i}(\gamma_i)=-1$.
Hence it is sufficient to show that $\gamma_i(X)=X$ for any $i\in I$ and $X\in\CX$.
By the commutativity of $\Gamma$ we have only to show that for any $i\in I$ there exists some $X\in\CX$ such that $\gamma_i(X)=X$.
Hence the proof is reduced to showing that for any $i\in I$ there exists some $v\in\tilde{\CV}$ such that $v$ and $\gamma_i v$ are contained in the same connected component of $\tilde{\CV}$.
We may assume that $G=SL_2(\BC)$.
Then we can check the assertion by a direct computation.
Details are omitted.
 \end{proof}

We regard $A_1$ as the ring of functions on the quasi-affine variety $\tilde{\CB}$.
We also regard 
$ZE_\zeta^{(\ell)}$ as the ring of functions on $\tilde{\CB}\times K\times H$.
We denote by $\tilde{\CZ}'_\zeta$, $\tilde{\CZ}_\zeta$ the sheaf of $\CO_{\tilde{\CB}\times K\times H}$-algebras corresponding to 
$ZD'_\zeta{}^{(\ell)}$, ${ZD}_\zeta^{(\ell)}$ respectively.
In order to prove Theorem \ref{thm:center} it is sufficient to show 
\[
\tilde{\CZ}'_\zeta\cong\tilde{\CZ}_\zeta\cong
\CO_{\tilde{\CV}},
\]
where $\tilde{\CV}$ is regarded as a reduced scheme.
\begin{lemma}
\label{lem:defining-V}
For $\varphi\in A_1(\lambda)\subset A_\zeta(\ell\lambda)$ with $\lambda\in\Lambda^+$ we have
\[
\Omega_i(\varphi)\in ZE_\zeta^{(\ell)}\qquad(i=1, 2),
\]
and we have
\begin{align}
\label{eq:lem:defining-V1}
\Omega_1(\varphi)(N^-g,(n_1h,n_2h^{-1}),t)&=
\varphi(N^-t^\ell gn_2h^{-1}),
\\
\label{eq:lem:defining-V2}
\Omega_2(\varphi)(N^-g,(n_1h,n_2h^{-1}),t)&=
\varphi(N^-t^{-\ell} gn_1h)
\end{align}
for $g\in G$, $n_1\in N^+$, $n_2\in N^-$, $t, h\in H$.
\end{lemma}
\begin{proof}
We may assume that $\varphi\in A_1(\lambda)_\xi\subset
A_\zeta(\ell\lambda)_{\ell\xi}$.

Take bases $\{f_r\}_r$, $\{v_r\}_r$ of $\BC[N^-]$ and $U(\Gn^-)$ respectively such that $\langle f_r,v_{r'}\rangle=\delta_{r, r'}$, where $\langle\,\,,\,\,\rangle:\BC[N^-]\times U(\Gn^-)\to\BC$ is the canonical Hopf pairing.
Define $\{x_r\}_r\subset U_\zeta^+\cap Z_{Fr}(U_\zeta)$
by $\tau^L(x_r,y)=\langle f_r,\pi(y)\rangle$ for any $y\in U_\zeta^{L,-}$.
We can take 
$\{y_r\}_r\subset U^{L,-}_\zeta$, 
$\{y'_s\}_s\subset \Ker(\pi|U^{L,-}_\zeta)$, 
$\{x'_s\}_s\subset U^{+}_\zeta$, 
such that $\pi(y_r)=v_r$ for any $r$,
$\{y_r\}_r\sqcup\{y'_s\}_s$ is a basis of  $U_\zeta^{L,-}$,
$\{x_r\}_r\sqcup\{x'_s\}_s$ is a basis of  $U_\zeta^{+}$,
$\tau^L(x'_s,y_r)=0$ for any $r, s$, and $\tau^L(x'_s,y'_{s'})=\delta_{s,s'}$.
Then we have
\begin{align*}
\Omega_1(\varphi)&=\sum_r(\pi(y_r)\cdot\varphi)x_rk_{-\ell\xi}e(\ell\lambda)
+
\sum_s(\pi(y'_s)\cdot\varphi)x'_sk_{-\ell\xi}e(\ell\lambda)\\
&=
\sum_r(v_r\cdot\varphi)x_rk_{-\ell\xi}e(\ell\lambda)\in ZE_\zeta^{(\ell)}.
\end{align*}
By Lemma \ref{lem:unipotent} below and the description of the isomorphism $\BC[K]\cong Z_{Fr}(U_\zeta)$ given in Section \ref{sec:QE} we obtain
\begin{align*}
&\Omega_1(\varphi)(N^-g,(n_1h,n_2h^{-1}),t)
=\sum_r
((v_r\cdot\varphi)(N^-g))
f_r(n_2)
\theta_{-\xi}(h)
\theta_{\ell\lambda}(t)\\
=&((n_2\cdot\varphi)(N^-g))\theta_{\xi}(h^{-1})
\theta_{\lambda}(t^\ell)
=(n_2h^{-1}\cdot\varphi\cdot t^{\ell})(N^-g)\\
=&\varphi(N^-t^{\ell}gn_2h^{-1}).
\end{align*}
\eqref{eq:lem:defining-V1} is proved.
The proof of \eqref{eq:lem:defining-V2} is similar and omitted.
\end{proof}
The proof of the following result is standard and left to the readers.
\begin{lemma}
\label{lem:unipotent}
Let $N$ be a unipotent algebraic group over $\BC$ with Lie algebra $\Gn$.
Denote by $\sum_rv_r\otimes f_r$ the canonical element of (a completion of) $\BC[N]\otimes U(\Gn)$ with respect to the canonical Hopf pairing $\BC[N]\times U(\Gn)\to\BC$.
Then for any finite dimensional rational $N$-module $M$ we have
\[
gm=\sum_rf_r(g)v_rm\qquad(g\in N, m\in M).
\]
\end{lemma}
We define $F_i:\tilde{\CB}\times K\times H\to\tilde{\CB}\,\,(i=1,2)$ by
\begin{align*}
&F_1(N^-g, (n_1h, n_2h^{-1}),t)=N^-t^{\ell} gn_2h^{-1},
\\
&F_2(N^-g, (n_1h, n_2h^{-1}),t)=N^-t^{-\ell} gn_1h
\end{align*}
for $g\in G, n_1\in N^+, n_2\in N^-, h, t\in H$.
\begin{lemma}
\label{lem:defining-V2}
The equations
\[
\varphi\circ F_1(x)=\varphi\circ F_2(x)
\qquad(\varphi\in A_1)
\]
for $x\in \tilde{\CB}\times K\times H$ give defining equations of $\tilde{\CV}$, which is reduced at any point of $\tilde{\CV}$.
\end{lemma}
\begin{proof}
Since $\tilde{\CB}$ is quasi-affine, the equations
\[
\varphi(x)=\varphi(y)
\qquad(\varphi\in A_1)
\]
for $(x,y)\in \tilde{\CB}\times \tilde{\CB}$ give defining equations of the diagonal subvariety
\[
(\tilde{\CB}\times \tilde{\CB})_{\rm diag}
=\{(x,x)\mid x\in\tilde{\CB}\}
\]
of $\tilde{\CB}\times \tilde{\CB}$.
Hence by the cartesian diagram
\[
\begin{CD}
\tilde{\CV}@>>>(\tilde{\CB}\times \tilde{\CB})_{\rm diag}\\
@VVV@VVV\\
\tilde{\CB}\times K\times H
@>>{(F_1,F_2)}>
\tilde{\CB}\times \tilde{\CB}
\end{CD}
\]
it is sufficient to show that $(F_1,F_2):
\tilde{\CB}\times K\times H
\to
\tilde{\CB}\times \tilde{\CB}$
is a smooth morphism.
Define 
$\alpha:\tilde{\CB}\times K\times H\to \tilde{\CB}\times K\times H$, 
$\beta:\tilde{\CB}\times K\times H\to \tilde{\CB}\times G\times H$ and 
$\gamma:\tilde{\CB}\times G\times H\to \tilde{\CB}\times \tilde{\CB}$ by
\begin{align*}
&\alpha(N^-g, (n_1h,n_2h^{-1}),t)
=(N^-t^{\ell}gn_2h^{-1},(n_1h,n_2h^{-1}),t),\\
&\beta(N^-g, (n_1h,n_2h^{-1}),t)
=(N^-g,hn_2^{-1}n_1h,t^2),\\
&\gamma(N^-g, x,t)
=(N^-g,N^-t^{-\ell}gx)
\end{align*}
for $g, x\in G$, $n_1\in N^+$, $n_2\in N^-$, $h, t\in H$.
Let us show that $\beta$ is smooth.
For that it is sufficient to show that $N^+\times N^-\times H
\ni(n_1,n_2,h)\mapsto hn_2^{-1}n_1h\in G$ is smooth.
This morphism is a composite of an isomorphism 
$N^+\times N^-\times H
\ni(n_1,n_2,h)
\mapsto
(h^{-1}n_1h,hn_2h^{-1},h)\in
N^+\times N^-\times H$ and
a smooth morphism
$N^+\times N^-\times H
\ni(n_1,n_2,h)
\mapsto
n_2^{-1}h^2n_1\in
G$.
Hence $\beta$ is smooth.
Then by the cartesian diagram
\[
\begin{CD}
\tilde{\CB}\times K\times H@>
{(F_1,F_2)}>>\tilde{\CB}\times \tilde{\CB}\\
@V{\alpha}VV@VV{\id}V\\
\tilde{\CB}\times K\times H
@>>{\gamma\circ\beta}>
\tilde{\CB}\times \tilde{\CB}
\end{CD}
\]
and the smoothness of $\beta$ it is sufficient to show that $\gamma$ is smooth.
Since the group $\tilde{G}=G\times H$ acts on $\tilde{\CB}$ from the right by $N^-g\cdot(x,t)=N^-t^{-\ell}gx$, we can identify $\tilde{\CB}$ with $\tilde{N}^-\backslash\tilde{G}$, where $\tilde{N}^-=\{(x,t)\in G\times H\mid t^{-\ell}x\in N^-\}$.
Under this identification $\gamma:(\tilde{N}^-\backslash\tilde{G})\times\tilde{G}\to
(\tilde{N}^-\backslash\tilde{G})\times(\tilde{N}^-\backslash\tilde{G})$ is given by
$\gamma(\tilde{N}^-\tilde{x},\tilde{g})=(\tilde{N}^-\tilde{x},\tilde{N}^-\tilde{x}\tilde{g})$, and hence the assertion is clear.
\end{proof}
By Lemma \ref{lem:defining-V}, Lemma \ref{lem:defining-V2} we have a sequence
\[
\CO_{\tilde\CV}\to\tilde{\CZ}'_\zeta\to\tilde{\CZ}_\zeta
\]
of surjective homomorphisms of $\CO_{\tilde{\CB}\times K\times H}$-algebras.
Hence in order to prove Theorem \ref{thm:center} it is sufficient to show that $\CO_{\tilde\CV}\to\tilde{\CZ}_\zeta$ is an isomorphism.

\subsection{}
By De Concini-Procesi \cite[11.7]{DP} the center $Z$ of $E_\zeta^{(\ell)}$ is endowed with a Poisson algebra structure by
\begin{equation}
\label{eq:Poisson:E}
\{\overline{a},\overline{b}\}
=\overline{[a,b]/\ell(q^{\ell}-q^{-\ell})}
\qquad(a, b\in E_\BA^{(\ell)}, \overline{a}, \overline{b}\in Z).
\end{equation}
Here 
\[
E_\BA^{(\ell)}=\bigoplus_{\lambda\in\Lambda^+}
A_\BA(\ell\lambda)\otimes U_\BA\otimes\BA[\Lambda]
\subset E_\BA.
\]
Note that 
$ZE_\zeta^{(\ell)}$ is a subalgebra of $Z$.
We will show that $ZE_\zeta^{(\ell)}$ is a Poisson subalgebra of $Z$.

We endow $C_\BA\otimes_\BA  U_\BA$ with an $\BA$-algebra structure such that $C_\BA\otimes 1$, $1\otimes U_\BA$ are subalgebras naturally isomorphic to $C_\BA$, $U_\BA$ respectively, and 
\[
(1\otimes u)(\varphi\otimes 1)=
\sum_{(u)}u_{(0)}\cdot\varphi\otimes u_{(1)}\qquad
(u\in U_\BA, \varphi\in C_\BA).
\]
Then the center $Z(C_\zeta\otimes U_\zeta)$ of $C_\zeta\otimes U_\zeta$ is endowed with a Poisson algebra structure by
\begin{equation}
\label{eq:Poisson:CU}
\{\overline{a},\overline{b}\}
=\overline{[a,b]/\ell(q^{\ell}-q^{-\ell})}
\quad(a, b\in C_\BA\otimes_\BA  U_\BA, \overline{a}, \overline{b}\in Z(C_\zeta\otimes U_\zeta)).
\end{equation}
Similarly to Lemma \ref{lem:center}
we see that $\BC[G]\otimes Z_{Fr}(U_\zeta)$ is a subalgebra of $Z(C_\zeta\otimes U_\zeta)$.
We first give a description of the Poisson bracket of $\BC[G]\otimes Z_{Fr}(U_\zeta)$.

Denote by $\Gk$  the Lie algebra of $K$.
We identify $K$ with a subgroup of $G\times G$ and regard $\Gk$ as a subalgebra of $\Gg\oplus \Gg$.
Namely,
\[
\Gk=\{(h+a,-h+b)\mid h\in\Gh, a\in\Gn^+, b\in\Gn^-\}.
\]
Denote by $S$ the diagonal subgroup $\{(g,g)\mid g\in G\}$ of $G\times G$.
Then its Lie algebra $\Gs$ is given by 
\[
\Gs=\{(x,x)\mid x\in \Gg\}
\subset\Gg\oplus\Gg.
\]
In particular, we have $\Gg\oplus\Gg=\Gk\oplus\Gs$.
We sometimes identify $S$ with $G$ by 
$(g,g)\leftrightarrow g$.
Define a symmetric bilinear form 
$\tilde{\epsilon}$ on  $\Gg\oplus\Gg$ by
\[
\tilde{\epsilon}((x_1,x_2),(y_1,y_2))
=\epsilon(x_1,y_1)-\epsilon(x_2,y_2)
\qquad(x_1, x_2, y_1, y_2\in\Gg),
\]
where $\epsilon:\Gg\times\Gg\to\BC$ is the invariant symmetric bilinear form on $\Gg$ whose restriction to $\Gh\times\Gh$ induces the bilinear form \eqref{eq:killing} on $\Gh^*$.
Then $\tilde{\epsilon}|\Gk\times\Gk$ and $\tilde{\epsilon}|\Gs\times\Gs$ are identically zero, and $\tilde{\epsilon}|\Gk\times\Gs$ is non-degenerate.
In other words $(\Gg\oplus\Gg,\Gk,\Gs)$ is a Manin triple with respect to $\tilde{\epsilon}$. 
In particular, $\BC[K]$ and $\BC[S]$ are Poisson Hopf algebras
(see Drinfeld \cite{D}, De Concini-Procesi \cite{DP}).
We will sometimes identify $\Gg^*$ and $\Gk^*$ with $\Gk$ and $\Gg$ respectively via the non-degenerate pairing $\tilde{\epsilon}|\Gk\times\Gs$ and the identification $\Gg\cong\Gs$.

In general let $A$ be an algebraic group over $\BC$ with Lie algebra $\Ga$.
For $a\in\Ga$ we define vector fields $L_a$, $R_a$ on $A$ by
\begin{align*}
&(L_af)(g)=\frac{d}{dt}f(g\exp(ta))|_{t=0}\qquad
(g\in A, f\in\CO_{A,g}),\\
&(R_af)(g)=\frac{d}{dt}f(\exp(-ta)g)|_{t=0}\qquad
(g\in A, f\in\CO_{A,g}).
\end{align*}
For $\xi\in\Ga^*$ we define 1-forms 
$L^*_\xi$, $R^*_\xi$ on $A$ by
\[
\langle L_a, L^*_\xi\rangle=
\langle R_a, R^*_\xi\rangle=\langle a,\xi\rangle\qquad
(a\in\Ga).
\]
By De Concini-Lyubashenko \cite{DL}, De Concini-Procesi \cite{DP},  Gavarini \cite{Gav}, and Tanisaki \cite{T3} we have the following.
\begin{proposition}
\label{prop:Poisson1}
$\BC[G]\otimes Z_{Fr}(U_\zeta)$ is closed under the Poisson bracket \eqref{eq:Poisson:CU}.
More precisely we have the following.
\begin{itemize}
\item[\rm(i)]
$\BC[G]$ is closed under the Poisson bracket \eqref{eq:Poisson:CU}.
Moreover, the isomorphism $\BC[G]\cong\BC[S]$ induced by the identification $G\cong S$ preserves the Poisson structures,
where $\BC[S]$ is regarded as a Poisson algebra via the Manin triple $(\Gg\oplus\Gg,\Gk,\Gs)$.
Namely, the Poisson tensor 
$\delta\in\bigwedge^2TG$ 
is given by
\begin{align*}
&\delta_g((L^*_\eta)_g,(L^*_{\eta'})_g)
=\tilde{\epsilon}(p_\Gs(\widetilde{\Ad}(g)(\eta)),\widetilde{\Ad}(g)(\eta')),\\
&\delta_g((R^*_\eta)_g,(R^*_{\eta'})_g)
=-\tilde{\epsilon}(p_\Gs(\widetilde{\Ad}(g^{-1})(\eta)),\widetilde{\Ad}(g^{-1})(\eta'))
\end{align*}
for $g\in G$, $\eta, \eta'\in\Gk\cong\Gg^*$.
Here, $\widetilde{\Ad}:G\to GL(\Gg\oplus\Gg)$ is the restriction of the adjoint action of $G\times G$ on $\Gg\oplus\Gg$ to $G\cong S$, and $p_\Gs:\Gg\oplus\Gg\to\Gs$ is the projection with respect to the direct sum decomposition $\Gg\oplus\Gg=\Gk\oplus\Gs$.
\item[\rm(ii)]
$Z_{Fr}(U_\zeta)$ is closed under the Poisson bracket \eqref{eq:Poisson:CU}.
Moreover, the isomorphism $Z_{Fr}(U_\zeta)\cong\BC[K]$ $($see \eqref{eq:Fr-center}$)$ preserves the Poisson structures,
where $\BC[K]$ is regarded as a Poisson algebra via the Manin triple $(\Gg\oplus\Gg,\Gk,\Gs)$.
Namely, the Poisson tensor 
$\delta\in\bigwedge^2TK$ 
is given by
\begin{align*}
&\delta_k((L^*_\xi)_k,(L^*_{\xi'})_k)
=\tilde{\epsilon}(p_\Gk(\widetilde{\Ad}(k)(\xi)),\widetilde{\Ad}(k)(\xi')),\\
&\delta_k((R^*_\xi)_k,(R^*_{\xi'})_k)
=-\tilde{\epsilon}(p_\Gk(\widetilde{\Ad}(k^{-1})(\xi)),\widetilde{\Ad}(k^{-1})(\xi'))
\end{align*}
for $k\in K$, $\xi, \xi'\in\Gg\cong\Gk^*$.
Here, $\widetilde{\Ad}:K\to GL(\Gg\oplus\Gg)$ is the restriction of the adjoint action of $G\times G$ on $\Gg\oplus\Gg$ to $K$, and $p_\Gk:\Gg\oplus\Gg\to\Gk$ is the projection with respect to the direct sum decomposition $\Gg\oplus\Gg=\Gk\oplus\Gs$. 
\item[\rm(iii)]
The restriction of the Poisson tensor $\delta\in\bigwedge^2T(G\times K)$ with respect to \eqref{eq:Poisson:CU} to $TG\otimes TK$ is given by 
\[
\delta_{(g,k)}((L^*_\eta)_g,(R^*_\xi)_k)=\tilde{\epsilon}(\xi,\eta)
\]
for
$g\in G, k\in K, \eta\in\Gk\cong\Gg^*, \xi\in\Gg\cong\Gk^*$.
\end{itemize}
\end{proposition}
\begin{lemma}
\label{lem:Poisson3}
The Poisson structure of $\BC[G]$ induces a Poisson structure of $\BC[N^-\backslash G]$.
\end{lemma}
\begin{proof}
It is sufficient to show that $A_1$ is closed under the Poisson bracket of $\BC[G]$ given by
\[
\{\overline{a},\overline{b}\}
=\overline{[a,b]/\ell(q^{\ell}-q^{-\ell})}
\qquad(a, b\in C_\BA, \overline{a},\overline{b}\in\BC[G]).
\]
Let $\varphi, \psi\in A_1$.
Then we have 
$\{\varphi,\psi\}=\overline{[a,b]/\ell(q^{\ell}-q^{-\ell})}$
for $a, b\in C_\BA$ such that  $\overline{a}=\varphi$ and $\overline{b}=\psi$.
We may assume $a, b\in A_\BA$.
Then we have $\{\varphi,\psi\}\in\BC[G]\cap A_\zeta=A_1$.
\end{proof}
For $a\in\Gg$ we denote by $\overline{L}_a$ the vector field on $N^-\backslash G$ induced by $L_a$.
Namely, 
\[
(\overline{L}_af)(N^-g)
=
\frac{d}{dt}f(N^-g\exp(ta))|_{t=0}\qquad
(g\in G, f\in\CO_{N^-\backslash G,N^-g}).
\]
Then we have a surjective linear map
\[
\Gg\to T(N^-\backslash G)_{N^-g}
\qquad(a\mapsto(\overline{L}_a)_{N^-g})
\]
with kernel $\Ad(g^{-1})(\Gn^-)$.
Hence we have
\begin{align*}
&T^*(N^-\backslash G)_{N^-g}\\
\cong&\{\eta\in\Gg^*\mid\langle\Ad(g^{-1})(\Gn^-),\eta
\rangle=\{0\}\}\\
\cong&\{\eta\in\Gk\mid
\tilde{\epsilon}(
\Ad(g^{-1})(\Gn^-),\eta
)=\{0\}\}\\
=&\{(y_1,y_2)\in\Gk\mid
\epsilon(
\Ad(g^{-1})(\Gn^-),y_1-y_2
)=\{0\}\}\\
=&\{(y_1,y_2)\in\Gk\mid
y_1-y_2\in\Ad(g^{-1})(\Gb^-)
\}.
\end{align*}
For $N^-g\in N^-\backslash G$ we set 
\begin{align*}
\Gk_{N^-g}=&
\{\eta\in\Gg^*\mid\langle\Ad(g^{-1})(\Gn^-),\eta
\rangle=\{0\}\}\\
=&
\{(y_1,y_2)\in\Gk\mid
y_1-y_2\in\Ad(g^{-1})(\Gb^-)
\}.
\end{align*}
For $\eta\in\Gk_{N^-g}$ we define $\overline{L}^*_\eta\in T^*(N^-\backslash G)_{N^-g}$ by
\[
\langle \overline{L}_a, \overline{L}^*_\eta\rangle=\langle a,\eta\rangle\qquad
(a\in\Gg).
\]

We easily obtain the following from Proposition \ref{prop:Poisson1}.
\begin{proposition}
\label{prop:Poisson2}
$ZE_\zeta^{(\ell)}$ is closed under the Poisson bracket \eqref{eq:Poisson:E}.
Moreover, under the identification
$ZE_\zeta^{(\ell)}\cong\BC[N^-\backslash G]\otimes \BC[K]\otimes\BC[H]$
the corresponding Poisson tensor of the manifold
$(N^-\backslash G)\times K\times H$ 
is given by
\begin{align*}
&\delta_{(N^-g,k,t)}(
\overline{L}^*_\eta,\overline{L}^*_{\eta'})
=\tilde{\epsilon}(p_\Gs(\widetilde{\Ad}(g)(\eta)),\widetilde{\Ad}(g)(\eta'))
&(\eta, \eta'\in\Gk_{N^-g}),
\\
&\delta_{(N^-g,k,t)}(
R^*_\xi,R^*_{\xi'})
=-\tilde{\epsilon}(p_\Gk(\widetilde{\Ad}(k^{-1})(\xi)),\widetilde{\Ad}(k^{-1})(\xi'))
&(\xi, \xi'\in\Gg),
\\
&\delta_{(N^-g,k,t)}(
\overline{L}^*_\eta,R^*_\xi)
=\tilde{\epsilon}(\xi,\eta)
&(\eta\in\Gk_{N^-g}, \xi\in\Gg),\\
&\delta_{(N^-g,k,t)}(\overline{L}^*_\eta,
L^*_\lambda)=
-\tilde{\epsilon}(\Ad(g^{-1})(c_\lambda),\eta)/2\ell
&(\eta\in\Gk_{N^-g}, \lambda\in\Gh^*),
\\
&\delta_{(N^-g,k,t)}
((T^*H)_t,(T^*K)_{k}\oplus(T^*H)_t)=\{0\}.
\end{align*}
Here, $c_\lambda$ for $\lambda\in\Gh^*$ denotes the element of $\Gh$ such that $\mu(c_\lambda)=(\lambda,\mu)$ for any $\mu\in\Gh^*$.
\end{proposition}
\begin{proposition}
\label{prop:symplectic}
$\tilde{\CV}$ is a Poisson submanifold of $(N^-\backslash G)\times K\times H$ with non-degenerate Poisson tensor.
In particular $\tilde{\CV}$ is a symplectic manifold.
\end{proposition}
\begin{proof}
Denote by 
\[
\rad(\delta_{(N^-g,k,t)})
\subset
(T^*(N^-\backslash G))_{N^-g}\oplus
(T^*K)_k\oplus (T^*H)_t
\]
the radical of the Poisson tensor $\delta$ at $(N^-g,k,t)
\in(N^-\backslash G)\times K\times H$.
Then it is sufficient to show $((T\tilde{\CV})_{(N^-g,k,t)})^\perp
=\rad(\delta_{(N^-g,k,t)})$
for any $(N^-g,k,t)\in\tilde{\CV}$.
Here, $(T\tilde{\CV})_{(N^-g,k,t)}$ is identified with a subspace of $(T(N^-\backslash G))_{N^-g}\oplus
(TK)_k\oplus (TH)_t$, and $((T\tilde{\CV})_{(N^-g,k,t)})^\perp$ denotes the subspace of $(T^*(N^-\backslash G))_{N^-g}\oplus
(T^*K)_k\oplus (T^*H)_t$ which is orthogonal to $(T\tilde{\CV})_{(N^-g,k,t)}$ with respect to the canonical pairing between the tangent and the cotangent spaces.

Let us first compute $\rad(\delta_{(N^-g,k,t)})$ using Proposition \ref{prop:Poisson2}.
Assume $y=\overline{L}^*_\eta+R^*_\xi+L^*_\lambda\in
\rad(\delta_{(N^-g,k,t)})$ for
$\eta=(\eta_1,\eta_2)\in\Gk_{N^-g}$, $\xi\in \Gg$, $\lambda\in\Gh$.
Note that the condition for $(\eta_1,\eta_2)\in\Gk$ to be contained in $\Gk_{N^-g}$ is equivalent to 
\begin{equation}
\label{eq:prop:symplectic:1}
\epsilon(\Ad(g)(\eta_1-\eta_2),\Gn^-)=\{0\}.
\end{equation}
By $\delta_{(N^-g,k,t)}(y,R^*_{\xi'})=0$ for any $\xi'\in\Gg$
we have
\begin{equation}
\label{eq:prop:symplectic:2}
\widetilde{\Ad}(k^{-1})(\eta)-\widetilde{\Ad}(k^{-1})(\xi)\in\Gs.
\end{equation}
By $\delta_{(N^-g,k,t)}(y,L^*_\mu)=0$ for any $\mu\in\Gh^*$
we have
\begin{equation}
\label{eq:prop:symplectic:3}
\epsilon(\Ad(g)(\eta_1-\eta_2),\Gh)=\{0\}.
\end{equation}
By $\delta_{(N^-g,k,t)}(y,\overline{L}^*_{\eta'})=0$ for any $\eta'\in\Gk_{N^-g}$
we have
\begin{equation}
\label{eq:prop:symplectic:4}
p_\Gs(\widetilde{\Ad}(g)(\eta))-\Ad(g)(\xi)+\frac{1}{2\ell}c_\lambda
\in\Gn^-.
\end{equation}
By \eqref{eq:prop:symplectic:1} and \eqref{eq:prop:symplectic:3} we have
\begin{equation}
\label{eq:prop:symplectic:5}
\Ad(g)(\eta_1-\eta_2)\in\Gn^-.
\end{equation}
By
\[
\widetilde{\Ad}(g)(\eta)
=(\Ad(g)(\eta_1),\Ad(g)(\eta_1))+(0,-\Ad(g)(\eta_1-\eta_2))
\]
and \eqref{eq:prop:symplectic:5} we have
$p_\Gs(\tilde{\Ad}(g)(\eta))=\Ad(g)(\eta_1)$.
Hence \eqref{eq:prop:symplectic:4} is equivalent to 
\begin{equation}
\label{eq:prop:symplectic:6}
\xi=\eta_1+\Ad(g^{-1})(c_\lambda/2\ell+z)\qquad
(z\in\Gn^-).
\end{equation}
Substituting \eqref{eq:prop:symplectic:6} to \eqref{eq:prop:symplectic:2} we obtain
\begin{equation}
\label{eq:prop:symplectic:7}
\Ad(g\kappa(k)^{-1}g^{-1})(c_\lambda/2\ell+z)=
c_\lambda/2\ell+z+\Ad(g)(\eta_1-\eta_2).
\end{equation}
In the case $(N^-g,k,t)\in\tilde{\CV}$ we have $g\kappa(k)^{-1}g^{-1}\in t^{-2\ell}N^-$ and hence 
\[
\Ad(g\kappa(k)^{-1}g^{-1})(c_\lambda/2\ell+z)
-(c_\lambda/2\ell+z)\in\Gn^-.
\]
Therefore, for each $\lambda\in\Gh^*$ and $z\in\Gn^-$ there exists unique $\eta=(\eta_1,\eta_2)\in\Gk$ satisfying \eqref{eq:prop:symplectic:5}, \eqref{eq:prop:symplectic:7}.
We conclude that 
$\rad(\delta_{(N^-g,k,t)})$ 
for $(N^-g,k,t)\in\tilde{\CV}$ consists of 
\begin{equation}
\label{eq:prop:symplectic:7a}
y(\lambda,z)=\overline{L}^*_\eta+R^*_\xi+L^*_\lambda
\qquad(\lambda\in\Gh^*, z\in\Gn^-),
\end{equation}
where $\eta=(\eta_1,\eta_2)\in\Gk_{N^-g}$ and $\xi\in\Gg$ are uniquely determined by
\eqref{eq:prop:symplectic:7} and \eqref{eq:prop:symplectic:6}.
In particular we have $\dim \rad(\delta_{(N^-g,k,t)})=\dim\tilde{\CB}$.
Since the codimension of $\tilde{\CV}$ in $\tilde{\CB}\times K\times H$ is also $\dim\tilde{\CB}$, we have only to show 
\[
\langle
\rad(\delta_{(N^-g,k,t)}),
(T\tilde{\CV})_{(N^-g,k,t)}\rangle=\{0\}.
\]
By the description of $\tilde{\CV}$ 
as a covering of an open subset of $\CW$
(see proof of Lemma \ref{lem:tilde-V} for the notation)
we see easily that $(T\tilde{\CV})_{(N^-g,k,t)}$ is spanned by the tangent vectors
$\overline{L}_a+R_b\,\,
(a\in\Gg, b=(b_1,b_2)\in\Gk)$
with
\begin{equation}
\label{eq:prop:symplectic:8}
\Ad(\kappa(k))(a)-a-\Ad(\kappa(k))(b_2)+b_1=0,
\end{equation}
and 
$R_{b'}+L_c\,\,(b'=(b'_1,b'_2)\in\Gk, c\in\Gh)$ with
\begin{equation}
\label{eq:prop:symplectic:9}
\Ad(g)(b'_1-\Ad(\kappa(k))(b'_2))-2\ell c\in\Gn^-.
\end{equation}
Take $y(\lambda,z)$ as in \eqref{eq:prop:symplectic:7a},
and set $u=c_\lambda/2\ell+z$.
For $\overline{L}_a+R_b$ satisfying \eqref{eq:prop:symplectic:8} we have
\begin{align*}
&\langle y(\lambda,z),\overline{L}_a+R_b\rangle\\
=&\epsilon(a,\eta_1-\eta_2)+\epsilon(b_1-b_2,\xi)\\
=&\epsilon(a,\Ad(\kappa(k)^{-1}g^{-1})(u)-\Ad(g^{-1})(u))
+\epsilon(b_1-b_2,\eta_1+\Ad(g^{-1})(u))\\
=&\epsilon(\Ad(\kappa(k))(a)-a,\Ad(g^{-1})(u))
+\epsilon(b_1-b_2,\eta_1+\Ad(g^{-1})(u))\\
=&\epsilon(-b_1+\Ad(\kappa(k))(b_2),\Ad(g^{-1})(u))
+\epsilon(b_1-b_2,\eta_1+\Ad(g^{-1})(u))\\
=&\epsilon(\Ad(\kappa(k))(b_2)-b_2,\Ad(g^{-1})(u))
+\epsilon(b_1-b_2,\eta_1)\\
=&\epsilon(b_2,\Ad(\kappa(k)^{-1}g^{-1})(u)-\Ad(g^{-1})(u))
+\epsilon(b_1-b_2,\eta_1)\\
=&\epsilon(b_2,\eta_1-\eta_2)
+\epsilon(b_1-b_2,\eta_1)\\
=&-\epsilon(b_2,\eta_2)
+\epsilon(b_1,\eta_1)
\end{align*}
by \eqref{eq:prop:symplectic:6}, \eqref{eq:prop:symplectic:7}.
We have $b_1\in h+\Gn^+, b_2\in-h+\Gn^-, \eta_1\in h'+\Gn^+, b_2\in -h'+\Gn^-$ for some $h, h'\in\Gh$ and hence
\[
\langle y(\lambda,z),\overline{L}_a+R_b\rangle
=
-\epsilon(b_2,\eta_2)
+\epsilon(b_1,\eta_1)
=-\epsilon(h,h')+\epsilon(h,h')=0.
\]
For $R_{b'}+L_c$ satisfying \eqref{eq:prop:symplectic:9} we have
\begin{align*}
&
\langle y(\lambda,z),R_{b'}+L_c\rangle\\
=&\epsilon(\xi,b_1'-b_2')+\lambda(c)\\
=&\epsilon(\eta_1+\Ad(g^{-1})(u),b_1'-b_2')+\lambda(c)\\
=&\epsilon(u,\Ad(g)(b_1'-b_2'))+
\epsilon(\eta_1,b_1'-b_2')+\lambda(c)
\end{align*}
by \eqref{eq:prop:symplectic:6}.
By $u\in\Gb^-$ we have $\epsilon(u,\Gn^-)=\{0\}$.
Hence by \eqref{eq:prop:symplectic:9}, \eqref{eq:prop:symplectic:7} we have
\begin{align*}
&
\langle y(\lambda,z),R_{b'}+L_c\rangle\\
=&\epsilon(u,\Ad(g\kappa(k))(b_2')-\Ad(g)(b_2'))
-\epsilon(c_\lambda/2\ell,2\ell c)+
\epsilon(\eta_1,b_1'-b_2')+\lambda(c)\\
=&\epsilon(\Ad(\kappa(k^{-1})g^{-1})(u)-\Ad(g^{-1})(u),b_2')+
\epsilon(\eta_1,b_1'-b_2')\\
=&\epsilon(\eta_1-\eta_2,b_2')+
\epsilon(\eta_1,b_1'-b_2')\\
=&-\epsilon(\eta_2,b_2')+
\epsilon(\eta_1,b_1')\\
=&0.
\end{align*}
The proof is complete.
\end{proof}
 \begin{remark}
 {\rm 
 We can similarly show that $\CV$ is a connected symplectic manifold with
 $\dim\CV=2\dim\CB$ (see \cite{T3}).
 }
 \end{remark}

\subsection{}
Let us finish the proof of Theorem \ref{thm:center}. 
Set $J_1=\Ker(ZE_\zeta^{(\ell)}\to {ZD}_\zeta^{(\ell)})$.
Since $ZE_\zeta^{(\ell)}\to {ZD}_\zeta^{(\ell)}$ is a morphism of Poisson algebras, 
$J_1$ is a Poisson ideal of $ZE_\zeta^{(\ell)}$.
Hence $\sqrt{J_1}$ is also a Poisson ideal.
It follows that the support of $\tilde{\CZ}_\zeta$ is a Poisson subvariety of $\tilde{\CB}\times K\times H$ contained in $\tilde\CV$. 
By Lemma \ref{lem:tilde-V} and Proposition \ref{prop:symplectic}
$\tilde\CV$ contains no non-empty Poisson subvariety except for $\tilde\CV$ itself.
Therefore, we have only to show  $\tilde{\CZ}_\zeta\ne0$.
Since $A_1$ contains no non-trivial zero divisors, 
the composit of 
$A_1\to D_\zeta\to\End_\BC(A_\zeta)$ is injective, where $A_1\to D_\zeta$ is given by $\varphi\mapsto\ell_\varphi$.
Hence $A_1\to{ZD}_\zeta^{(\ell)}$ is also injective.
It follows that $\tilde{\CZ}_\zeta\supset\CO_{\tilde{\CB}}\ne0$.
The proof of Theorem \ref{thm:center} is now complete.

\section{Azumaya properties}
\label{sec:Azumaya}
\subsection{}
By Lemma \ref{lem:center} and Theorem \ref{thm:center} 
$Fr_*\DD'_{\CB_\zeta}$ and $Fr_*\DD_{\CB_\zeta}$
are sheaves of $\CO_{\CB}$-algebras containing 
$p_{\CV*}\CO_\CV$ as a central subalgebra.
Hence we can consider their localizations 
\[
\tilde{\DD}'_{\CB_\zeta}
=p_{\CV}^{-1}Fr_*\DD'_{\CB_\zeta}
\otimes_{p_{\CV}^{-1}p_{\CV*}\CO_\CV}
\CO_\CV,\qquad
\tilde{\DD}_{\CB_\zeta}
=p_{\CV}^{-1}Fr_*\DD_{\CB_\zeta}
\otimes_{p_{\CV}^{-1}p_{\CV*}\CO_\CV}
\CO_\CV.
\]
on $\CV$.
They are $\CO_\CV$-algebras, and we have a natural $\CO_{\CV}$-algebra homomorphism
$
\tilde{\DD}'_{\CB_\zeta}
\to
\tilde{\DD}_{\CB_\zeta}.
$
The first purpose of this section is to prove the following.
\begin{theorem}
\label{thm:Azumaya}
We have 
\[
\tilde{\DD}'_{\CB_\zeta}
\cong
\tilde{\DD}_{\CB_\zeta}.
\]
Moreover, $\tilde{\DD}_{\CB_\zeta}$ is an Azumaya algebra of 
rank $\ell^{2|\Delta^+|}$ on $\CV$.
Namely, $\tilde{\DD}_{\CB_\zeta}$ is locally free as an $\CO_\CV$-module, and for any $v\in\CV$ the fiber $\tilde{\DD}_{\CB_\zeta}(v)$ is isomorphic to the matrix algebra $M_{\ell^{|\Delta^+|}}(\BC)$ as a $\BC$-algebra.
\end{theorem}
We need some preliminaries.

\begin{lemma}
\label{lem:T-on-ZE}
For $w\in W$ regard $\Theta_w\subset A_1$ as a subset of $E_\zeta$.
Then with respect to the action of $\BB$ on $E_\zeta$ we have $T_{w^{-1}}^{-1}\star \Theta_{e}=\Theta_{w}$.
Moreover, we have $T\star ZE_\zeta^{(\ell)}=ZE_\zeta^{(\ell)}$ for any $T\in\BB$.
Hence we have also $T\star Z{D'_\zeta}^{(\ell)}=Z{D'_\zeta}^{(\ell)}$ and $T\star ZD_\zeta^{(\ell)}=ZD_\zeta^{(\ell)}$ for any $T\in\BB$.
\end{lemma}
\begin{proof}
The first half is already shown in Lemma \ref{lem:braid-action-on-D}.
In view of Lemma \ref{lem:braid-D}
the only non-trivial part is to show $T_i^{\pm1}\star A_1\subset ZE_\zeta^{(\ell)}$ for $i\in I$.
Let $\varphi\in A_1$.
By Lemma \ref{lem:braid-Fr} we have $T_i^{\pm1}(\varphi)\in A_1$.
Then we see easily that $T_i^{\pm1}\star\varphi\in A_1\otimes Z_{Fr}(U_\zeta)$ by Lemma \ref{lem:braid-D} (ii).
\end{proof}
\begin{lemma}
$\tilde{\DD}'_{\CB_\zeta}$ is locally generated by $\ell^{2|\Delta^+|}$ sections.
\end{lemma}
\begin{proof}
It is sufficient to show that for any $w\in W$ the $(\Theta_{w}^{-1}Z{D'_\zeta}^{(\ell)})(0)$-module $(\Theta_{w}^{-1}{D'_\zeta}^{(\ell)})(0)$ is generated by $\ell^{2|\Delta^+|}$ elements.
By Lemma \ref{lem:T-on-ZE} we have $T_{w^{-1}}^{-1}\star \Theta_{e}=\Theta_{w}$ and $T_{w^{-1}}^{-1}\star Z{D'_\zeta}^{(\ell)}=Z{D'_\zeta}^{(\ell)}$.
Hence we may assume $w=e$ from the beginning.

Note that $(\Theta_{e}^{-1}{D'_\zeta}^{(\ell)})(0)$ is generated by the elements
\[
\jmath(u), \,\,\jmath(\Phi), \,\,\jmath(e(\lambda))\qquad
(u\in U_\zeta, \Phi\in(\Theta_{e}^{-1}A_\zeta)(0), \lambda\in\Lambda),
\]
while  $(\Theta_{e}^{-1}Z{D'_\zeta}^{(\ell)})(0)$ is generated by the elements
\[
\jmath(u), \,\,\jmath(\Phi), \,\,\jmath(e(\lambda))\qquad
(u\in Z_{Fr}(U_\zeta), \Phi\in(\Theta_{e}^{-1}A_1)(0),  \lambda\in\Lambda).
\]

We first show 
\[
\jmath(y_pk_{\beta_p})\in
\jmath((\Theta_e^{-1}A_\zeta)(0)
U_\zeta^{\geqq0}\BC[\Lambda])
\]
for any $p$ by induction on $\Ht(\beta_p)$.
By Proposition \ref{prop:local-verma} we can take $\lambda$, $\gamma$, $\varphi$, $s$ as in Proposition \ref{prop:local-formula1} satisfying $\gamma=\beta_p$, $(Sx^L_p)(\varphi s^{-1})=1$ and $(Sx^L_{p'})(\varphi s^{-1})=0$ for $p'\ne p$ with $\beta_{p'}=\beta_p$.
Then the assertion follows from Proposition \ref{prop:local-formula1}.

We next show
\[
\jmath(k_\mu)\in
\jmath((\Theta_e^{-1}A_\zeta)(0)
U_\zeta^{+})
(\Theta_{e}^{-1}Z{D'_\zeta}^{(\ell)})(0)
\]
for any $\mu\in\Lambda$.
We see easily that there exists some $\lambda\in\Lambda^+$ such that $\mu-2\lambda\in\ell\Lambda$.
Write
$\jmath(k_\mu)
=\jmath(k_{2\lambda})\jmath(k_{\mu-2\lambda})$.
Then we have $\jmath(k_{\mu-2\lambda})\in
(\Theta_{e}^{-1}Z{D'_\zeta}^{(\ell)})(0)$ by $k_{\mu-2\lambda}\in Z_{Fr}(U_\zeta)$.
Hence the assertion follows from Proposition \ref{prop:local-formula2}.

It follows that 
\[
(\Theta_{e}^{-1}{D'_\zeta}^{(\ell)})(0)
=
\jmath((\Theta_e^{-1}A_\zeta)(0))
\jmath(U_\zeta^{+})
(\Theta_{e}^{-1}Z{D'_\zeta}^{(\ell)})(0).
\]
By definition  $U_\zeta^+$ is a free $U_\zeta^+\cap Z_{Fr}(U_\zeta)$-module of rank $\ell^{|\Delta^+|}$.
Moreover, $(\Theta_e^{-1}A_\zeta)(0)$ is a free $(\Theta_e^{-1}A_1)(0)$-module of  rank $\ell^{|\Delta^+|}$ by Proposition \ref{prop:AzetaA1}.
We are done.
\end{proof}

By \eqref{eq:Har-center-in-D}
we have an $\CO_\CV$-algebra homomorphism
\[
U_\zeta\otimes_{Z(U_\zeta)}\CO_\CV
\to
\tilde{\DD}_{\CB_\zeta},
\]
where $Z(U_\zeta)\to\CO_\CV$ is given by 
\[
\CV\to K\times_{H/W}(H/{W\circ})
\,\,(\cong \Spec\,Z(U_\zeta))\qquad
(B^-g,k,t)\mapsto(k,[t^2]))
\]
(see Corollary \ref{cor:center}).

We set 
\begin{align*}
H_{ur}&=\{
t\in H\mid
\alpha\in\Delta,\,
{\theta_\alpha(t)}^{\ell}=1\,\Longrightarrow\,
{\theta_\alpha(t)}=\zeta^{-(2\rho,\alpha)}\},\\
\CV_{ur}&=
\{(B^-g, k,t)\in\CV\mid
t^2\in H_{ur}\}.
\end{align*}
Note that $W\circ H_{ur}=H_{ur}$.
By Brown-Gordon 
\cite[Theorem 2.5, Theorem 4.9]{BG} (see also Brown-Goodearl \cite[Theorem 4.3]{BG0})
we have the following.
\begin{proposition}
\label{prop:BG-Azumaya}
$U_\zeta\otimes_{Z(U_\zeta)}\CO_{K\times_{H/W}(H_{ur}/{W\circ})}$ is an Azumaya algebra of rank $\ell^{2|\Delta^+|}$ on $K\times_{H/W}(H_{ur}/{W\circ})$.
\end{proposition}
In particular, 
$U_\zeta\otimes_{Z(U_\zeta)}\CO_{\CV_{ur}}$ is an Azumaya algebra of rank $\ell^{2|\Delta^+|}$ on $\CV_{ur}$.
On the other hand, since $\tilde{\DD}_{\CB_\zeta}$ is a quotient of $\tilde{\DD}'_{\CB_\zeta}$, $\tilde{\DD}_{\CB_\zeta}$ is also locally generated by $\ell^{2|\Delta^+|}$ sections.
Hence we obtain
\[
U_\zeta\otimes_{Z(U_\zeta)}\CO_{\CV_{ur}}
\cong
\tilde{\DD}_{\CB_\zeta}|_{\CV_{ur}}
\]
by Lemma \ref{lem:Azumaya-generality} below.
In particular, $\tilde{\DD}_{\CB_\zeta}|_{\CV_{ur}}$ is an Azumaya algebra of rank $\ell^{2|\Delta^+|}$.

\begin{lemma}
\label{lem:Azumaya-generality}
Let $X$ be an algebraic variety over $\BC$, and let $f:\CA\to\CA'$ be a homomorphism of $\CO_X$-algebras.
Assume that
$\CA$ is an Azumaya algebra of rank $n^2$ on $X$
and that
$\CA'$ is coherent and locally generated by $n^2$ sections as an $\CO_X$-module.
Assume also that the fiber $\CA'(x)$ is not zero for any $x\in X$.
Then $f$ is an isomorphism.
\end{lemma}
\begin{proof}
For each $x\in X$ 
consider the $\BC$-algebra homomorphism 
 $f_x:\CA(x)\to\CA'(x)$ for the fibers.
Then $f_x$ is a non-zero homomorphism since it sends $1_{\CA(x)}$ to $1_{\CA'(x)}$ which is non-zero by ${\CA'(x)}\ne\{0\}$.
Hence by the simplicity of $\CA(x)$ we conclude that $f_x$ is injective.
On the other hand we have $\dim\CA(x)=n^2$ and $\dim\CA'(x)\leqq n^2$ by our assumption.
It follows that $f_x$ is an isomorphism for any $x\in X$.

Define $\CA''$ by the exact sequence
\[
\CA \rightarrow \CA'\rightarrow \CA''\rightarrow 0.
\]
Then we have $\CA''(x)=\{0\}$ for any $x\in X$ by the surjectivity of $f_x$. 
Hence $\CA''=\{0\}$ by Nakayama's lemma.
It follows that $f$ is an epimorphism.

Let us show that $f$ is a monomorphism.
We may assume that $X$ is affine and $\CA$ is free.
Set $R=\Gamma(X,\CO_X)$, $A=\Gamma(X,\CA)$, $A'=\Gamma(X,\CA')$.
We need to show that the homomorphism $F:A\to A'$ of $R$-modules corresponding to $f$ is injective.
Note that $A$ is isomorphic to $R^{n^2}$.
By the injectivity of $f_x$ for $x\in X$ the homomorphism
$F_\Gm:R/\Gm\otimes_RA\to R/\Gm\otimes_RA'$ is injective for any maximal ideal $\Gm$ of $R$.
Hence by the commutative diagram
\[
\begin{CD}
A&@>{F}>>&A'\\
@VVV&&@VVV\\
R/\Gm\otimes_RA
&
@>>{F_\Gm}>
&
R/\Gm\otimes_RA'
\end{CD}
\]
and $\Ker(A\to R/\Gm\otimes_RA)=\Gm A$ 
we have 
\[
\Ker(F)\subset\bigcap_\Gm\Gm A
\cong\bigcap_\Gm\Gm R^{n^2}
=(\bigcap_\Gm\Gm R
)^{n^2}=\{0\}.
\]
\end{proof}

For $\mu\in\Lambda$ we define $t_\mu\in H$ by
\[
\theta_\lambda(t_\mu)=\zeta^{(\lambda,\mu)}\qquad(\lambda\in\Lambda).
\]
Note that we have $t_\mu^{\ell}=1$.
We consider the automorphism
\[
\xi_\mu:\CV\to\CV
\qquad
((B^-g, k,t)\mapsto (B^-g, k, t_\mu t))
\]
of the algebraic variety $\CV$.
\begin{lemma}
\label{lem:ur-translate}
We have
\[
\bigcup_{\mu\in\Lambda}\xi_\mu(\CV_{ur})=\CV.
\]
\end{lemma}
\begin{proof}
We need to show that
$\bigcup_{\mu\in\Lambda}t_{2\mu}H_{ur}=H$.
Let $t\in H$, and set
\[
\Delta'=\{\alpha\in\Delta\mid\theta_\alpha(t)^\ell=1\},\qquad
Q'=\sum_{\alpha\in\Delta'}\BZ\alpha.
\]
Then it is sufficient to show that 
there exists some $\mu\in\Lambda$ such that $\theta_\alpha(t)=\theta_\alpha(t_{2\mu})$ for any $\alpha\in\Delta'$.
Note that $t$ determines uniquely an element $\varphi\in\Hom_\BZ(Q',\BZ/\ell\BZ)$ such that $\theta_\alpha(t)=\zeta^{\varphi(\alpha)}\;(\alpha\in\Delta')$.
Hence our assertion is equivalent to the surjectivity of 
the map
\[
2\Lambda\to\Hom_\BZ(Q',\BZ/\ell\BZ)
\qquad(\mu\mapsto[\gamma\mapsto\overline{(\mu,\gamma)}]).
\]
By Dynkin \cite[Theorem 5.2]{Dy} there exist a root subsystem $\tilde{\Delta}$ of $\Delta$ satisfying
\[
\sum_{\beta\in\tilde{\Delta}}\BQ\beta=\sum_{\alpha\in\Delta}\BQ\alpha
\]
and 
a set of simple roots $\tilde{\Pi}$ of $\tilde{\Delta}$ such that 
$\Delta'\cap\tilde{\Pi}$ is a set of simple roots for $\Delta'$.
Set $\tilde{Q}=\sum_{\alpha\in\tilde{\Delta}}\BZ\alpha$.
Since the natural map
$\Hom_\BZ(\tilde{Q},\BZ/\ell\BZ)\to \Hom_\BZ(Q',\BZ/\ell\BZ)$ is surjective, it is sufficient to show that the map
\[
2\Lambda\to\Hom_\BZ(\tilde{Q},\BZ/\ell\BZ)
\qquad(\mu\mapsto[\gamma\mapsto\overline{(\mu,\gamma)}])
\]
is surjective.
Set $\tilde{M}=\Hom_\BZ(\tilde{Q},\BZ)$.
Then we have
\[
\Hom_\BZ(\tilde{Q},\BZ/\ell\BZ)\cong \tilde{M}\otimes_\BZ\BZ/\ell\BZ,
\]
and $2\Lambda\to \tilde{M}\;(\mu\mapsto[\gamma\mapsto{(\mu,\gamma)}])$ is an injective homomorphism of free $\BZ$-modules with the same rank.
Therefore, it is sufficient to show that $|\tilde{M}/2\Lambda|$ is relatively prime to $\ell$.
Set $M=\Hom_\BZ(Q,\BZ)$.
Then we have inclusions 
\[
2\Lambda\subset\Lambda\subset M\subset \tilde{M}
\]
of  free $\BZ$-modules with the same rank.
By
\[
|\Lambda/2\Lambda|=2^{|I|}, \qquad
|M/\Lambda|
=\prod_{i\in I}\frac{(\alpha_i,\alpha_i)}2
\]
and our assumption on $\ell$ it is sufficient to show that $|\tilde{M}/M|$ is relatively prime to $\ell$.
Note $|\tilde{M}/M|=|Q/\tilde{Q}|$.
By Dynkin \cite[Theorem 5.3]{Dy}
we have a sequence
\[
\tilde{\Delta}=\Delta_{r+1}\subset\cdots\subset \Delta_2\subset \Delta_1=\Delta
\]
of root subsystems of $\Delta$ and sets $\Pi_i$ of simple roots for $\Delta_i$ such that $\Pi_{i+1}=(\Pi_i\setminus\{\gamma_i\})\cup\{-\theta_i\}$
for each $i=1,\dots, r$, where $\gamma_i\in \Pi_i$ and $\theta_i$ is the highest root of the irreducible component of $\Delta_i$ containing $\gamma_i$.
Set $Q_i=\sum_{\alpha\in\Delta_i}\BZ\alpha$, and
write $\theta_i=\sum_{\gamma\in\Pi_i}m_\gamma\gamma$.
Then we see easily that $|Q_i/Q_{i+1}|=m_{\gamma_i}$.
By our assumption on $\ell$ we can easily check that $m_{\gamma_i}$ is relatively prime to $\ell$.
\end{proof}
\begin{lemma}
For $\mu\in\Lambda$ 
the $\CO_\CV$-algebras
$\xi_\mu^*\tilde{\DD}_{\CB_\zeta}$ and $\tilde{\DD}_{\CB_\zeta}$ 
are isomorphic locally on $\CB$.
If 
\end{lemma}
\begin{proof}
We will show 
\[
\xi_\mu^*\tilde{\DD}_{\CB_\zeta}|_{p_\CV^{-1}\CB_w}
\cong
\tilde{\DD}_{\CB_\zeta}|_{p_\CV^{-1}\CB_w}
\] 
for any $w\in W$.
It is sufficient to verify that there exists a $\BC$-algebra automorphism of  
$\Theta_{w}^{-1}D_\zeta^{(\ell)}(0)$
which induces
\[
\xi_\mu^*:\Theta_{w}^{-1}ZD_\zeta^{(\ell)}(0)
\to
\Theta_{w}^{-1}ZD_\zeta^{(\ell)}(0)
\]
under the identification $p_\CV^{-1}\CB_w
=\Spec\,\Theta_{w}^{-1}ZD_\zeta^{(\ell)}(0)$.
Note that 
\[
\Theta_{w}^{-1}D_\zeta^{(\ell)}(0)
\cong
\tilde\Theta_{w}^{-1}{D_\zeta}(0).
\]
Hence the problem is to construct an automorphism of the $\BC$-algebra $\tilde\Theta_{w}^{-1}{D_\zeta}(0)$ satisfying 
\[
\varphi\mapsto\varphi,\quad
\deru_z\mapsto\deru_z,\quad
\sigma_\lambda\mapsto\zeta^{(\lambda,\mu)}\sigma_\lambda
\quad(\varphi\in \Theta_{w}^{-1}A_1(0), z\in Z_{Fr}(U_\zeta), \lambda\in\Lambda).
\]
Take $c\in A_\zeta(\mu)_{w^{-1}\mu}\setminus\{0\}
\subset\tilde{\Theta}_w$
Then the automorphism
\[
\tilde\Theta_{w}^{-1}{D_\zeta}(0).
\to
\tilde\Theta_{w}^{-1}{D_\zeta}(0)\qquad
(P\mapsto c^{-1} P c)
\]
satisfies the desired property.
\end{proof}

By Lemma \ref{lem:ur-translate} we conclude that $\tilde{\DD}_{\CB_\zeta}$ is an Azumaya algebra of rank $\ell^{2|\Delta^+|}$ on $\CV$.
Recall that there exists a surjection 
$\tilde{\DD}'_{\CB_\zeta}\to\tilde{\DD}_{\CB_\zeta}$.
By the arguments above there exists locally a generator system of $\tilde{\DD}'_{\CB_\zeta}$ consisting of $\ell^{2|\Delta^+|}$ sections which gives a (local) free basis of $\tilde{\DD}_{\CB_\zeta}$.
Hence $\tilde{\DD}'_{\CB_\zeta}\to\tilde{\DD}_{\CB_\zeta}$ is an isomorphism.
The proof of Theorem \ref{thm:Azumaya} is now complete.

By Theorem \ref{thm:Azumaya}  and Lemma \ref{lem:Fr1} we have the following.
\begin{corollary}
\label{cor:Azumaya}
$\omega^*D'_\zeta\cong\omega^* D_\zeta$.
\end{corollary}
Theorem \ref{thm:Azumaya} also implies the following.
\begin{corollary}
\label{cor:Azumaya2}
$\tilde{\CZ}_\zeta$ coincides with the center of $\tilde{\DD}_{\CB_\zeta}$.
Hence ${\CZ}_\zeta$ coincides with the center of $Fr_*{\DD}_{\CB_\zeta}$.
\end{corollary}

\subsection{}
Define
\[
\delta:\CV\to K\times_{H/W} H
\]
by $\delta(B^-g,k,t)=(k,t)$, where $H\to H/W$ is given by $t\mapsto[t^{2\ell}]$.
In the rest of this section we will prove the following result.
\begin{theorem}
\label{thm:split-Azumaya} 
For any $(k, t)\in K\times_{H/W}H$,
the restriction of
$\tilde{\DD}_{\CB_\zeta}$ to ${\delta^{-1}(k,t)}$ is a split Azumaya algebra.
\end{theorem}
We need some preliminaries.
\begin{lemma}
\label{lem:Azumaya-generality2} 
Let $\CA$ be an Azumaya algebra on an algebraic variety $X$.
Assume that $\CM$ is a locally free right $\CA$-module of rank one.
Then $\CEnd_\CA(\CM)$ is an Azumaya algebra whose rank is the same as that of $\CA$.
Moreover, if $\CA$ is a split Azumaya algebra, then $\CEnd_\CA(\CM)$ is also a split Azumaya algebra.
\end{lemma}
\begin{proof}
Let $V$ be a finite-dimensional vector space over a field $k$ and regard $M=\End_k(V)$ as a right $\End_k(V)$-module by the right multiplication.
Then the left multiplication of $\End_k(V)$ induces 
a canonical isomorphism 
\begin{equation}
\label{eq:End}
\End_{\End_k(V)}(M)\cong\End_k(V)
\end{equation}
of $k$-algebras.
Hence the first half of our theorem holds.
Let us show the second half.
Assume $\CA=\CEnd_{\CO_X}(\CW)$ for a locally free $\CO_X$-module $\CW$.
Then it is sufficient to show $\CEnd_{\CA}(\CM)\cong\CEnd_{\CO_X}(\CM\otimes_\CA\CW)$.
This also follows from \eqref{eq:End}.
\end{proof}

\begin{proposition}
\label{prop:Azumaya-translation}
For any $\mu\in\Lambda$
there exists a locally free right $\tilde{\DD}_{\CB_\zeta}$-module
$\CL_\mu$ 
of rank one such
\[
\xi_\mu^*\tilde{\DD}_{\CB_\zeta}
\cong
\CEnd_{\tilde{\DD}_{\CB_\zeta}}
(\CL_\mu).
\]
\end{proposition}
\begin{proof}
Note that $D_\zeta$ is a $\Lambda$-graded $D_\zeta$-bimodule by the left and the right multiplications.
Define $D_\zeta[\mu]$ to be the $\Lambda$-graded $D_\zeta$-bimodule which coincides with $D_\zeta$ as a $D_\zeta$-bimodule and the grading is given by $(D_\zeta[\mu])(\lambda)=D_\zeta(\lambda+\mu)$.
Then we obtain a $\Lambda$-graded $D_\zeta^{(\ell)}$-bimodule 
\[
D_\zeta[\mu]^{(\ell)}
=\bigoplus_{\lambda\in\Lambda}
(D_\zeta[\mu])(\ell\lambda)
\qquad
((D_\zeta[\mu]^{(\ell)})(\lambda)
=D_\zeta(\mu+\ell\lambda)).
\]
Note that the left and the right actions of $ZD_\zeta^{(\ell)}$ on $D_\zeta[\mu]^{(\ell)}$ are different.
In fact we have 
\[
z\cdot P=P\cdot\tilde{\xi}_\mu(z)
\qquad(z\in ZD_\zeta^{(\ell)}),
\]
where $\tilde{\xi}_\mu:ZD_\zeta^{(\ell)}\to ZD_\zeta^{(\ell)}$ is the algebra automorphism corresponding to $\xi_\mu$.
Regard $D_\zeta[\mu]^{(\ell)}$ as a $ZD_\zeta^{(\ell)}$-module by the right action and consider its localization $\CL_\mu$ on $\CV$.
Then the right action of $D_\zeta^{(\ell)}$ on $D_\zeta[\mu]^{(\ell)}$ induces 
a right $\tilde{\DD}_{\CB_\zeta}$-module structure of $\CL_\mu$.

Let us show that
$\CL_\mu$ is a locally free right $\tilde{\DD}_{\CB_\zeta}$-module
of rank one.
For that it is sufficient to show that $j_w^*\omega_\CB^*(D_\zeta[\mu]^{(\ell)})$ is a free right $j_w^*\omega_\CB^*(D_\zeta^{(\ell)})$-module of rank one for any $w\in W$.
Note 
\begin{align*}
j_w^*\omega_\CB^*(D_\zeta^{(\ell)})
&\cong
(D_\zeta^{(\ell)}\otimes_{A_1}\Theta_w^{-1}A_1)(0)
\cong
(\tilde{\Theta}_w^{-1}D_\zeta)(0),\\
j_w^*\omega_\CB^*(D_\zeta[\mu]^{(\ell)})
&\cong
(D_\zeta[\mu]^{(\ell)}\otimes_{A_1}\Theta_w^{-1}A_1)(0)
\cong
(\tilde{\Theta}_w^{-1}D_\zeta)(\mu).
\end{align*}
Take $s\in A_\zeta(\mu)_{w^{-1}\mu}\setminus\{0\}$.
Then we have $s\in \tilde{\Theta}_w$, and the map
\[
(\tilde{\Theta}_w^{-1}D_\zeta)(0)\to
(\tilde{\Theta}_w^{-1}D_\zeta)(\mu)
\qquad
(P\mapsto sP)
\]
gives an isomorphism of right $(\tilde{\Theta}_w^{-1}D_\zeta)(0)$- modules.
We have shown that 
$\CL_\mu$ is a locally free right $\tilde{\DD}_{\CB_\zeta}$-module
of rank one.

On the other hand 
regard $D_\zeta^{(\ell)}$ as a $\Lambda$-graded $ZD_\zeta^{(\ell)}$-algebra by the modified $ZD_\zeta^{(\ell)}$-module structure given by 
\[
z\circ P=\tilde{\xi}_\mu(z)P\quad(z\in ZD_\zeta^{(\ell)},\,\,P\in D_\zeta^{(\ell)}).
\]
Then the localization of $D_\zeta^{(\ell)}$ on $\CV$ with respect to the modified $\Lambda$-graded $ZD_\zeta^{(\ell)}$-algebra structure coincides with $\xi_\mu^*\tilde{\DD}_{\CB_\zeta}$.
Hence we obtain an $\CO_\CV$-algebra homomorphism
\[
\xi_\mu^*\tilde{\DD}_{\CB_\zeta}
\to
\CEnd_{\tilde{\DD}_{\CB_\zeta}}
(\CL_\mu)
\]
induced by the left multiplication of $D_\zeta^{(\ell)}$ on $(D_\zeta[\mu])^{(\ell)}$.
By Theorem \ref{thm:Azumaya}, Lemma \ref{lem:Azumaya-generality2} and Lemma \ref{lem:Azumaya-generality} this is an isomorphism.
\end{proof}

Let us finish the proof of Theorem \ref{thm:split-Azumaya}.
By Propositon \ref{prop:BG-Azumaya} the assertion holds when $t\in H_{ur}$. 
The general case is reduced to this case by Proposition \ref{prop:Azumaya-translation} and Lemma \ref{lem:Azumaya-generality2}. 
The proof of Theorem \ref{thm:split-Azumaya} is now complete.

\begin{remark}{\rm
Similarly to \cite[Thoerem 5.1.1 (a)]{BMR} we can show that
for any $(k, t)\in K\times_{H/W}H$,
the Azumaya algebra $\tilde{\DD}_{\CB_\zeta}$ splits on the formal neighborhood in $\CV$ of ${\delta^{-1}(k,t)}$.
}
\end{remark}

\end{document}